\documentclass[10pt]{article}
\usepackage{latexsym}
\usepackage{indentfirst}
\usepackage{amssymb,amsfonts,amsmath,amsthm}
\usepackage{graphicx}
\usepackage{mathrsfs}
\usepackage{array}
\usepackage{color}
\usepackage{epstopdf}
\usepackage[all]{xy}
\usepackage{hyperref}
\usepackage{url}

\usepackage{geometry}
\newtheorem{thm}{Theorem}[section]
\newtheorem{prop}[thm]{Proposition}
\newtheorem{lem}[thm]{Lemma}
\newtheorem{conj}[thm]{Conjecture}
\newtheorem{cor}[thm]{Corollary}
\newtheorem{exam}[thm]{Example}

\newtheorem{rmk}[thm]{Remark}
\newtheorem{dfn}[thm]{Definition}

\numberwithin{equation}{section}

\newcommand{\frakA}{{\mathfrak A}}
\newcommand{\frakB}{{\mathfrak B}}

\newcommand{\frakS}{{\mathfrak S}}

\newcommand{\frakX}{{\mathfrak X}}

\newcommand{\bB}{{\mathbb B}}
\newcommand{\bC}{{\mathbb C}}

\newcommand{\bM}{{\mathbb M}}
\newcommand{\bN}{{\mathbb N}}

\newcommand{\bQ}{{\mathbb Q}}

\newcommand{\bZ}{{\mathbb Z}}

\newcommand{\calH}{{\mathcal H}}
\newcommand{\calI}{{\mathcal I}}

\newcommand{\calL}{{\mathcal L}}
\newcommand{\calM}{{\mathcal M}}

\newcommand{\calO}{{\mathcal O}}

\newcommand{\rA}{{\mathrm A}}

\newcommand{\rC}{{\mathrm C}}

\newcommand{\rH}{{\mathrm H}}

\newcommand{\rL}{{\mathrm L}}
\newcommand{\rM}{{\mathrm M}}

\newcommand{\rR}{{\mathrm R}}
\newcommand{\rS}{{\mathrm S}}
\newcommand{\rT}{{\mathrm T}}

\newcommand{\rW}{{\mathrm W}}

\newcommand{\Zp}{{\bZ_p}}

\newcommand{\Qp}{{\bQ_p}}

\newcommand{\Cp}{{\bC_p}}

\newcommand{\Ainf}{{\mathrm{A_{inf}}}}

\newcommand{\OX}{{\widehat \calO_X}}
\newcommand{\OC}{{\calO\bC}}                           

\newcommand{\dlog}{{\mathrm{dlog}}}         
\newcommand{\End}{{\mathrm{End}}}           
\newcommand{\Gal}{{\mathrm{Gal}}}           
\newcommand{\HIG}{{\mathrm{HIG}}}           
\newcommand{\id}{{\mathrm{id}}}             
\newcommand{\Ker}{{\mathrm{Ker}}}           
\newcommand{\nil}{\mathrm{nil}}
\newcommand{\pr}{{\mathrm{pr}}}             
\newcommand{\Rep}{{\mathrm{Rep}}}           
\newcommand{\RGamma}{{\mathrm{R\Gamma}}}    
\newcommand{\Sh}{{\mathrm{Sh}}}             
\newcommand{\Spa}{{\mathrm{Spa}}}           
\newcommand{\Spf}{{\mathrm{Spf}}}           
\newcommand{\Vect}{{\mathrm{Vect}}}         

\newcommand{\GL}{{\mathrm{GL}}}             

\newcommand{\cyc}{{\mathrm{cyc}}}           
\newcommand{\et}{{\mathrm{\acute{e}t}}}    
\newcommand{\geo}{{\mathrm{geo}}}           
\newcommand{\gp}{{\mathrm{gp}}}             
\newcommand{\pd}{{\mathrm{pd}}}
\newcommand{\perf}{\mathrm{perf}}           
\newcommand{\proet}{{\mathrm{pro\acute{e}t}}}

\usepackage{relsize}
\usepackage[bbgreekl]{mathbbol}
\DeclareSymbolFontAlphabet{\mathbb}{AMSb} 
\DeclareSymbolFontAlphabet{\mathbbl}{bbold}
\newcommand{\Prism}{{\mathlarger{\mathbbl{\Delta}}}} 

\newcommand{\ya}{{\rangle}}
\newcommand{\za}{{\langle}}

\begin{document}
\title{Hodge--Tate crystals on the logarithmic prismatic sites of semi-stable formal schemes}

\author{Yu Min\footnote{Y. M.: Morningside Center of Mathematics No.505, Chinese Academy of Sciences, ZhongGuancun East Road 55, Beijing, 100190, China. {\bf Email:} yu.min@amss.ac.cn} and Yupeng Wang\footnote{Y. W.: Morningside Center of Mathematics No.505, Chinese Academy of Sciences, ZhongGuancun East Road 55, Beijing, 100190, China. {\bf Email:} wangyupeng@amss.ac.cn}}


\maketitle

\begin{abstract}
  Let $\calO_K$ be a complete discrete valuation ring of mixed characteristic $(0,p)$ with a perfect residue field. In this paper, for a semi-stable $p$-adic formal scheme $\frakX$ over $\calO_K$ with rigid generic fibre $X$ and canonical log structure $\calM_{\frakX} = \calO_{\frakX}\cap\calO_X^{\times}$, we study Hodge--Tate crystals over the absolute logarithmic prismatic site $(\frakX,\calM_{\frakX})_{\Prism}$. As an application, we give an equivalence between the category of rational Hodge--Tate crystals on the absolute logarithmic prismatic site $(\frakX,\calM_{\frakX})_{\Prism}$ and the category of enhanced log Higgs bundles over $\frakX$, which leads to an inverse Simpson functor from the latter to the category of generalised representations on $X_{\proet}$.
\end{abstract}
\tableofcontents
\section{Introduction}
  The theory of prismatic cohomology was introduced by Bhatt--Scholze in \cite{BS-a}, which can be understood as a ``universal'' $p$-adic cohomology theory in the sense that it specialises to most cohomology theories (e.g. \'etale cohomology, crystalline cohomology, de Rham cohomology, Hodge--Tate cohomology and so on) in $p$-adic geometry. The original paper \cite{BS-a} was devoted to the relative theory, which has several geometric applications in the (integral) $p$-adic Hodge theory. For example, it recovers most works of \cite{BMS-a} and \cite{BMS-b} and can be used to study  local $p$-adic Simpson correspondence (c.f. \cite{MT}, \cite{Tian}, etc.). 
  
  On the other hand, the theory admits a variant which is called the absolute prismatic theory and has more arithmetic applications. For example, for a bounded $p$-adic formal scheme $\frakX$, the category of Laurent $F$-crystals on the absolute prismatic site $(\frakX)_{\Prism}$ is equivalent to the category of \'etale $\Zp$-local systems on the adic generic fibre $X$ of $\frakX$ (c.f. \cite{Wu}, \cite{BS-b}, \cite{MW-a}, etc.). In particular, one can recover some results in the classical theory of $(\varphi,\Gamma)$-modules on Galois representations through this equivalence. If moreover, $\frakX$ is smooth over $\calO_K$, the ring of integers of a finite extension $K$ of $\Qp$, then one can establish an equivalence between the category of prismatic $F$-crystals on $(\frakX)_{\Prism}$ and the category of crystalline $\Zp$-local systems on $X_{\et}$ (c.f. \cite{BS-b}, \cite{DL} for $\frakX = \Spf(\calO_K)$ and \cite{DLMS}, \cite{GR} for general $\frakX$). So it seems that the absolute prismatic theory could shed light on studying $p$-adic representations.
  
  There is also an application of the absolute theory to $p$-adic Simpson correspondence for rigid spaces with good reductions $\frakX$ over $\calO_K$. Indeed, the authors established an equivalence between the category of rational Hodge--Tate crystals on $(\frakX)_{\Prism}$ and the category of enhanced Higgs bundles on $\frakX_{\et}$ with coefficients in $\calO_{\frakX}[\frac{1}{p}]$ (c.f. \cite{MW-c} or Theorem \ref{MW-c}). On the other hand, the category of rational Hodge--Tate crystals on the sub-site $(\frakX)_{\Prism}^{\perf}$ of perfect prisms is equivalent to the category of generalised representations (i.e. vector bundles with coefficients in $\widehat \calO_X$) on $X_{\proet}$. So the restriction gives rise to an inverse Simpson functor from the category of enhanced Higgs bundles to the category of generalised representations, which turns out to be fully faithful. The construction is closely related with Sen theory. For example, when $\frakX = \Spf(\calO_K)$, we get classical Sen operators for $C$-representations of $G_K$ (c.f. \cite{Gao}, \cite{BL-a}, etc.). 
  
  However, in practice, there is no means that a rigid analytic variety $X$ over $K$ always admits a good reduction (if this is the case, the Galois representations associated to the \'etale cohomology of $X$ must be crystalline, which is not always true). So it is necessary to generalize stories mentioned above to rigid spaces with non-smooth reductions. A meaningful generalisation is that one may consider those with semi-stable reductions, which can be viewed as smooth objects in logarithmic geometry. Indeed, up to enlarging base field, rigid spaces always admit semi-stable reductions (c.f. \cite{Har}). This suggests us to study the logarithmic analogue of prismatic theory. The fundamental contribution in this direction is due to Koshikawa \cite{Kos}. Indeed, he introduced notions of log prism and logarithmic (relative) prismatic site, and then generalized several results of \cite{BS-a} to the logarithmic case. Another application of logarithmic prismatic theory is due to Du--Liu \cite{DL}. They showed the category of prismatic $F$-crystals on absolute logarithmic prismatic site $(\calO_K)_{\Prism,\log}$ can be identified with the category of semi-stable $\Zp$-representations of $G_K$, which obviously generalises the story about crystalline representations in the good reduction case. To the best of our knowledge, up to now, one could not see this from the original prismatic theory. So there do exist some new phenomena in the logarithmic setting. This paper is devoted to another direction and more precisely, to generalising main results in \cite{MW-b} and \cite{MW-c}. In other words, we study Hodge--Tate crystals in the logarithmic case, and show these objects can be identified with certain Higgs bundles and construct an inverse Simpson functor for these Higgs bundles (taking values in generalised representations). The specific statements of our results can be founded in the next subsection.
  
  Finally, it is worth pointing out that the usual absolute prismatic theory was recently independently studied by Drinfeld \cite{Dri} and Bhatt--Lurie \cite{BL-a}, \cite{BL-b} in a stacky way. Indeed, one can relate a $p$-adic formal scheme $\frakX$ to a certain stack, which is called the prismatization of $\frakX$, such that studying prismatic theory on $(\frakX)_{\Prism}$ amounts to studying coherent theory on the corresponding stack. So one may ask whether we can recover and generalise all known results mentioned above by using this stack. In particular, we hope that one can get an integral $p$-adic non-abelian Hodge theory in this way (for example, see \cite[Remark 9.2]{BL-b} when $\frakX$ is equipped with a $\delta$-structure over $\rW(k)$). 
\subsection{Main results}
  We use notations in Notations \ref{Notations} freely for stating our main results. Throughout this subsection, we fix a complete discretely valued $p$-adic field $K$ and let $\frakX$ be a semi-stable $p$-adic formal scheme over $\calO_K$ with rigid generic fibre $X$ and canonical log structure $\calM_{\frakX} = \calO_{\frakX}\cap\calO_X^{\times}$. We first specify the meaning of absolute logarithmic prismatic site $(\frakX,\calM_{\frakX})_{\Prism}$. 
  
  Recall that a (bounded) log prism defined by Koshikawa in \cite{Kos} is a tuple $(A,I,M,\delta_{\log})$ consisting of a bounded prism $(A,I)$ together with a prelog structure $\alpha:M\to A$ such that $(A,M)$ is $(p,I)$-adically log-affine, and a map $\delta_{\log}:M\to A$ compatible with $\delta$-structure of $A$ in the following sense:
  
       $(1)$ $\delta_{\log}(e_M) = 0$, where $e_M$ is the unity of the commutative monoid $M$;
       
       $(2)$ for any $m\in M$, $\alpha(m)^p\delta_{\log}(m) = \delta(\alpha(m))$;
       
       $(3)$ for any $m_1,m_2\in M$, $\delta_{\log}(m_1m_2) = \delta_{\log}(m_1)+\delta_{\log}(m_2)+p\delta_{\log}(m_1)\delta_{\log}(m_2)$.\\
 In particular, for any $m\in M$, $\varphi(\alpha(m)) = \alpha(m)^p(1+p\delta_{\log}(m))$ and hence $\varphi$ is a ``Frobenius'' endomorphism of log ring $(A,M)$. An object in $(\frakX,\calM_{\frakX})_{\Prism}$ is a log prism $(A,I,M,\delta_{\log})$, whose underlying log structure is integral, together with a structure map $f:\Spf(A/I)\to\frakX$ which induces an exact closed immersion $(\Spf(A/I),f^*\calM_{\frakX})\to (\Spf(A),\underline M)$ of log $(p,I)$-adic formal schemes. The morphisms in $(\frakX,\calM_{\frakX})_{\Prism}$ are defined in the obvious way and the topology is generated by ($p$-completely) flat covers. See Definition \ref{Dfn-log site} for details. Similar to the usual case, one can check that the rule $(A,I,M,\delta_{\log})\mapsto A$ (resp. $A/I$) defines a sheaf on $(\frakX,\calM_{\frakX})_{\Prism}$, which will be denoted by $\calO_{\Prism}$ (resp. $\overline \calO_{\Prism}$).

 For a given prism $(A,I)$ together with a log structure $\alpha:M\to A$, it is a strong restriction on $\alpha$ that there is a $\delta_{\log}$-structure $\delta_{\log}:M\to A$ making $(A,I,M,\delta_{\log})$ a log prism. For example, one can prove that 
 \begin{lem}[Lemma \ref{Lem-factorization}]\label{Intro-structure on deltalog}
   If a log prism $(A,I,M,\delta_{\log})$ is perfect (i.e. $(A,I)$ is a perfect prism in the usual sense), then for any $m\in M$, there exists an $x\in A$ such that $\alpha(m) = [\bar m](1+px)$, where $[\overline m]$ is the Teichm\"uller lifting of $\overline m$, the reduction of $\alpha(m)$ modulo $p$.
 \end{lem}
 On the other hand, being an object in $(\frakX,\calM_{\frakX})_{\Prism}$ is also a strong restriction on a log prism $(A,I,M,\delta_{\log})$. Indeed, if this is the case, (at least in our setting) the log structure $M\to A$ is somehow determined by $(\frakX,\calM_{\frakX})$ (c.f. Lemma \ref{Structure lemma}). This combined with Lemma \ref{Intro-structure on deltalog} shows that for a perfect prism $(A,I)$ in the usual prismatic site $(\frakX)_{\Prism}$, there exists a unique way to equip $A$ with a $\delta_{\log}$-structure $\delta_{\log}:M\to A$ such that $(A,I,M,\delta_{\log})\in(\frakX,\calM_{\frakX})_{\Prism}$ (cf. Lemma \ref{Lem-uniqueness}). In particular, when $\frakX = \Spf(R)$ is affine small; that is, $R$ is \'etale over $ \calO_K\za T_0,\dots, T_r, T_{r+1}^{\pm 1},\dots, T_d^{\pm 1}\ya/(T_0\cdots T_r-\pi)$, we have the identity of sites
 \[(\frakX,\calM_{\frakX})_{\Prism}^{\perf} = (\frakX)_{\Prism}^{\perf}.\]
 From this, we obtain the following result.
 \begin{lem}[Corollary \ref{Cor-Equal perfect site}]\label{Intro-Equal site}
   For a semi-stable $p$-adic formal scheme $\frakX$ over $\calO_K$, there is an equivalence of topoi $\Sh((\frakX)_{\Prism}^{\perf})\simeq \Sh((\frakX,\calM_{\frakX})_{\Prism}^{\perf})$.
 \end{lem}
 As a corollary, we have
 \begin{prop}[Theorem \ref{Thm-perfect HT is generalised repn}]\label{Intro-HTC vs GRep}
   There is a canonical equivalence of categories
   \[\Vect((\frakX,\calM_{\frakX})_{\Prism}^{\perf},\overline \calO_{\Prism}[\frac{1}{p}])\simeq\Vect(X_{\proet},\widehat \calO_X).\]
 \end{prop}
  Here, $\Vect(X_{\proet},\widehat \calO_X)$ denotes the category of generalised representations and $\Vect((\frakX,\calM_{\frakX})_{\Prism}^{\perf},\overline \calO_{\Prism}[\frac{1}{p}])$ denotes the category of rational Hodge--Tate crystals on the perfect site. The notion of Hodge--Tate crystals is similar to that in \cite{MW-c} and is specified as follows:
  \begin{dfn}[Definition \ref{Dfn-HT crystal}]\label{Intro-Dfn-HTC}
    By a {\bf Hodge--Tate crystal} on $(\frakX,\calM_{\frakX})_{\Prism}$, we mean a sheaf $\bM$ of $\overline \calO_{\Prism}$-modules satisfying the following properties:

          $(1)$ For any $\frakA = (A,I,M,\delta_{\log})\in (\frakX,\calM_{\frakX})_{\Prism}$, $\bM(\frakA)$ is a finite projective $A/I$-module.
          
          $(2)$ For any morphism $\frakA = (A,I,M,\delta_{\log})\to \frakB = (B,IB,N,\delta_{\log})$ in $(\frakX,\calM_{\frakX})_{\Prism}$, there is a canonical isomorphism
          \[\bM(\frakA)\otimes_{A/I}B/IB\to \bM(\frakB).\]

      We denote by $\Vect((\frakX,\calM_{\frakX})_{\Prism},\overline \calO_{\Prism})$ the category of Hodge--Tate crystals on $(\frakX,\calM_{\frakX})_{\Prism}$. 
      
      Similarly, we define {\bf rational Hodge--Tate crystals} on $(\frakX,\calM_{\frakX})_{\Prism}$ (resp. $(\frakX,\calM_{\frakX})_{\Prism}^{\perf}$) by replacing $\overline \calO_{\Prism}$ by $\overline \calO_{\Prism}[\frac{1}{p}]$ and denote the corresponding category by $\Vect((\frakX,\calM_{\frakX})_{\Prism},\overline \calO_{\Prism}[\frac{1}{p}])$ (resp. $\Vect((\frakX,\calM_{\frakX})^{\perf}_{\Prism},\overline \calO_{\Prism}[\frac{1}{p}])$).
  \end{dfn}
  Thanks to Proposition \ref{Intro-HTC vs GRep}, the restriction functor
  \[R:\Vect((\frakX,\calM_{\frakX})_{\Prism},\overline \calO_{\Prism}[\frac{1}{p}])\to\Vect((\frakX,\calM_{\frakX})_{\Prism}^{\perf},\overline \calO_{\Prism}[\frac{1}{p}])\]
  induces a functor from $\Vect((\frakX,\calM_{\frakX})_{\Prism},\overline \calO_{\Prism}[\frac{1}{p}])$ to $\Vect(X_{\proet},\widehat \calO_X)$. It can be shown that this functor is actually fully faithful (c.f. Corollary \ref{Cor-R is ff}).
  
  In order to state our main result, we introduce the notion of ``enhanced log Higgs bundles''.

  \begin{dfn}[Definition \ref{Dfn-enhanced Higgs module-G}]\label{Intro-Dfn-LogHiggs}
   By an {\bf enhanced log Higgs bundle} on $\frakX_{\et}$ with coefficients in $\calO_{\frakX}$, we mean a triple $(\calM,\theta_{\calM},\phi_{\calM})$ satisfying the following properties:
   
       $(1)$ $\calM$ is a locally finite free $\calO_{\frakX}$-module and \[\theta_{\calM}:\calM\to\calM\otimes_{\calO_{\frakX}}\widehat \Omega^1_{\frakX,\log}\{-1\}\]
       defines a nilpotent Higgs field on $\calM$, i.e. $\theta_M$ is a section of $\underline{\End}(\calM)\otimes_{\calO_{\frakX}}\widehat \Omega^1_{\frakX,\log}\{-1\}$ which is nilpotent and satisfies $\theta_{\calM}\wedge\theta_{\calM}=0$. Here ``$\bullet\{-1\}$'' denotes the Breuil--Kisin twist of $\bullet$. Denote by $\HIG(\calM,\theta_{\calM})$ the induced Higgs complex.
       
       $(2)$ $\phi_M\in\End(\calM)$ is ``topologically nilpotent'' in the following sense:
       \[\lim_{n\to+\infty}\prod_{i=0}^n(\phi_M+i\pi E'(\pi))=0\]
       and induces an endomorphism of $\HIG(\calM,\theta_M)$; that is, the following diagram
       \begin{equation*}
           \xymatrix@C=0.45cm{
           \calM \ar[rr]^{\theta_{\calM}\quad}\ar[d]^{\phi_{\calM}}&&\calM\otimes\widehat \Omega^1_{\frakX,\log}\{-1\} \ar[rr]^{\quad\theta_{\calM}}\ar[d]^{\phi_{\calM}+\pi E'(\pi)\id}&&\cdots\ar[rr]^{\theta_{\calM}\quad}&&\calM\otimes\widehat \Omega^d_{\frakX,\log}\{-d\}\ar[d]^{\phi_{\calM}+d\pi E'(\pi)\id}\\
           \calM \ar[rr]^{\theta_{\calM}\quad}&&\calM\otimes\widehat \Omega^1_{\frakX,\log}\{-1\} \ar[rr]^{\quad\theta_{\calM}}&&\cdots\ar[rr]^{\theta_{\calM}\quad}&&\calM\otimes\widehat \Omega^d_{\frakX,\log}\{-d\}
           }
       \end{equation*}
       commutes. Denote $\HIG(\calM,\theta_{\calM},\phi_{\calM})$ the total complex of this bicomplex.

   Denote by $\HIG^{\log}_*(\frakX,\calO_{\frakX})$ the category of enhanced log Higgs bundles over $\frakX$. Similarly, we define enhanced log Higgs bundles on $\frakX_{\et}$ with coefficients in $\calO_{\frakX}[\frac{1}{p}]$ and denote the corresponding category by $\HIG^{\log}_*(\frakX,\calO_{\frakX}[\frac{1}{p}])$. When $\frakX = \Spf(R)$ is small affine, we also denote $\HIG^{\log}_*(\frakX,\calO_{\frakX})$ (resp. $\HIG^{\log}_*(\frakX,\calO_{\frakX}[\frac{1}{p}])$) by $\HIG^{\log}_*(R)$ (resp. $\HIG^{\log}_*(R[\frac{1}{p}]))$ and call objects in this category {\bf enhanced log Higgs modules} over $R$ (resp. $R[\frac{1}{p}]$).
 \end{dfn}
 
  Let $\HIG_{G_K}^{\nil}(X_C)$ be the category of $G_K$-Higgs bundles on $X_{C,\et}$ (see Definition \ref{Dfn-Higgs module with Galois-G}). Then there exists a functor 
 \[F_1: \HIG^{\log}_*(\frakX,\calO_{\frakX}[\frac{1}{p}])\to \HIG_{G_K}^{\nil}(X_C),\]
 which sends an enhanced log Higgs bundle $(\calM,\theta_{\calM},\phi_M)$ to the $G_K$-Higgs module 
 \begin{equation*}
     (\calH=\calM\otimes_{\calO_{\frakX}}\calO_{X_{C}},\theta_{\calH}=\theta_{\calM}\otimes\lambda(\zeta_p-1)\id)
 \end{equation*}
 with the $G_K$-action on $\calH$ induced by the formulae
 \begin{equation*}
     g\mapsto (1-c(g)\lambda(1-\zeta_p)\pi E'(\pi))^{-\frac{\phi_M}{\pi E'(\pi)}}.
 \end{equation*}
 One can check that $F_1$ is fully faithful (c.f. Lemma \ref{Lem-F1 is ff}).

 Now, we state our main result as follows:
 \begin{thm}[Theorem \ref{Thm-HT crystal as enhanced Higgs field}, \ref{Thm-local Prismatic cohomology}, \ref{Thm-HT crystal as Higgs bundles-G}]\label{Intro-Main result}
     Assume $\frakX$ is a semi-stable $p$-adic formal scheme over $\calO_K$ of relative dimension $d$. Then there is an equivalence from the category $\Vect((\frakX,\calM_{\frakX})_{\Prism},\overline \calO_{\Prism}[\frac{1}{p}])$ of rational Hodge--Tate crystals to the category ${\rm HIG}^{\log}_*(\frakX,\calO_{\frakX}[\frac{1}{p}])$ of enhanced log Higgs bundles with coefficients in $\calO_{\frakX}[\frac{1}{p}]$, which fits into the following commutative diagram
     \begin{equation*}
       \xymatrix@C=0.5cm{
         \Vect((\frakX,\calM_{\frakX})_{\Prism},\overline \calO_{\Prism}[\frac{1}{p}])\ar[r]^R\ar[d]^{\simeq}&\Vect((\frakX,\calM_{\frakX})_{\Prism}^{\perf},\overline \calO_{\Prism}[\frac{1}{p}])\ar[r]^{\qquad\simeq }&\Vect(X,\OX)\ar[d]^{\simeq }\\
         \HIG^{\log}_*(\frakX, \calO_{\frakX}[\frac{1}{p}])\ar[rr]^{F_1}&&\HIG^{\nil}_{G_K}(X_{C}).
       }
   \end{equation*}
     Here, all arrows are fully faithful and we use ``$\simeq$'' to denote equivalences of categories.
     
     If moreover $\frakX = \Spf(R)$ is small affine, then the above equivalence upgrades to the integral and derived level. More precisely, there is an equivalence (depending on the framing chosen) from the category $\Vect((\frakX,\calM_{\frakX})_{\Prism},\overline \calO_{\Prism})$ of Hodge--Tate crystals to the category ${\rm HIG}^{\log}_*(R)$ of enhanced log Higgs modules over $R$. In this case, there is a quasi-isomorphism 
     \[\RGamma((\frakX,\calM_{\frakX})_{\Prism},\bM)\simeq \HIG(H,\Theta,\phi)\]
     for any Hodge--Tate crystal $\bM$ with associated enhanced log Higgs module $(H,\Theta,\phi)$.
 \end{thm}
 \begin{rmk}
   In the classical Sen theory, given a $C$-representation $V$ with Sen operator $\phi_V$, One can always assume the matrix of $\phi_V$ has coefficients in $K$ (\cite[Theorem 5]{Sen}) with a suitable choice of basis of $V$. However, for a matrix $A$ with coefficients in $K$, it does not necessary come from a $C$-representation. Our result shows that if \[\lim_{n\to+\infty}\prod_{i=0}^n(A+i\pi E'(\pi)) = 0,\]
   then $A$ can be realized as the Sen operator of some $C$-representations. We conjecture a similar result in the relative case (c.f. Conjecture \ref{Intro-Conjecture}).
 \end{rmk}
 As a corollary, we get the following finiteness result, which is an analogue of \cite[Theorem 2.8, 2.9]{Tian} and \cite[Corollary 2.21]{MW-c} in the absolute logarithmic case.
 \begin{cor}
   Keep notations as in Theorem \ref{Intro-Main result}. Then for any Hodge--Tate crystal $\bM\in\Vect((\frakX,\calM_{\frakX})_{\Prism},\overline \calO_{\Prism})$, $\rR\nu_*(\bM)$ is a perfect complex of $\calO_{\frakX}$-modules with tor-amplitude in $[0,d+1]$, where $\nu_*:\Sh((\frakX,\calM_{\frakX})_{\Prism})\to\Sh(\frakX_{\et})$ is the natural morphism of topoi. As a consequence, if moreover $\frakX$ is proper, then $\RGamma((\frakX,\calM_{\frakX})_{\Prism},\bM)$ is a perfect complex of $\calO_K$-modules with tor-amplitude in $[0,2d+1]$. Similar results hold for rational Hodge--Tate crystals.
 \end{cor} 
 Note that if $\frakX$ is smooth, one can regard it as a log formal scheme, whose canonical log structure is actually induced by the composition $(\bN\xrightarrow{1\mapsto \pi}\calO_K\to\calO_{\frakX})$. Then one can consider the absolute logarithmic prismatic site $(\frakX,\calM_{\frakX})_{\Prism}$. Clearly, the rule $(A,I,M,\delta_{\log})\mapsto(A,I)$ defines a functor $(\frakX,\calM_{\frakX})_{\Prism}\to(\frakX)_{\Prism}$, which gives rise to a natural functor 
 \[\Vect((\frakX)_{\Prism},\overline \calO_{\Prism}[\frac{1}{p}])\to \Vect((\frakX,\calM_{\frakX})_{\Prism},\overline \calO_{\Prism}[\frac{1}{p}]).\]
 Comparing Theorem \ref{Intro-Main result} with \cite[Theorem 1.12]{MW-c} (or Theorem \ref{MW-c}), we get the following result:
 
 \begin{cor}[Corollary \ref{Cor-fully f-rel-local}, \ref{Cor-R is ff}]\label{Intro-good vs semistable}
    There is a commutative diagram of categories
    \[
   \xymatrix@C=0.45cm{
     \Vect((\frakX)_{\Prism},\overline \calO_{\Prism}[\frac{1}{p}]) \ar[d]^{\simeq}\ar[r]& \Vect((\frakX,\calM_{\frakX})_{\Prism},\overline \calO_{\Prism}[\frac{1}{p}])\ar[d]^{\simeq}\\
     \HIG^{\nil}_*(\frakX,\calO_{\frakX}[\frac{1}{p}])\ar[r]&\HIG^{\log}_*(\frakX,\calO_{\frakX}[\frac{1}{p}]),
   }
   \]
where the bottom functor sends an enhanced Higgs bundle $(\calH,\theta_{\calH},\phi_{\calH})$ to the enhanced log Higgs bundle $(\calH,\theta_{\calH},\pi\phi_{\calH})$ (and hence is fully faithful). In particular, the top horizontal arrow is fully faithful. When $\frakX=\Spf(R)$ is small affine, the diagram also holds on the integral level.
 \end{cor}
 Indeed, the commutative diagram in Corollary \ref{Intro-good vs semistable} is also compatible with the equivalence in Proposition \ref{Intro-HTC vs GRep} and Simpson correspondence (c.f. Remark \ref{Rmk-Big comm diagram}).  
 
 Finally, for a generalised representation $\calL$ with associated $G_K$-Higgs bundle $(\calH,\theta_{\calH})$, one can define an arithmetic Sen operator $\phi_{\calL}$ on $\calH$ such that $(\calH,\theta_{\calH},\phi_{\calL})$ is an arithmetic Higgs bundle in the sense of Definition \ref{Dfn-arithmetic Higgs module-G}. In particular, when $\frakX = \Spf(\calO_K)$, the arithmetic Sen operator coincides with classical Sen operator.
 
 Now, for a rational Hodge--Tate crystal $\bM\in \Vect((\frakX,\calM_{\frakX})_{\Prism},\overline \calO_{\Prism}[\frac{1}{p}])$ with associated enhanced log Higgs bundle $(\calH,\theta_{\calH},\phi_{\calH})$ and generalised representation $\calL$, we make the following conjecture:
 \begin{conj}\label{Intro-Conjecture}
   With notations as above, we have $\phi_{\calL} = -\frac{\phi_{\calH}}{\pi E'(\pi)}$.\footnote{The authors knew from Hui Gao that he could confirm the conjecture by using a relative version of Kummer Sen theory in \cite{Gao}.}
 \end{conj}
  Similar conjectures have already appeared in \cite{MW-b} and \cite{MW-c}. We confirm the above conjecture in the case for $\frakX = \Spf(\calO_K)$ (c.f. Theorem \ref{Thm-crys-vs-rep-abs}).

\subsection{Notations}\label{Notations}
  In this paper, we fix a complete discrete valuation ring $\calO_K$ of mixed characteristic $p$ with perfect residue field $k$ and fractional field $K$. Let $\bar K$ be a fixed algebraic closure of $K$, $G_K:=\Gal(\bar K/K)$ be the absolute Galois group of $K$, $C$ be the $p$-adic completion of $\bar K$ and $\pi$ be a fixed uniformizer of $\calO_K$ with minimal ploynomial $E(u)\in\frakS:=\rW(k)[[u]]$ over $\rW(k)$. We fix  compatible systems $\{\zeta_{p^n}\}_{n\geq 0}$ and $\{\pi^{\frac{1}{p^n}}\}_{n\geq 0}$ of primitive $p$-power roots of $1$ and $\pi$, respectively. Then we get elements $\epsilon = (1,\zeta_p,\dots)$ and $\pi^{\flat} = (\pi,\pi^{\frac{1}{p^n}},\dots)$ in $\calO_C^{\flat}$. Define $K_{\cyc} = \cup_{n\geq 1}K(\zeta_{p^n})$ and $K_{\infty} = \cup_{n\geq 1}K(\pi^{\frac{1}{p^n}})$. Then $K_{\cyc,\infty}:=K_{\cyc}K_{\infty} = \cup_{n\geq 1}K_n$ for $K_n=K(\zeta_{p^n},\pi^{\frac{1}{p^n}})$. Denote by $\hat L$ the $p$-adic completion for $L\in\{K_{\cyc},K_{\infty},K_{\cyc,\infty}\}$. Note that both $K_{\cyc,\infty}$ and $K_{\cyc}$ are Galois extension of $K$. We denote by $\widehat G_K = \Gal(K_{\cyc,\infty}/K)$ and $\Gamma_K = \Gal(K_{\cyc}/K)$ the corresponding Galois groups and then get an exact sequence\footnote{The sequence always splits and induces an isomorphism $\widehat G_K\simeq \Zp(1)\rtimes\Gamma_K$ when $p\geq 3$ (cf. \cite[Lemma 5.1.2]{Liu}) and for suitable choice of $\pi$ when $p=2$ (cf. \cite[Lemma 2.1]{Wang}). In these cases, it is easy to see that $K_{\cyc}\otimes_K K_{\infty} = K_{\cyc,\infty}$.}
  \[1\to\Zp(1)\to\widehat G_K\to \Gamma_K\to 1.\]
  Let $\chi:\Gamma_K\to\bZ_p^{\times}$ be the cyclotomic character and $c:G_K\to\Zp$ be the map determined by $\sigma(\pi^{\flat}) = \epsilon^{c(\sigma)}\pi^{\flat}$ for any $\sigma\in G_K$. Let $\Ainf = \rW(\calO_C^{\flat})$ be the infinitesimal period ring of Fontaine and $\theta:\Ainf\to\calO_C$ be the natural surjection. Then $\Ker(\theta)$ is generated by either $E([\pi^{\flat}])$ or $\xi = \frac{\mu}{\varphi^{-1}(\mu)}$ for $\mu = [\epsilon]-1$. Define $\lambda = \theta(\frac{\xi}{E[\pi^{\flat}]})$, which is a unit in $\calO_C$. We always equip $\calO_K$ with the log structure associated to $\bN\xrightarrow{1\to \pi}\calO_K$ and equip $\frakS$ with the $\delta$-structure such that $\varphi(u) = u^p$.
  
  In this paper, we essentially consider semi-stable $p$-adic formal schemes\footnote{When we say a $p$-adic formal scheme is semi-stable, we always assume it is separated.} $\frakX$ over $\calO_K$ with rigid generic fibres $X$. In other words, \'etale locally, $\frakX = \Spf(R)$ is affine such that there exists an \'etale morphism 
  \[\Box:\calO_K\za T_0,\dots, T_r,T_{r+1}^{\pm 1},\dots, T_d^{\pm 1}\ya/\za T_0\cdots T_r-\pi\ya \to R\]
  of $\calO_K$-algebras for some $0\leq r\leq d$. In this case, we say $\frakX$ or $R$ is {\bf small affine}. We also say $\frakX$ {\bf admits a chart} if this is the case. We often consider such an $\frakX$ as a log formal scheme with the canonical log structure $\calM_{\frakX}\to\calO_{\frakX}$ for $\calM_{\frakX} = \calO_{\frakX}\cap\calO_X^{\times}$. When $\frakX = \Spf(R)$ is small affine as above, $\calM_{\frakX}$ is induced by the pre-log structure $\bN^{r+1}\xrightarrow{e_i\mapsto T_i,\forall i}R$, where $e_i$ denotes the generator of $(i+1)$-th component of $\bN^{r+1}$ for $0\leq i\leq r$.

  \subsection{Organizations}

  In Section $2$, we give a quick review on logarithmic prismatic site and prove some basic properties for logarithmic prismatic theory. In section $3$, we study Hodge--Tate crystals on absolute logarithmic prismatic site for a semi-stable $p$-adic formal scheme over $\calO_K$ and show Simpson correspondence in the local case. We first deal with the $\Spf(\calO_K)$ case and then move to the higher dimensional case. In Section $4$, we glue local constructions and deduce a global theory.
\section*{Acknowledgement}
  The authors thank Heng Du for valuable discussions. The first author is supported by China Postdoctoral Science Foundation E1900503.
\section{The logarithmic prismatic site}
 The logarithmic prismatic site was introduced by Koshikawa \cite{Kos} as an analogue of the prismatic site defined in \cite{BS-a} in the theory of logarithmic geometry. In this section, we give a quick review of its definition and provide some basic relevant properties. 
 
 \begin{dfn}[\emph{\cite[Definition 2.2]{Kos}}]
   A {\bf $\delta_{\log}$-ring} is a tuple $(A,\delta,M\xrightarrow{\alpha} A,\delta_{\log}:M\to A)$ consisting of a $\delta$-ring $(A,\delta)$, a prelog structure $\alpha:M\to A$ and a map $\delta_{\log}:M\to A$ such that

       $(1)$ $\delta_{\log}(e_M) = 0$, where $e_M$ is the unity of the commutative monoid $M$;
       
       $(2)$ for any $m\in M$, $\alpha(m)^p\delta_{\log}(m) = \delta(\alpha(m))$;
       
       $(3)$ for any $m_1,m_2\in M$, $\delta_{\log}(m_1m_2) = \delta_{\log}(m_1)+\delta_{\log}(m_2)+p\delta_{\log}(m_1)\delta_{\log}(m_2).$

   We often write $(A,M,\delta_{\log})$ for a $\delta_{\log}$-ring for simplicity.
   
   A morphism $f:(A,M,\delta_{\log})\to (B,N,\delta_{\log})$ of $\delta_{\log}$-rings is a morphism of prelog rings compatible with both $\delta$ and $\delta_{\log}$.
 \end{dfn}

 Let $(A,M\xrightarrow{\alpha}A,\delta_{\log})$ be a $\delta_{\log}$-ring and $\varphi$ be the induced Frobenius endomorphism on $A$. Then the rule
 \[m\mapsto 1+p\delta_{\log}(m)\]
 defines a morphism of monoids $M\to 1+pA$ such that for any $m\in M$, 
 \begin{equation}\label{Equ-Frob on log structure}
     \varphi(\alpha(m)) = \alpha(m)^p(1+p\delta_{\log}(m)).
 \end{equation}
 
 If $p$ belongs to the Jacobson radical of $A$, then $1+pA\subset A^{\times}$. This implies that for any $n\geq 1$, $\varphi^n(\alpha(m))$ differs from $\alpha(m)^{p^m}$ by a unit. Therefore if $\alpha: M\to A$ is furthermore a log structure (i.e. the restriction of $\alpha$ on $\alpha^{-1}(A^{\times})$ induces an isomorphism $\alpha^{-1}(A^{\times})\to A^{\times}$), then the rule $m\mapsto m^p\alpha^{-1}(1+p\delta_{\log}(m))$ defines an endomorphism $\phi$ of $M$ and $(\varphi:A\to A,\phi:M\to M)$ is an endomorphism of the log structure $(M\to A)$. In this case, $(\varphi,\phi)$ is a lifting of the Frobenius endomorphism of the log structure $(M^a\to A/p)$, where $(M^a\to A/p)$ is the associated log structure on $A/p$ induced by the prelog structure $(M\to A/p)$.

 \begin{dfn}[\emph{\cite[Definition 3.3]{Kos}}]
   A {\bf (bounded) prelog prism} is a tuple $(A,I,M,\delta_{\log})$ such that $(A,I)$ is a (bounded) prism and $(A,M,\delta_{\log})$ is a $\delta_{\log}$-ring. A bounded prelog prism $(A,I,M,\delta_{\log})$ is called a {\bf log prism} if $(A,M)$ is $(p,I)$-adically log-affine. Morphisms of (pre-)log prisms are defined as morphisms of underlying $\delta_{\log}$-rings.
 \end{dfn}
 
 Every bounded prelog prism is uniquely associated to a log prism in the following sense:
 
 \begin{lem}[\emph{\cite[Proposition 2.14,Corollary 2.15]{Kos}}]\label{Lem-prolog-to-log}
   Let $(A,M\xrightarrow{\alpha}A,\delta_{\log})$ be a $\delta_{\log}$-ring and $N$ be the monoid fitting into the following push-out diagram of monoids:
     \[
     \xymatrix@C=0.5cm{
     \alpha^{-1}(A^{\times})\ar[rr]\ar[d]&& A^{\times}\ar[d]\\
     M\ar[rr]&& N
     }.
     \]
   Then $(A,N)$ admits a unique $\delta_{\log}$-structure compatible with $(A,M,\delta_{\log})$. If moreover $A$ is classically $J$-adically complete for some finite generated ideal $J$ of $A$ containing $p$ (e.g. $(A,I,M,\delta_{\log})$ is a bounded prelog prism and $J=I+pA$), then the above construction is compatible with the \'etale localization of the log $J$-adic formal scheme $(\Spf(A),\underline M^a)$.
 \end{lem}
 If no ambiguity appears, for a bounded prelog prism $(A,I,M,\delta_{\log})$, we denote its associated log prism by $(A,I,M^a,\delta_{\log})$, where $M^a:=\Gamma(\Spf(A),\underline M^a)$.

 Now, we follow \cite[Definition 4.1]{Kos} and \cite[Definition 5.0.6]{DL} to define the absolute logarithmic prismatic site for a log $p$-adic formal scheme $(\frakX,\calM_{\frakX})$.
 
 \begin{dfn}[Absolute logarithmic prismatic site]\label{Dfn-log site}
   Let $(\frakX,\calM_{\frakX})$ be a log $p$-adic formal scheme. We denote by $(\frakX,\calM_{\frakX})_{\Prism}$ the opposite of the category whose objects are log prisms $\frakA=(A,I,M,\delta_{\log})$ with integral log structures together with maps of $p$-adic formal schemes $f_{\frakA}:\Spf(A/I)\to \frakX$ which induce exact closed immersions of log $(p,I)$-adic formal schemes $(\Spf(A/I),f_{\frakA}^*\calM_{\frakX})\to(\Spf(A),\underline M^a)$\footnote{This amounts to that the log structure $f_{\frakA}^*\calM_{\frakX}$ on $\Spf(A/I)$ coincides with that  associated to the composition $M\to A\to A/I$.}. 
   A morphism 
   \[g:\frakA = (A,I,M,\delta_{\log}) \to \frakB = (B,J,N,\delta_{\log})\]
   of objects $\frakA,\frakB\in (\frakX,\calM_{\frakX})_{\Prism}$ is a morphism of log prisms making the following diagram
   \[
   \xymatrix@C=0.45cm{
     (\Spf(B),\underline N^a)\ar[rr]&&(\Spf(A),\underline M^a)\\
     (\Spf(A/I),f_{\frakB}^*\calM_{\frakX})\ar[rr]\ar[u]\ar[rd]&&(\Spf(A/I),f_{\frakA}^*\calM_{\frakX})\ar[u]\ar[ld]\\
     & (\frakX,\calM_{\frakX})
   }
   \]
   of morphisms of log $(p,I)$-adic formal schemes commute. Such a morphism is a cover if and only if $A\to B$ is $(p,I)$-adically faithfully flat and the induced morphism of $(\Spf(B),\underline N^a)\to (\Spf(A),\underline M^a)$ is strict\footnote{This amounts to that the log structure $N$ on $B$ coincides with that associated to the composition $M\to A\to B$.}. We define the structure sheaf $\calO_{\Prism}$ on $(\frakX,\calM_{\frakX})_{\Prism}$ by sending an object $(A,I,M,\delta_{\log})$ to $A$.
 \end{dfn}
  When $\frakX = \Spf(R)$ is affine and the log structure $\calM_{\frakX}$ is induced by an integral prelog structure $(M\to R)$, we also denote the absolute logarithmic prismatic site $(\frakX,\calM_{\frakX})_{\Prism}$ by $(M\to R)_{\Prism}$.

 \begin{dfn}[The perfect logarithmic prismatic site]
   A prelog prism $(A,I,M,\delta_{\log})$ is called {\bf perfect} if $(A,I)$ is a perfect prism. For a log $p$-adic formal scheme $(\frakX,\calM_{\frakX})$, we denote by $(\frakX,\calM_{\frakX})_{\Prism}^{\perf}$ the full subcategory of $(\frakX,\calM_{\frakX})_{\Prism}$ whose objects are perfect log prisms.
 \end{dfn}
 
 
     
    
 \begin{rmk}[Perfection]\label{perfection}
   For a prelog prism $(A,I,M,\delta_{\log})$ with induced Frobenius lifting $\varphi$, we denote by $A_{\perf}$ the $(p,I)$-completion of the perfection of $A$ with respect to $\varphi$. Then $(A_{\perf},I,M\to A\to A_{\perf},\delta_{\log})$ is the initial object in the category of prelog prisms over $(A,I,M,\delta_{\log})$.
 \end{rmk}
 
 Now we give some basic examples of log prisms used in this paper.
 \begin{exam}\label{Exam-abs case}
       $(1)$ Let $M_{\frakS}\to\frakS$ be the log structure on $\frakS$ associated to the prelog structure $\bN\xrightarrow{1\mapsto u}\frakS$ and let $\delta_{\log}(u) = 0$. Then $(\frakS,(E),M_{\frakS},\delta_{\log})$ is a log prism in $(\bN\xrightarrow{1\mapsto \pi}\calO_K)_{\Prism}$, which will be referred as {\bf Breuil--Kisin log prism}.
       
       $(2)$ Let $M_{\Ainf}\to\Ainf$ be the log structure on $\Ainf$ associated to the prelog structure $\bN\xrightarrow{1\mapsto [\pi^{\flat}]}\Ainf$. Then $(\Ainf,(\xi),M_{\Ainf},\delta_{\log})$ is a perfect log prism in $(\bN\xrightarrow{1\mapsto \pi}\calO_K)_{\Prism}$, which will be referred as {\bf Fontaine log prism}.
       
       $(3)$ The map $\iota:\frakS\to \Ainf$ sending $u$ to $[\pi^{\flat}]$ induces a morphism 
       \[(\frakS,(E),M_{\frakS},\delta_{\log})\to (\Ainf,(\xi),M_{\Ainf},\delta_{\log})\]
       of log prisms, which is a cover in $(\bN\xrightarrow{1\mapsto \pi}\calO_K)_{\Prism}$.
 \end{exam}
 
 \begin{exam}\label{Exam-rel case}
   Let $\Box:\calO_K\za T_0,\dots, T_r,T_{r+1}^{\pm 1},\dots, T_d^{\pm 1}\ya/\za T_0\cdots T_r-\pi\ya \to R$ be an \'etale morphism and $\alpha:\bN^{r+1}\to R$ be the morphism of monoids sending $e_i$ to $T_i$ for any $0\leq i\leq r$, where $e_i$ denotes the generator of $(i+1)$-th component of $\bN^{r+1}$. Let $\widehat R_{\infty}$ be the $p$-complete base change to $\widehat A_{\infty}$,
which is the $p$-adic completion of the union $\cup_{n\geq 0} A_n$ for
   \[A_n = \calO_C\za T_0^{\frac{1}{p^{n}}},\dots, T_r^{\frac{1}{p^{n}}},T_{r+1}^{\pm \frac{1}{p^{n}}},\dots, T_d^{\pm \frac{1}{p^{n}}}\ya/\za T_0^{\frac{1}{p^{n}}}\cdots T_r^{\frac{1}{p^{n}}}-\pi^{\frac{1}{p^{n}}}\ya,\]
   along the \'etale morphism $\Box$.
   
       $(1)$ Let $\frakS(R)$ be the unique lifting of $R$ over 
       \[\frakS\za T_0,\dots, T_r,T_{r+1}^{\pm 1},\dots, T_d^{\pm 1}\ya/\za T_0\cdots T_r-u\ya\]
       induced by $\Box$ and $\bN^{r+1}\to \frakS(R)$ be the morphism of monoids sending $e_i$ to $T_i$ for any $0\leq i\leq r$ which induces a log structure $M_{\frakS(R)}$ on $\frakS(R)$. Put $\delta_{\log}(e_i) = \delta(T_j) = 0$ for any $0\leq i\leq r<j\leq d$. Then $(\frakS(R),(E),M_{\frakS(R)},\delta_{\log})$ is a log prism in $(\bN^{r+1}\xrightarrow{\alpha}R)_{\Prism}$.
       
       $(2)$ Let $M_{\Ainf(\widehat R_{\infty})}$ be the log structure on $\Ainf(\widehat R_{\infty})$ associated to the prelog structure $\bN^{r+1}\to \Ainf(\widehat R_{\infty})$ sending $e_i$ to $[T_i^{\flat}]$, where $T_i^{\flat} = (T_i,T_i^{\frac{1}{p}},\dots)$ in $\widehat R_{\infty}^{\flat}$ for any $0\leq i\leq r$. Then $(\Ainf(\widehat R_{\infty}),(\xi),M_{\Ainf(\widehat R_{\infty})},\delta_{\log})$ is a perfect log prism in $(\bN^{r+1}\xrightarrow{\alpha}R)_{\Prism}$.
       
       $(3)$ The map $\frakS(R)\to\Ainf(\widehat R_{\infty})$ sending $u$ to $[\pi^{\flat}]$ and $T_i$'s to $[T_i^{\flat}]$'s defines a morphism 
       \[(\frakS(R),(E),M_{\frakS(R)},\delta_{\log})\to(\Ainf(\widehat R_{\infty}),(\xi),M_{\Ainf(\widehat R_{\infty})},\delta_{\log})\]
       of log prisms, which is a cover in $(\bN^{r+1}\xrightarrow{\alpha}R)_{\Prism}$.
 \end{exam}
 
 The following structure lemma says that in some cases, log structures of prisms in $(\frakX,\calM_{\frakX})_{\Prism}$ are ``uniformly'' determined by the log structure $\calM_{\frakX}$ on $\frakX$.

 \begin{lem}\label{Structure lemma}
     $(1)$ For any log prism $(A,I,M,\delta_{\log})$ in $ (\bN\xrightarrow{1\mapsto \pi}\calO_K)_{\Prism}$, there exists a lifting $u$ of $\pi\in A/I$ such that the log structure $M$ is associated to the prelog structure $(\bN\xrightarrow{1\mapsto u}A)$.
     
     $(2)$ Let $(\bN^{r+1}\to R)$ be as in Example \ref{Exam-rel case}. Then for any log prism $(A,I,M\xrightarrow{\alpha}A,\delta_{\log})\in (\bN^{r+1}\to R)_{\Prism}$, there exist liftings $t_i$'s of $T_i$'s in $A$ such that $M$ is associated to the prelog structure $\bN^{r+1}\to A$ which sends $e_i$ to $t_i$ for any $0\leq i\leq r$.
 \end{lem}
 \begin{proof}
   The item $(1)$ is \cite[Lemma 5.0.10]{DL}. The proof of item $(2)$ is similar but for the convenience of readers, we show details here.
   
   By virtues of Definition \ref{Dfn-log site}, the log structures on $A/I$ associated to the prelog structures $(\bN^{r+1}\to R\to A/I)$ and $(M\to A\to A/I)$ coincide. Therefore, there exists $m_i$'s in $M$ such that $\alpha(m_i)=\bar u_iT_i$ in $(A/I)$ for $\bar u_i$'s in $(A/I)^{\times}$. Since $A$ is $I$-adically complete, $\bar u_i$'s lift to units $u_i$'s in $A^{\times}$. So replacing $m_i$'s by $m_iu_i$'s, we may assume $\alpha(m_i) =T_i\mod I$. Let $M'\to A$ be the log structure associated to the prelog structure
   \[(\bN^{r+1}\xrightarrow{e_i\mapsto m_i,\forall 0\leq i\leq r} M\xrightarrow{\alpha} A).\]
   Then we get a morphism $f:(M'\to A)\to (M\to A)$ of log structures whose induced log structures modulo $I$ coincide. It remains to show $f$ is the identity itself. Since $A$ is $I$-complete, this follows from deformation theory of log structures (e.g. \cite[Theorem 8.36]{Ols}) directly.
 \end{proof}
 
 \begin{rmk}[Rigidity of log structures]\label{rigidity}
   Let $(A,I,M,\delta_{\log})\in (\frakX,\calM_{\frakX})_{\Prism}$ be a log prism. For any $A/I$-algebra $\overline B$, which is endowed with the log structure $(M^a\to \overline B)$ associated to $(M\to A/I\to \overline B)$, the deformation problem for lifting $(M^a\to \overline B)$ to a log prism over $(A,I,M,\delta_{\log})$ coincides with that for lifting $\overline B$ to a prism over $(A,I)$. Indeed, by \cite[Lemma 8.22, Lemma 8.26]{Ols}, there exists a canonical quasi-isomorphism between ($(p,I)$-completions of) cotangent complexes $\rL_{(M^a\to \overline B)/(M\to A/I)}$ and $\rL_{\overline B/(A/I)}$. In particular, for any $(B,IB,N,\delta_{\log})\in (\frakX,\calM_{\frakX})_{\Prism}$ over $(A,I,M,\delta_{\log})$, the log structure $N\to B$ has to be associated to the prelog structure $(M\to A\to B)$.
 \end{rmk}
 \begin{rmk}
   Let $(A,I,M,\delta_{\log})\to(B,IB,N,\delta_{\log})$ be a cover in $(\frakX,\calM_{\frakX})_{\Prism}$. For any log prism $(C,IC,L,\delta_{\log})\in (\frakX,\calM_{\frakX})_{\Prism}$ over $(A,I,M,\delta_{\log})$, let $D$ denote the $(p,I)$-complete tensor product $B\widehat \otimes_AC$ which is endowed with the log structure $(L^a\to D)$ associated to the composition $L\to C\to D$. Then $\delta_{\log}: L\to C\to D$ defines a log prism $(D,ID,L^a,\delta_{\log})$ which is the pushout of the diagram
   \[(B,IB,N,\delta_{\log})\leftarrow(A,I,M,\delta_{\log})\rightarrow(C,IC,L,\delta_{\log}).\]
   In particular, let $(B^{\bullet},IB^{\bullet})$ be the \v Cech nerve of the cover $(A,I)\to (B,IB)$ of prisms and let $N^{\bullet}\to B^{\bullet}$ be the log structure induced by $M\to A\to B^{\bullet}$. Then $(B^{\bullet},IB^{\bullet},N^{\bullet},\delta_{\log})$ is the \v Cech nerve of the cover $(A,I,M,\delta_{\log})\to(B,IB,N,\delta_{\log})$ of log prisms in $(\frakX,\calM_{\frakX})_{\Prism}$. The argument in the proof of \cite[Corollary 3.12]{BS-a} shows that the presheaves $(A,I,M,\delta_{\log})\mapsto A$ and $(A,I,M,\delta_{\log})\mapsto A/I$ are indeed sheaves. The latter is denoted by $\overline \calO_{\Prism}$.
 \end{rmk}
 At the end of this section, we want to study structures of (certain) perfect log prisms. We begin with the following lemma:
 
 \begin{lem}\label{Lem-factorization}
   Let $(A,I)$ be a perfect prism and $R:= A/I$. Then for any $x\in A$ such that $\varphi(x) = x^p(1+py)$ for some $y\in A$, $x$ factors as 
   \[x = [a]\prod_{i=1}^{\infty}(1+p\varphi^{-i}(y))^{p^{i-1}},\]
   where $a$ is the reduction of $x$ modulo $p$.
 \end{lem}
 Note that $[a]\prod_{i=1}^{\infty}(1+p\varphi^{-i}(y))^{p^{i-1}}$ is well-defined as $A$ is classically $p$-complete.
 \begin{proof}
   Put $y_n = \varphi^{-n}(y)$ for any $n\geq 0$. We are reduced to showing that for any $n\geq 1$,
   \[x \equiv [a]\prod_{i=1}^{\infty}(1+py_i)^{p^{i-1}}\mod p^n .\]
   Write $x = [a]+px_1$ for the unique $x_1\in A$. Then the assumption on $x$ can be restated as follows:
   \[[a]^p+p\varphi(x_1) = ([a]+px_1)^p(1+p\varphi(y_1)).\]
   So we get that $\varphi(x_1)\equiv \varphi([a]y_1) \mod p$ and a fortiori that $x_1 \equiv [a]y_1 \mod p$. Therefore, we can write $x = [a](1+py_1)+p^2x_2$ for the unique $x_2\in A$.
   
   Assume $x = [a]\prod_{i=1}^n(1+py_i)^{p^{i-1}}+p^{n+1}x_{n+1}$ for some $n\geq 1$ with $x_{n+1}\in A$. Then the assumption on $x$ says that
   \[\begin{split}
   [a]^p\prod_{i=0}^{n-1}(1+py_i)^{p^i}+p^{n+1}\varphi(x_{n+1}) &= ([a]\prod_{i=1}^n(1+py_i)^{p^{i-1}}+p^{n+1}x_{n+1})^p(1+py)\\
   & \equiv [a]^p\prod_{i=1}^n(1+py_i)^{p^i}(1+py) \mod p^{n+2}.
   \end{split}\]
   As a consequence, we have
   \[p^{n+1}x_{n+1} \equiv [a]\prod_{i=1}^n(1+py_i)^{p^{i-1}}((1+py_{n+1})^{p^n}-1)\mod p^{n+2}.\]
   In particular, we get $x \equiv [a]\prod_{i=1}^{n+1}(1+py_i)^{p^{i-1}}\mod p^{n+2}$ as desired.
 \end{proof}
  The following corollary is clear.
 \begin{cor}\label{Cor-final log structure}
   Let $(A,I,M\xrightarrow{\alpha}A,\delta_{\log})$ be a perfect log prism and $R = A/I$, then $\alpha(M)\subset [R^{\flat}]\cdot (1+pA)$, where $[R^{\flat}]: = \{[x]| x\in R^{\flat}\}$. In particular, if we equip $A$ with the log structure $(N\to A)$ associated to the prelog structure $R^{\flat}\xrightarrow{\sharp} A$, then there exists a unique morphism of log prisms $(A,I,M,\delta_{\log})\to(A,I,N,\delta_{\log})$.
 \end{cor}

 As mentioned in \cite[Theorem 3.10]{BS-a}, there is a canonical equivalence between the category of perfect prisms and the category of perfectoid rings. We want to generalise this to the logarithmic case.
 \begin{lem}\label{Lem-p-power}
   Let $S$ be a $p$-torsion free perfectoid ring. Then for any $x\in S[\frac{1}{p}]$ such that $x^p\in S$, we have $x\in S$.
 \end{lem}
 \begin{proof}
   This is well-known and we provide a proof here for the convenience of readers.
   
   By \cite[Lemma 3.9]{BMS-a}, there exists a $\varpi\in S^{\flat}$ such that $\varpi^{\sharp}$ is a unit multiple of $p$. Define $\rho:=\varphi^{-1}(\varpi)^{\sharp}$. Then \cite[Lemma 3.10]{BMS-a} implies that the Frobenius induces an isomorphism 
   \[S/\rho S\xrightarrow{\cong}S/\rho^p S.\]
   Now, assume we are given an $x\in S[\frac{1}{p}]$ such that $x^p\in S$. Then there exists an integer $N$ such that $\rho^Nx\in S$ and we denote by $l$ the smallest one satisfying above property. It is enough to show that $l\leq 0$. Otherwise, assume that $l\geq 1$. Then we see that $\rho^lx$ is not zero in $S/\rho S$, which implies that $\rho^{pl}x^p$ does not vanish modulo $\rho^pS$. This violates to that $l\geq 1$ together with $x^p\in S$.
 \end{proof}
 \begin{lem}\label{Lem-tilt}
     Let $R$ be a perfectoid ring with tilt $R^{\flat}$. Let $a_1,a_2\in R^{\flat}$ and $x\in R$. If $a_1^{\sharp}=a_2^{\sharp}(1+px)$ and $a_1^{\sharp}\mid p$, then there exists a unit $u\in R^{\flat}$ such that $a_1=a_2u$ and $u^{\sharp} = 1+px$.
 \end{lem}
 \begin{proof}
   Let $\varpi\in R^{\flat}$ such that $\varpi^{\sharp}$ is a unit multiple of $p$. 
   
   We first assume $R$ is $p$-torsion free and hence $R^{\flat}$ is $\varpi$-torsion free. In this case, $c:=\frac{a_1}{a_2}$ is well-defined in $R^{\flat}[\frac{1}{\varpi}]$ with $c^{\sharp} = (1+px)$. By Lemma \ref{Lem-p-power}, for any $n\geq 0$, $(\varphi^{-n}(c))^{\sharp}\in R$. So $c\in R^{\flat}$. Combining this with that $c^{\sharp}\in R^{\times}$, we see that $c\in (R^{\flat})^{\times}$ and therefore $u = c$ is desired.
   
   In general, by \cite[Proposition 3.2]{Bha}, if we put $S = R/R[\sqrt{pR}]$, $\overline R=R/\sqrt{pR}$ and $\overline S=S/\sqrt{pS}$, then there is a fibre square of perfectoid rings:
    \begin{equation*}
        \xymatrix@C=0.5cm{
         R\ar[rr]\ar[d]&&S\ar[d]\\
         \overline R\ar[rr]&&\overline S.}
    \end{equation*}
   By what we have proved, there exists a $c\in (S^{\flat})^{\times}$ which is represented by the compatible sequence $\{c_n\}_{n\geq 0}$ such that $a_1=a_2c$ and $c^{\sharp}=c_0 = (1+px)$ in $S^{\flat}$. Since $c_0$ coincides with $1$ in $\overline S$, so are $c_n$'s. In particular, denote by $\tilde c_n$ the element $(1,c_n)\in\overline R\times_{\overline S}S\cong R$, then $\tilde c_0=1+px$ and $\tilde c_{n+1}^p=\tilde c_n$ for all $n$. Put $\tilde c\in R^{\flat}$ which is represented by $\{\tilde c_n\}$. Then $u=\tilde c$ is desired.
 \end{proof}
 \begin{rmk}
   We only proved above Lemma \ref{Lem-tilt} for $R$ absolutely integral closed in the early draft. This stronger version is due to some discussions with Heng Du.
 \end{rmk}

 \begin{lem}\label{Lem-uniqueness}
   For any perfect log prisms $(A,I,M_1,\delta_{\log})$ and $(A,I,M_2,\delta_{\log})$ in $(\bN\xrightarrow{1\mapsto \pi}\calO_K)_{\Prism}$ or $(\bN^{r+1}\xrightarrow{\alpha}R)_{\Prism}$ with the same underlying prism, we always have 
   $(A,I,M_1,\delta_{\log})=(A,I,M_2,\delta_{\log})$.
 \end{lem}
 \begin{proof}
   We only deal with the $(\bN^{r+1}\xrightarrow{\alpha}R)_{\Prism}$ case and the $(\bN\xrightarrow{1\mapsto \pi}\calO_K)_{\Prism}$ case follows from the similar argument.
   
   By Lemma \ref{Structure lemma}, for any $0\leq i\leq r$, there are liftings $t_{i1},t_{i2}\in A$ of $T_i$ such that the log structures $M_1\to A$ and $M_2\to A$ are associated to prelog structures $(\bN^{r+1}\xrightarrow{e_i\mapsto t_{i1},\forall 0\leq i\leq r}A)$ and $(\bN^{r+1}\xrightarrow{e_i\mapsto t_{i2},\forall 0\leq i\leq r}A)$ for $0\leq i\leq r$, respectively.
   By Corollary \ref{Cor-final log structure}, there are $a_{i1},a_{i2}\in (A/I)^{\flat}$ and $x_{i1},x_{i2}\in A$ such that $t_{i1}=[a_{i1}](1+px_{i1})$ and $t_{i2}=[a_{i2}](1+px_{i2})$ for any $0\leq i\leq r$.
   Let $\theta:A\to A/I$ be the natural surjection. Then for all $i$, we have 
   \[a_{i1}^{\sharp}(1+p\theta(x_{i1})) = T_i = a_{i2}^{\sharp}(1+p\theta(x_{i2})).\]
   Since $A/I$ is classically $p$-complete, for any $0\leq i\leq r$, there exists a $y_i\in A/I$ such that $a_{i1}^{\sharp} = a_{i2}^{\sharp}(1+py_i)$. By Lemma \ref{Lem-tilt}, we deduce that $a_{i1}$ and $a_{i2}$ differ from a unit in $(A/I)^{\flat}$. So for any $0\leq i\leq r$, there exists $u_i\in A^{\times}$ such that $t_{i1}=t_{i2}u_i$ as desired.
 \end{proof}

 \begin{prop}\label{Equal perfect site}

       
    Let $(\frakX,\calM_{\frakX})_{\Prism}$ be $(\bN\xrightarrow{1\mapsto \pi}\calO_K)_{\Prism}$ or $(\bN^{r+1}\xrightarrow{\alpha}R)_{\Prism}$. Then one can identify the sites $(\frakX,\calM_{\frakX})^{\perf}_{\Prism} = (\frakX)_{\Prism}^{\perf}$ via the forgetful functor $(A,I,M,\delta_{\log})\mapsto (A,I)$.
 \end{prop}
 \begin{proof}
   We only deal with the $(\bN^{r+1}\to R)_{\Prism}$ case. By Lemma \ref{Lem-uniqueness}, it is enough to show that for any perfect prism $(A,I)$, there exists a log structure $M\to A$ making $(A,I,M,\delta_{\log})$ an object in $(\bN^{r+1}\to R)_{\Prism}$. Let $\varpi\in (A/I)^{\flat}$ such that $\varpi^{\sharp}$ is a unit multiple of $p$ and let $\rho = \varphi^{-1}(\varpi)^{\sharp}$ as before. By \cite[Lemma 3.9]{BMS-a}, for any $0\leq i\leq r$, one can find a compatible sequence $\{t_{i,n}\}_{n\geq 0}$ in $A/I$ such that $t_{i,0}$ coincides with $T_i$ and for any $n\geq 0$, $t_{i,n+1}^p$ coincides with $t_{i,n}$ modulo $\rho p A/I$. Let $\tilde t_i\in (A/I)^{\flat}$ be the element determined by $\{t_{i,n}\}_{n\geq 0}$. Then we obtain that 
   \[\tilde t_i^{\sharp} = \lim_{n\to +\infty}t_{i,n}^{p^n}.\]
   In particular, we have
   \[\tilde t_i^{\sharp}\equiv T_i \mod \rho pA/I.\]
   Write $\tilde t_i^{\sharp} = T_i+\rho p x_i$ for some $x_i\in A/I$.
   For any $0\leq i\leq r$, put $T_i':= \prod_{0\leq j\leq r,j\neq i}T_j$. Then $T_iT_i' = \pi$. In particular, there exists some $y_i\in R$ such that $p = T_iT_i'y_i$.
   Now, we obtain that 
   \[\tilde t_i^{\sharp} = T_i+\rho p x_i = T_i(1+\rho T_i'x_iy_i) \]
   is a unit multiple of $T_i$. It is easy to see that the log structure associated to 
   \[\oplus_{i=0}^r\bN e_i\xrightarrow{e_i\mapsto \tilde t_i,~\forall i}A\]
   is desired. We win!
 \end{proof}
 
 Let $\frakX$ be a separated semi-stable formal scheme over $\calO_K$ with the rigid analytic generic fibre $X$ equipped with the log structure $\calM_{\frakX} = \calO_X^{\times}\cap\calO_{\frakX}\to\calO_{\frakX}$. Since \'etale locally on $\frakX$, $(\frakX,\calM_{\frakX})$ is induced by the prelog structure $\bN^{r+1}\to R$ as described in Example \ref{Exam-rel case}, then the following result follows from Proposition \ref{Equal perfect site} directly.
 
 \begin{cor}\label{Cor-Equal perfect site}
   With notations as above, we have $\Sh((\frakX)_{\Prism}^{\perf}) = \Sh((\frakX,\calM_{\frakX})_{\Prism}^{\perf})$.
 \end{cor}
 
 \begin{rmk}
   In this subsection, most of results (from Lemma \ref{Structure lemma} to Corollary \ref{Cor-Equal perfect site}) are proved for log formal schemes $(\frakX,\calM_{\frakX})$ whose log structures are locally induced by finite generated free monoids (i.e. locally, $\frakX = \Spf(R)$ with $\calM_{\frakX}$ induced by $\bN^s\to R$ for some $s\geq 0$). We believe that those relevant results still hold true in a more general setting and for example, hold for those $(\frakX,\calM_{\frakX})$ with $\calM_{\frakX}$ just fine (or even integral). Since results given above are enough for our use, we will not handle the most general case in this paper.
 \end{rmk}
 \section{The Hodge--Tate crystals on absolute logarithmic prismatic site}
   In this section, we always assume $\frakX$ is a semi-stable $p$-adic formal scheme over $\calO_K$ with the log structure $\calM_{\frakX} = \calO_X^{\times}\cap\calO_{\frakX}$, where $X$ is the rigid analytic generic fibre of $\frakX$. \'Etale locally, $\frakX=\Spf(R)$ and $\calM_{\frakX}$ is induced by the prelog structure $\bN^{r+1}\to R$ as described in Example \ref{Exam-rel case}.
   
   Our purpose is to generalize results in \cite{MW-b} and \cite{MW-c} to the logarithmic case. That is, we want to study Hodge--Tate crystals on the absolute logarithmic prismatic site $(\frakX,\calM_{\frakX})_{\Prism}$. Before moving on, we specify the meaning of Hodge--Tate crystals as follows:
   
   \begin{dfn}[The Hodge--Tate crystals]\label{Dfn-HT crystal}
      By a {\bf Hodge--Tate crystal} on $(\frakX,\calM_{\frakX})_{\Prism}$, we mean a sheaf $\bM$ of $\overline \calO_{\Prism}$-modules satisfying the following properties:

          $(1)$ For any $\frakA = (A,I,M,\delta_{\log})\in (\frakX,\calM_{\frakX})_{\Prism}$, $\bM(\frakA)$ is a finite projective $A/I$-module.
          
          $(2)$ For any morphism $\frakA = (A,I,M,\delta_{\log})\to \frakB = (B,IB,N,\delta_{\log})$ in $(\frakX,\calM_{\frakX})_{\Prism}$, there is a canonical isomorphism
          \[\bM(\frakA)\otimes_{A/I}B/IB\to \bM(\frakB).\]

      We denote by $\Vect((\frakX,\calM_{\frakX})_{\Prism},\overline \calO_{\Prism})$ the category of Hodge--Tate crystals on $(\frakX,\calM_{\frakX})_{\Prism}$. 
      
      Similarly, we define {\bf rational Hodge--Tate crystals} on $(\frakX,\calM_{\frakX})_{\Prism}$ (resp. $(\frakX,\calM_{\frakX})_{\Prism}^{\perf}$) by replacing $\overline \calO_{\Prism}$ by $\overline \calO_{\Prism}[\frac{1}{p}]$ and denote the corresponding category by $\Vect((\frakX,\calM_{\frakX})_{\Prism},\overline \calO_{\Prism}[\frac{1}{p}])$ (resp. $\Vect((\frakX,\calM_{\frakX})^{\perf}_{\Prism},\overline \calO_{\Prism}[\frac{1}{p}])$).
   \end{dfn}
   
   In \cite{MW-b} and \cite{MW-c}, the authors proved the following theorem:
   \begin{thm}[\emph{\cite[Theorem 1.6, 1.12]{MW-c}}]\label{MW-c}
     Assume $\frakX$ is a smooth $p$-adic formal scheme over $\calO_K$ of relative dimension $d$. Then there is an equivalence from the category $\Vect((\frakX)_{\Prism},\overline \calO_{\Prism}[\frac{1}{p}])$ of rational Hodge--Tate crystals to the category ${\rm HIG}^{\nil}_*(\frakX,\calO_{\frakX}[\frac{1}{p}])$ of triples $(\calH,\Theta,\phi)$\footnote{Such a triple was referred as an {\bf enhanced Higgs bundle} in \emph{loc.cit.}.} consisting of 
     
     $(1)$ a Higgs bundle $(\calH,\Theta)$ with coefficients in $\calO_{\frakX}[\frac{1}{p}]$ such that $\Theta$ is nilpotent and 
     
     $(2)$ an endomorphism $\phi$ of $\calH$ such that $[\Theta,\phi]=E'(\pi)\Theta$ and $\lim_{n\to+\infty}\prod_{i=0}^n(\phi+iE'(\pi)) = 0$.
     
     The equivalence fits into the following commutative diagram
     \begin{equation*}
       \xymatrix@C=0.5cm{
         \Vect((\frakX)_{\Prism},\overline \calO_{\Prism}[\frac{1}{p}])\ar[r]^R\ar[d]^{\simeq}&\Vect((\frakX)_{\Prism}^{\perf},\overline \calO_{\Prism}[\frac{1}{p}])\ar[r]^{\quad\simeq }&\Vect(X,\OX)\ar[d]^{\simeq }\\
         \HIG^{\nil}_*(\frakX, \calO_{\frakX}[\frac{1}{p}])\ar[rr]&&\HIG^{\nil}_{G_K}(X_{C}),
       }
   \end{equation*}
     wihere all arrows are fully faithful functors. Here, we use ``$\simeq$'' to denote equivalences of categories.
     
     If moreover $\frakX = \Spf(R)$ is small affine, then the above equivalence upgrades to the integral and derived level. More precisely, there is an equivalence (depending on the chosen framing) from the category $\Vect((R)_{\Prism},\overline \calO_{\Prism})$ of Hodge--Tate crystals to the category ${\rm HIG}^{\nil}_*(R)$ of triples $(H,\Theta,\phi)$\footnote{Such a triple was referred as an {\bf enhanced Higgs module} in \emph{loc.cit.}.} consisting of a finite projective $R$-module $H$, a nilpotent Higgs field $\Theta$ on $H$ and an endomorphism $\phi$ of $H$ satisfying $[\Theta,\phi]=E'(\pi)\Theta$ and $\lim_{n\to+\infty}\prod_{i=0}^n(\phi+iE'(\pi)) = 0$. In this case, there is a quasi-isomorphism 
     \[\RGamma((R)_{\Prism},\bM)\simeq [H\xrightarrow{\Theta} H\otimes_R\widehat \Omega_R^1\{-1\}\xrightarrow{\Theta}\cdots\xrightarrow{\Theta} H\otimes_R\widehat \Omega_R^d\{-d\}]^{\phi=0}\]
     for any Hodge--Tate crystal $\bM$ with associated triple $(H,\Theta,\phi)$.
   \end{thm}
   In particular, when $\frakX = \Spf(\calO_K)$, the theorem can be formulated as follows:
   \begin{thm}[\emph{\cite[Theorem 1.2, 1.3, 1.7]{MW-b},\cite[Theorem 4.3.3]{Gao}}]\label{MW-b}
     The evaluation at Breuil--Kisin prism $(\frakS,(E))$ induces an equivalence from the category of Hodge--Tate crystals on $(\calO_K)_{\Prism}$ to the category ${\rm HT}(\calO_K)$  of pairs $(M,\phi_M)$ consisting of a finite prejective $\calO_K$-module $M$ and an $\calO_K$-linear endomorphism $\phi_M$ of $M$ such that 
     \[\lim_{n\to+\infty}\prod_{i=0}^n(\phi_M+iE'(\pi)) = 0.\]
     The above equivalence is still true for rational Hodge--Tate crystals by using finite dimensional $K$-vector spaces instead of finite projective $\calO_K$-modules. Moreover, let $\bM$ be a (rational) Hodge--Tate crystal on $(\calO_K)_{\Prism}$ with associated pair $(M,\phi_M)$, then the following assertions are true:
     
         $(1)$ There exists a natural quasi-isomorphism 
         \[\rR\Gamma((\calO_K)_{\Prism},\bM)\cong [M\xrightarrow{\phi_M}M].\]
         
         $(2)$ The restriction to $(\calO_K)_{\Prism}^{\perf}$ induces a fully faithful functor from the category of rational Hodge--Tate crystals over $\calO_K$ to the category of (semi-linear) continuous $C$-representations of $G_K$:
         \[\Vect((\calO_K)_{\Prism},\overline \calO_{\Prism}[\frac{1}{p}])\to \Rep_{G_K}(C).\]
         If we denote by $V$ the corresponding $C$-representation of $\bM$ and $\Theta_V$ the Sen operator of $V$, then we have $V=M\otimes_KC$ with $\Theta_V = -\frac{\phi_M}{E'(\pi)}$ and a natural quasi-isomorphism 
         \[\rR\Gamma((\calO_K)_{\Prism},\bM)\cong \rR\Gamma(G_K,V).\]
   \end{thm}
   
   The purpose of this section is to generalise these results to the logarithmic case (for semi-stable $\frakX$ with the canonical log structure $\calM_{\frakX}$).
 \subsection{The $\calO_K$ case}
   We first assume $\frakX = \Spf(\calO_K)$ with the log structure $\calM_{\frakX}$ corresponding to the prelog structure $(\bN\xrightarrow{1\mapsto \pi}\calO_K)$. We start with the following lemma:
   \begin{lem}\label{Lem-cover abs}
          $(1)$ The Breuil--Kisin log prism $(\frakS,(E),M_{\frakS},\delta_{\log})$ is a cover of the final object of the topos $\Sh((\bN\xrightarrow{1\mapsto \pi}\calO_K)_{\Prism})$.
          
          $(2)$ The Fontaine log prism $(\Ainf,(\xi),M_{\inf},\delta_{\log})$ is a cover of the final object of the topos $\Sh((\bN\xrightarrow{1\mapsto \pi}\calO_K)_{\Prism}^{\perf})$.
   \end{lem}
   \begin{proof}
          $(1)$ This is \cite[Proposition 5.0.16]{DL}\footnote{The authors knew from Heng Du that the result was first obtained by T. Koshikawa.}.
          
          $(2)$ 
          This follows from \cite[Lemma 2.2 (2)]{MW-b} and Proposition \ref{Equal perfect site}.
   \end{proof}
   
   For any $n\geq 0$, let $A^n = \rW(k)[[u_0,1-\frac{u_1}{u_0},\dots,1-\frac{u_n}{u_0}]]$ and define 
   \[\frakS^n = A^n\{\frac{1-u_1/u_0}{E(u_0)},\dots,\frac{1-u_n/u_0}{E(u_0)}\}^{\wedge}_{\delta},\]
   where the completion is with respect to the $(p,E(u_0))$-adic topology.
   Then $(\frakS^n,(E(u_0)))$ is a prism in $(\calO_K)_{\Prism}$ and $\frac{u_i}{u_0}$'s are units in $\frakS^n$ for all $i$. So the log structures on $\frakS^n$ associated to $\bN\xrightarrow{1\mapsto u_i}\frakS^n$ are independent of the choice of $i$. We denote this log structure by $(M_{\frakS^n}\to \frakS^n)$. Note that $E(u_i)-E(u_0)$ is divided by $u_i-u_0 = u_0E(u_0)\frac{u_i/u_0-1}{E(u_0)}$. By \cite[Lemma 2.24]{BS-a}, the ideal $E(u_i)\frakS^n$ is also independent of the choice of $i$.
   
   \begin{lem}\label{Lem-cech nerve abs}
     Keep notations as above.
     
     $(1)$ For any $n\geq 1$, the prism $(\frakS^n,(E(u_0)))$ is bounded. As a consequence, $(\frakS^n,(E(u_0)),M_{\frakS^n},\delta_{\log})\in (\bN\to\calO_K)_{\Prism}$.
         
     $(2)$ The cosimplicial log prism $(\frakS^{\bullet},(E),M_{\frakS^{\bullet}},\delta_{\log})$ is the \v Cech nerve associated to the Breuil--Kisin log prism.
   \end{lem}
   \begin{proof}
     This is essentially \cite[Lemma 5.0.12]{DL}, but for the convenience of readers, we repeat the proof here. 
     
     Note that $(\frakS,(E))$ is a bounded prism and $A^n$ is faithfully flat over $\frakS$. Since $p, E(u_0), \frac{u_1}{u_0}-1, \dots, \frac{u_n}{u_0}-1$ is a regular sequence in $A^n$, by \cite[Proposition 3.13]{BS-a}, $(\frakS^n,(E))$ is $(p,E)$-faithfully flat over $(\frakS,(E))$. In particular, it is bounded. 
     
     To finish the proof, it remains to show $(\frakS^n,(E),M_{\frakS^n},\delta_{\log})$ is the initial object in the category of log prisms $(A,I,M,\delta_{\log})$ in $(\bN\to\calO_K)_{\Prism}$ to which there are $n+1$ morphisms from $(\frakS,(E),M_{\frakS},\delta_{\log})$. We only deal with the $n=1$ case and the general case follows from a similar argument.
     
     Put $\frakS^n_0 = \rW(k)[[u_0,\dots,u_n]]\{\frac{u_0-u_1}{E(u_0)},\cdots,\frac{u_0-u_n}{E(u_0)}\}^{\wedge}_{\delta}$ with log structure induced by the map $\oplus_{i=0}^n\bN e_n\xrightarrow{e_i\mapsto u_i,\forall i} \frakS_0^n$. Then $(\frakS^{\bullet}_0,(E(u_0)))$ is the \v Cech nerve for Breuil--Kisin prism $(\frakS,(E))$ in the usual prismatic site $(\calO_K)_{\Prism}$. In particular, $(\frakS_0^1,(E(u_0)))$ is the self coproduct of $(\frakS,(E))$ in $(\calO_K)_{\Prism}$. Note that there are obvious maps $\delta_{\log}:\bN^{\bullet+1}\cong \bN e_0\oplus\cdots\oplus\bN e_{\bullet}\to\frakS^{\bullet}_0$ making $(\frakS_0^{\bullet},(E(u_0)),\bN^{\bullet+1},\delta_{\log})$ a cosimplicial prelog prism. 
     
     Let $(A,I,M,\delta_{\log})$ be a log prism with morphisms $f_0,f_1:(\frakS,(E),M_{\frakS},\delta_{\log})\to (A,I,M,\delta_{\log})$. Then there is a unique morphism of prelog prisms $f:(\frakS_0^1,(E(u_0)),\bN^2,\delta_{\log})\to(A,I,M,\delta_{\log})$. However, $f$ is not a morphism in $(\bN\to\calO_K)_{\Prism}$ as $(\frakS_0^1,(E(u_0)),\bN^2,\delta_{\log})$ is not an object in this category. The reason is that the morphism 
     \[(\Spf(\frakS_0^1/E(u_0)),\underline{(\bN)^a)})\to (\Spf(\frakS_0^1),\underline{(\bN^2)^a})\]
     is not exact, where the log structure for thr former is induced by $\bN\to\calO_K\to\frakS_0^1/E(u_0)$. Fortunately, by \cite[Proposition 3.7]{Kos}, there exists a log prism $(\tilde \frakS_0^1,(E(u_0)),N,\delta_{\log})\in (\bN\to\calO_K)_{\Prism}$\footnote{The log prism $(\tilde \frakS_0^1,(E(u_0)),N,\delta_{\log})$ is called the {\bf exactification} of $(\frakS_0^1,(E(u_0)),\bN^2,\delta_{\log})$ in \emph{loc.cit.}.} over $(\frakS_0^1,(E(u_0)),\bN^2,\delta_{\log})$ which is initial among log prisms $(C,(E(u_0)),L,\delta_{\log})\in(\calO_K)_{\Prism}$ fitting into the following commutative diagram
     \[
     \xymatrix@C=0.45cm{
       (\frakS_0^1,(\bN^2)^a)\ar[d]^{\rm pr}\ar[rr]&&(\tilde \frakS_0^1,N)\ar[d]\\
       (\frakS_0^1/E(u_0),(\bN\to\calO_K\to\frakS_0^1/E(u_0))^a)\ar[rr]&&(\tilde \frakS_0^1/E(u_0),(\bN\to\calO_K\to\tilde \frakS_0^1/E(u_0))^a),
     }
     \]
     where we require the right vertical map induces an exact closed immersion of $(p,E(u_0))$-adic log formal schemes. In particular, $(\tilde \frakS_0^1,(E(u_0)),N,\delta_{\log})$ is the self coproduct of Breuil--Kisin log prism in $(\bN\to \calO_K)_{\Prism}$. It remains to show $(\tilde \frakS_0^1,(E(u_0)),N,\delta_{\log}) = (\frakS^1,(E(u_0)),M_{\frakS^1},\delta_{\log})$. By construction, the projection ${\rm pr}$ appearing in above diagram is induced by $\pr:\bN^2\cong \bN e_0\oplus \bN e_1\to\bN$ with $\pr(e_i)=1$ for $i=1,2$. Consider the submonoid $M$ of $(\bN^2)^{\gp}\cong \bZ e_0\oplus\bZ e_1$ generated by $\bN^2$ and the kernel $(\pr^{\gp})^{-1}(0)$ of $\pr^{\gp}:(\bN^2)^{\gp}\to(\bN)^{\gp}$. Then $M = \bN e_0+\bN e_1+\bZ(e_0-e_1)$. By \cite[Construction 2.18]{Kos} (and \cite[Proposition 2.16]{Kos}), we see that $\tilde \frakS_0^1 = \rW(k)[[u_0,u_1]][(\frac{u_0}{u_1})^{\pm 1}]\{\frac{u_0-u_1}{E(u_0)}\}^{\wedge}_{\delta}$, where the completion is with respect to the $(p,E(u_0))$-adic topology. Now, \cite[Proposition 3.13]{BS-a} shows that $p,E(u_0)$ forms a regular sequence in $\tilde \frakS_0^1$ (as it does in $\rW(k)[[u_0,u_1]][(\frac{u_0}{u_1})^{\pm 1}]$) and hence so is $u_0,E(u_0)$. Since both $u_0$ and $E(u_0)$ divide $u_0-u_1$ in $\tilde \frakS_0^1$, we get $\frac{1-u_1/u_0}{E(u_0)}\in\tilde \frakS_0^1$. This implies that $\tilde \frakS_0^1 = \frakS^1$ as desired.
     
   \end{proof}
   
   For any $n\geq 0$, put
   \[\rA_{\inf}^n:=(\Ainf\widehat \otimes_{\rW(k)}\cdots\widehat \otimes_{\rW(k)}\Ainf[[1-\frac{u_1}{u_0},\dots,1-\frac{u_n}{u_0}]]\{\frac{1-u_1/u_0}{E(u_0)},\dots,\frac{1-u_n/u_0}{E(u_0)}\}^{\wedge}_{\delta})^{\perf},\]
   where $\Ainf\widehat \otimes_{\rW(k)}\Ainf\widehat \otimes_{\rW(k)}\cdots\widehat \otimes_{\rW(k)}\Ainf$ is the $(p,E)$-completed tensor product of $(n+1)$-copies of $\Ainf$ over $\rW(k)$ and $u_i$ denotes the corresponding $[\pi^{\flat}]$ of the $(i+1)$-component of this product. Then the log structure associated to $(\bN\xrightarrow{1\mapsto u_i}\rA_{\inf}^n)$ is independent of the choice of $i$. We denote this log structure by $M_{\rA_{\inf}^n}\to \rA_{\inf}^n$. The similar argument in the proof of Lemma \ref{Lem-cech nerve abs} shows that $(\rA_{\inf}^n,(\xi),M_{\rA_{\inf}^n},\delta_{\log})$ is a log prism and the cosimplicial log prism $(\rA_{\inf}^{\bullet},(\xi),M_{\rA_{\inf}^{\bullet}},\delta_{\log})$ is the \v Cech nerve associated to the Fontaine log prism $(\Ainf,(\xi),M_{\inf},\delta_{\log})\in (\bN\to \calO_K)_{\Prism}^{\perf}$.
   
   \begin{lem}\label{Lem-cech perf abs}
     The cosimplicial prism $(\rA_{\inf}^{\bullet},(\xi))$ is the \v Cech nerve associated to $(\Ainf,(\xi))\in (\calO_K)^{\perf}_{\Prism}$.
   \end{lem}
   \begin{proof}
     This is a corollary of Proposition \ref{Equal perfect site}.
   \end{proof}
   
   Similar to \cite[Theorem 3.24]{MW-b}, we have the following theorem:
   \begin{thm}\label{Thm-HT vs Rep-abs}
      The evaluation at Fontaine log prism $(\Ainf,(\xi),M_{\Ainf},\delta_{\log})$ induces an equivalence from the category of rational Hodge--Tate crystals $\Vect((\bN\xrightarrow{1\mapsto \pi}\calO_K)_{\Prism}^{\perf},\overline \calO_{\Prism}[\frac{1}{p}])$ to the category of $C$-representations $\Rep_{G_K}(C)$ of $G_K$.
   \end{thm}
   \begin{proof}
     This follows from Lemma \ref{Lem-cech perf abs} and the proof of \cite[Theorem 3.24]{MW-b}.
   \end{proof}
   
   Now we focus on (rational) Hodge--Tate crystals on $(\bN\xrightarrow{1\mapsto \pi}\calO_K)_{\Prism}$. As in \cite{MW-b}, the key point is to describe the structure of cosimplicial ring $\frakS^{\bullet}/E\frakS^{\bullet}$.
   
 \begin{lem}\label{pd poly-abs}
  For any $1\leq i\leq n$, denote $X_i$ the image of $\frac{1-u_0/u_i}{E(u_0)}\in \frakS^{n}$ modulo $(E)$, then $\frakS^n/(E) \cong \calO_K\{X_1,\dots, X_n\}^{\wedge}_{\rm pd}$ is the free pd-polynomial ring on the variables $X_1,\dots, X_n$. Moreover, for any $0\leq i\leq n+1$, let $p_i:\frakS^n\to\frakS^{n+1}$ be the structure morphism induced by the order-preserving map \[\{0,\dots,n\}\to\{0,\dots,i-1,i+1\dots,n+1\}\]
  then via above isomorphisms, we have
  \begin{equation}\label{Equ-pd poly-abs}
      p_i(X_j) = \left\{
      \begin{array}{rcl}
           (X_{j+1}-X_1)(1-\pi E'(\pi)X_1)^{-1}, & i=0  \\
           X_j, & j<i \\
           X_{j+1}, & 0<i\leq j.
      \end{array}
      \right.
  \end{equation}
 \end{lem}
 \begin{proof}
   Just modify the proof of \cite[Lemma 2.6]{MW-b}. The only difference appears in the formula
   (\ref{Equ-pd poly-abs}) for $p_0$. We need to show
   \[p_0(\frac{1-u_j/u_0}{E(u_0)}) = \frac{1-u_{j+1}/u_1}{E(u_1)} = \frac{u_1/u_0-u_{j+1}/u_0}{E(u_0)}\frac{u_0}{u_1}\frac{E(u_0)}{E(u_1)}\]
   goes to $(X_{j+1}-X_1)(1-\pi E'(\pi)X_1)^{-1}$ modulo $E$. But this follows from that
   \begin{equation*}
   \begin{split}
       E(u_0)-E(u_1)&\equiv E'(u_1)(u_0-u_1) \mod (u_0-u_1)^2\frakS^n\\
       &\equiv u_0E'(u_0) (1-u_1/u_0) \mod (E)^2
   \end{split}    
   \end{equation*}
   easily. We win!
 \end{proof}
 In order to compute logarithmic prismatic cohomology for Hodge--Tate crystals, we need the following lemma.
 \begin{lem}\label{Lem-cech-to-derived-abs}
   Let $\bM$ be a (rational) Hodge--Tate crystal on $(\bN\to\calO_K)_{\Prism}$ and $\frakA = (A,I,M,\delta_{\log})\in (\bN\to\calO_K)_{\Prism}$ be a log prism. Then for any $i\geq 1$, we have 
   \[\rH^i(\frakA,\bM) = 0.\]
 \end{lem}
 \begin{proof}
   For any log prism $\frakB = (B,IB,N,\delta_{\log})\in (\bN\to\calO_K)_{\Prism}$ which is a cover of $\frakA$, by Remark \ref{rigidity}, the \v Cech nerve associated to this cover exists and its underlying cosimplicial prism coincides with the \v Cech nerve associated to the cover $(A,I)\to (B,IB)$ of prisms in $(\calO_K)_{\Prism}$. Now the result follows from the same argument used in \cite[Lemma 3.11]{Tian}.
 \end{proof}
 
 Now we are able to prove the logarithmic analogue of \cite[Theorem 3.8, Theorem 3.12]{MW-b}:
 
 \begin{thm}\label{Thm-prismatic cohomology-abs}
   The evaluation at Breuil--Kisin log prism $(\frakS,(E),M_{\frakS},\delta_{\log})$ induces an equivalence from the category $\Vect((\bN\to\calO_K)_{\Prism}, \overline \calO_{\Prism})$ of Hodge--Tate crystals on $(\bN\to\calO_K)_{\Prism}$ to the category ${\rm HT}^{\log}(\calO_K)$  of pairs $(M,\phi_M)$ consisting of a finite projective $\calO_K$-module $M$ and an $\calO_K$-linear endomorphism $\phi_M$ of $M$ satisfying 
   \[\lim_{n\to+\infty}\prod^n_{i=0}(\phi_M+i\pi E'(\pi)) = 0.\]
   The logarithmic prismatic cohomology of $\bM$ is computed as follows:
   \[\rR\Gamma((\bN\to\calO_K)_{\Prism},\bM)\cong [M\xrightarrow{\phi_M}M].\]
   The similar result holds for rational Hogde--Tate crystals by replacing finite projective $\calO_K$-modules by finite dimensional $K$-vector spaces in the above statement.
 \end{thm}
 \begin{proof}
   By Lemma \ref{Lem-cover abs} (1), giving a Hodge--Tate crystal $\bM$ on $(\bN\to\calO_K)_{\Prism}$ amount to giving a finite projective $\calO_K$-module $M$ together with a stratification $\varepsilon$, i.e. an $\frakS^1/E$-linear isomorphism 
   \[
   \varepsilon: M\otimes_{\calO_K,p_0}\frakS^1/E\to M\otimes_{\calO_K,p_1}\frakS^1/E
   \]
   satisfying the cocycle condition with respect to the cosimplicial ring $\frakS^{\bullet}/(E)$.
   Comparing Lemma \ref{pd poly-abs} and Lemma \ref{Lem-cech-to-derived-abs} with \cite[Lemma 2.6]{MW-b} and \cite[Lemma 3.10]{MW-b}, respectively, and using \cite[Lemma 3.6]{MW-b} (for $\alpha = \pi E'(\pi)$), we see that the arguments for the proofs of \cite[Theorem 3.8]{MW-b} and \cite[Theorem 3.12]{MW-b} still work in the logarithmic case. Then the theorem follows.  
 \end{proof}
 \begin{rmk}\label{Rmk-stratification-abs}
   Let $\bM$ be a (rational) Hodge--Tate crystal on $(\bN\to\calO_K)_{\Prism}$ with associated pair $(M,\phi_M)$. Similar to \cite[Remark 3.9]{MW-b}, the stratification $\varepsilon$ on $M$ with respect to the cosimplicial ring $\frakS^{\bullet}/(E)$ is given by 
   \[\varepsilon = (1-\pi E'(\pi)X_1)^{\frac{-\phi_M}{\pi E'(\pi)}}:  M\{X_1\}^{\wedge}_{\pd}\to M\{X_1\}^{\wedge}_{\pd} \]
   via the canonical isomorphisms $M\otimes_{\calO_K,p_1}\calO_K\{X_1\}^{\wedge}_{\pd}\cong\bM(\frakS^1,(E),M_{\frakS^1},\delta_{\log})\cong M\otimes_{\calO_K,p_0}\calO_K\{X_1\}^{\wedge}_{\pd}$.
 \end{rmk}
 
 Now we define a restricted site $(\bN\to\calO_K)^{\prime}_{\Prism}$. Its underlying category is the full subcategory of $(\bN\to\calO_K)_{\Prism}$ spanned by the prisms admitting maps from the Breuil--Kisin log prism $(\frakS,(E),M_{\frakS},\delta_{\log})$ and the coverings are inherited from the site $(\bN\to\calO_K)_{\Prism}$.
 
 Using $(\bN\to\calO_K)^{\prime}_{\Prism}$ instead of $(\calO_K)_{\Prism}^{\prime}$ defined in \cite[Subsection 3.3]{MW-b}, we have the following result for prismatic crystals (which are defines as Hodge--Tate crystals by using $\calO_{\Prism}$ instead of $\overline \calO_{\Prism}$) on $(\bN\to\calO_K)_{\Prism}$:
 \begin{cor}
   For any prismatic crystal $\bM$, the logarithmic prismatic cohomology $\rR\Gamma((\bN\to\calO_K)_{\Prism},\bM)$ is concentrated in degrees in $[0,1]$.
 \end{cor}
 \begin{proof}
   Just modify the proof of \cite[Theorem 3.15]{MW-b}.
 \end{proof}
 
 Now, we want to study the restriction functor 
 \[\Vect((\bN\to\calO_K)_{\Prism}, \overline \calO_{\Prism}[\frac{1}{p}])\to \Vect((\bN\to\calO_K)_{\Prism}^{\perf}, \overline \calO_{\Prism}[\frac{1}{p}]).\]
 By virtues of Theorem \ref{Thm-HT vs Rep-abs} and Theorem \ref{Thm-prismatic cohomology-abs}, this functor can be viewed as a functor from the category of pairs $(V,\phi_V)$ as mentioned in Theorem \ref{Thm-prismatic cohomology-abs} to the category $\Rep_{G_K}(C)$ of $C$-representations of $G_K$.
 
 Note that Lemma \ref{Lem-cech perf abs} combined with \cite[Proposition 3.22]{MW-b} says that there is a natural isomorphism of cosimplicial rings
 \[\overline \calO_{\Prism}[\frac{1}{p}](\rA_{\inf}^{\bullet},(\xi),M_{\rA_{\inf}^{\bullet}},\delta_{\log})\cong \rC(G_K^{\bullet},C),\]
 where the latter is the cosimplicial ring of continuous functions from $G^{\bullet}_K$ to $C$. So for our purpose, we need to specify the function corresponding \[X_1\in\calO_K\{X_1\}^{\wedge}_{\pd}[\frac{1}{p}]\cong \overline \calO_{\Prism}[\frac{1}{p}](\frakS^1,(E),M_{\frakS^1},\delta_{\log})\]
 via the natural morphism 
 \[\overline \calO_{\Prism}[\frac{1}{p}](\frakS^1,(E),M_{\frakS^1},\delta_{\log})\to\overline \calO_{\Prism}[\frac{1}{p}](\rA_{\inf}^{1},(\xi),M_{\rA_{\inf}^{1}},\delta_{\log})\cong \rC(G_K,C).\]
 
 \begin{lem}\label{Lem-function-abs}
   For any $g\in G_K$, we have $X_1(g) = c(g)\lambda(1-\zeta_p)$,
   where $\lambda$ is the image of $\frac{\xi}{E([\pi^{\flat}])}$ in $\calO_C$ and $c(g)\in \Zp$ is determined by $g(\pi^{\flat}) = \pi^{\flat}\epsilon^{c(g)}$.
 \end{lem}
 \begin{proof}
   The proof is similar to that of \cite[Proposition 3.26]{MW-b}. Note that $X_1$ is the image of $\frac{1-u_1/u_0}{E(u_0)}$ modulo $E$ and that as functions in $\rC(G_K,\Ainf)$, $u_0(g) = [\pi^{\flat}]$ and $u_1(g) = g([\pi^{\flat}]) = [\pi^{\flat}\epsilon^{c(g)}]$. Therefore, we have
   \[X_1(g) = \frac{\xi}{E([\pi^{\flat}])}(1-[\epsilon^{\frac{1}{p}}])\frac{[\epsilon^{c(g)}]-1}{[\epsilon]-1} \mod E.\]
   So the result follows.
 \end{proof}
 
 \begin{thm}\label{Thm-crys-vs-rep-abs}
   The restriction functor 
   \[\Vect((\bN\to\calO_K)_{\Prism}, \overline \calO_{\Prism}[\frac{1}{p}])\to \Vect((\bN\to\calO_K)_{\Prism}^{\perf}, \overline \calO_{\Prism}[\frac{1}{p}]) \cong \Rep_{G_K}(C)\]
   is fully faithful. More precisely, for a rational Hodge--Tate crystal $\bM$ with associated pair $(M,\phi_M)$, the resulting $C$-representation is given by $V = M\otimes_KC$, on which $g\in G_K$ acts via 
   \[U(g) = (1-c(g)\pi\lambda(1-\zeta_p)E'(\pi))^{-\frac{\phi_M}{\pi E'(\pi)}}.\]
   The Sen operator $\Theta_V$ of $V$ is $ -\frac{\phi_M}{\pi E'(\pi)}$ and there is natural quasi-isomorphism 
   \[\rR\Gamma((\bN\to\calO_K)_{\Prism},\bM)\cong \rR\Gamma(G_K,V).\]
 \end{thm}
 \begin{proof}
   Let $\bM$ be a rational Hodge--Tate crystal with associated pair $(M,\phi_M)$ and $C$-representation $V$.
   Combining Remark \ref{Rmk-stratification-abs} and Lemma \ref{Lem-function-abs}, we deduce that $V=M\otimes_KC$, on which the action of $g\in G_K$ on $V$ is given by $U(g)$ as desired. The proof of \cite[Theorem 4.3.3]{Gao} shows that $\Theta_V = -\frac{\phi_M}{\pi E'(\pi)}$. Now using the same argument in the proof of \cite[Proposition 3.31]{MW-b}, we get
   \[\rR\Gamma((\bN\to\calO_K)_{\Prism},\bM)\cong \rR\Gamma(G_K,V).\]
   In particular, we have 
   \[\rH^0((\bN\to\calO_K)_{\Prism},\bM)\cong \rH^0(G_K,V).\]
   Since the equivalence in Theorem \ref{Thm-prismatic cohomology-abs} preserves tensor products and dualities, the full faithfulness follows directly. The proof is complete.
 \end{proof}
 
 Finally, we want to compare (rational) Hodge--Tate crystals on usual absolute prismatic site $(\calO_K)_{\Prism}$ and logarithmic one $(\bN\to \calO_K)_{\Prism}$.
 
 Note that the forgetful functor $(A,I,M,\delta_{\log})\mapsto (A,I)$ from $(\bN\to \calO_K)_{\Prism}$ to $(\calO_K)_{\Prism}$ induces a natural functor of the categories of Hodge--Tate crystals on corresponding sites
 \[\Vect((\calO_K)_{\Prism},\overline \calO_{\Prism})\to\Vect((\bN\to\calO_K)_{\Prism},\overline \calO_{\Prism}).\]
 By virtues of Theorem \ref{MW-b} and Theorem \ref{Thm-prismatic cohomology-abs}, we get a commutative diagram of categories
 \[
 \xymatrix@C=0.45cm{
   \Vect((\calO_K)_{\Prism},\overline \calO_{\Prism})\ar[d]^{\simeq}\ar[r]&\Vect((\bN\to\calO_K)_{\Prism},\overline \calO_{\Prism})\ar[d]^{\simeq}\\
   \rH\rT(\calO_K)\ar[r]&\rH\rT^{\log}(\calO_K),
 }
 \]
 where the vertical equivalences are induced by evaluating at Breuil--Kisin prism $(\frakS,(E))$ and Breuil--Kisin log prism $(\frakS,(E),M_{\frakS},\delta_{\log})$, respectively. Let $\bM$ be a Hodge--Tate crystal in $\Vect((\calO_K),\overline \calO_{\Prism})$ with associated pairs $(M,\phi_M)$ and $(M,\phi_M^{\log})$ in $\rH\rT(\calO_K)$ and $\rH\rT^{\log}(\calO_K)$, respectively. Comparing Lemma \ref{pd poly-abs} with \cite[Lemma 2.6]{MW-b}, we see that $\phi_M^{\log} = \pi\phi_M$. 
 So we obtain the following corollary:
 \begin{cor}\label{Cor-fully f-abs}
   The functor $\Vect((\calO_K)_{\Prism},\overline \calO_{\Prism})\to\Vect((\bN\to\calO_K)_{\Prism},\overline \calO_{\Prism})$ constructed above is fully faithful. More precisely, the functor fits into the following commutative diagram:
   \[
 \xymatrix@C=0.45cm{
   \Vect((\calO_K)_{\Prism},\overline \calO_{\Prism})\ar[d]^{\simeq}\ar[r]&\Vect((\bN\to\calO_K)_{\Prism},\overline \calO_{\Prism})\ar[d]^{\simeq}\\
   \rH\rT(\calO_K)\ar[r]&\rH\rT^{\log}(\calO_K),
 }
 \]
 where the bottom functor sends a pair $(M,\phi_M)$ in $\rH\rT(\calO_K)$ to $(M,\pi\phi_M)$ in $\rH\rT^{\log}(\calO_K)$. The result also holds for rational Hodge--Tate crystals.
 \end{cor}
 One can also compare cohomologies in this case. It is easy to deduce from Theorem \ref{MW-b} and Theorem \ref{Thm-prismatic cohomology-abs} that for any Hodge--Tate crystal $\bM\in\Vect((\calO_K)_{\Prism},\overline \calO_{\Prism})$, we get a quasi-isomorphism
 \[\RGamma((\calO_K),\bM[\frac{1}{p}])\xrightarrow{\simeq}\RGamma((\bN\to\calO_K)_{\Prism},\bM[\frac{1}{p}]).\]

 \subsection{The geometric case}
 
 Now let $\frakX$ be a semi-stable $p$-adic formal scheme over $\calO_K$ with the log structure $\calM_{\frakX} = \calO_X^{\times}\cap\calO_{\frakX}$ as in the beginning of this section. We will investigate both $\Vect(((\frakX,\calM_{\frakX})/(\frakS,E,M_{\frakS},\delta_{\log}))_{\Prism},\overline \calO_{\Prism})$ and $\Vect((\frakX,\calM_{\frakX})_{\Prism},\overline \calO_{\Prism})$. Since we will not change the log structure, we just write $\Vect((\frakX/(\frakS,(E))_{\Prism,\log},\overline\calO_{\Prism})$ and $\Vect((\frakX)_{\Prism,\log},\overline\calO_{\Prism})$ respectively for simplicity. Similar notations are also used for sites of perfect log prisms.
 
 \subsubsection{The relative case}
  This part will be a generalization of \cite{Tian} to the semi-stable case. We first assume $\frakX = \Spf(R)$ is small affine with $\calM_{\frakX}$ induced by the prelog structure as given in Example \ref{Exam-rel case}.
 
 \begin{lem}\label{Lem-cover rel}
   Keep notations as in Example \ref{Exam-rel case}.
   
       $(1)$ The log prism $(\frakS(R),(E),M_{\frakS(R)},\delta_{\log})$ is a cover of the final object in the topos $\Sh((R)_{\Prism,\log})$.
       
       $(2)$ The perfect log prism $(\Ainf(\widehat R_{\infty}),I,M_{\Ainf(\widehat R_{\infty})},\delta_{\log})$ is a cover of the final object in the topos $\Sh((R)_{\Prism,\log}^{\perf})$.
 \end{lem}
 \begin{proof}
       (1) For a log prism $(A,I,M,\delta_{\log})$ in $(R)_{\Prism,\log}$, we need to show there exists a cover $(C,IC,N,\delta_{\log})$ of $(A,I,M,\delta_{\log})$ which admits a morphism from $(\frakS(R),(E),M_{\frakS(R)},\delta_{\log})$. 
       
       As $R$ is \'etale over $\calO_K\za T_0,\dots, T_r,T_{r+1}^{\pm 1},\dots, T_d^{\pm 1}\ya/\za T_0\cdots T_r-\pi\ya$, we may assume \[R=\calO_K\za T_0,\dots, T_r,T_{r+1}^{\pm 1},\dots, T_d^{\pm 1}\ya/\za T_0\cdots T_r-\pi\ya.\]
       By Lemma \ref{Structure lemma}, there are liftings $t_i$'s of $T_i$'s in $A$ such that the log structure $M$ is associated to the prelog structure $\bN\xrightarrow{e_i\mapsto t_i,\forall 0\leq i\leq r}A$.
       Put $B = (A [u,T_0,\cdots,T_r,T_{r+1}^{\pm 1},\dots, T_d^{\pm 1}]/(T_0\cdots T_r-u))[(\frac{t_0}{T_0})^{\pm 1},\dots,(\frac{t_r}{T_r})^{\pm 1}]$. We claim $B$ is $(p,I)$-completely faithfully flat over $A$ and that $(\frac{t_0}{T_0}-1,\dots,\frac{t_r}{T_r}-1)$ forms a $(p,I)$-completely regular sequence relative to $A$. For simplicity, we assume $r=d=1$. The general case follows from the same argument. Note that we have $B=A[u,T_0,T_1,Z_0^{\pm 1},Z_1^{\pm 1}]/(T_0T_1-u,Z_0T_0-t_0,Z_1T_1-t_1)=A[Z_0^{\pm 1},Z_1^{\pm 1}]$. Now the claim easily follows.
       
       Write $J=(I,\frac{t_0}{T_0}-1,\dots,\frac{t_r}{T_r}-1)\subset B$. 
       Then by \cite[Proposition 3.13]{BS-a}, if we put $C=B\{\frac{J}{I}\}^{\wedge}_{\delta}$, then $(C,IC)$ is a flat cover of $(A,I)$. By virtues of Lemma \ref{Lem-prolog-to-log}, if we equip $C$ with the log structure $(N\to C)$ induced by $M\to A\to C$, then $(C, IC, N, \delta_{\rm log})$ is a flat cover of $(A,I,M,\delta_{\log})$.
       Note that $E(u)\equiv E(t_0\cdots t_r)\equiv 0 \mod I$ in $C$. We have $E(u)\in IC$ and hence $E(u)C = IC$. In particular, $C$ is also $(p,E(u))$-complete and $\frac{t_i}{T_i}$'s are invertible in $C$. So $T_i$'s can be viewed as elements in $N$. This gives us a morphism from $(\frakS(R),(E),M_{\frakS(R)},\delta_{\log})$ to $(C, IC, N, \delta_{\rm log})$. We are done.
       
       $(2)$ By Proposition \ref{Equal perfect site}, it is enough to show $(\Ainf(\widehat R_{\infty}),(\xi))$ is a cover of the final object in the topos $\Sh((R)_{\Prism}^{\perf})$. For any perfect prism $(A,I)\in (R)^{\perf}_{\Prism}$, the $p$-completed tensor product $A/I\widehat \otimes_R\widehat R_{\infty}$ is a $p$-completely faithfully flat quasi-syntomic algebra over $A/I$. By \cite[Proposition 7.11 (2)]{BS-a}, there is a cover of perfect prism $(A,I)\to (B,IB)$ with $B/IB$ an $\widehat R_{\infty}$-algebra. By deformation theory, there is a morphism of prisms $(\Ainf(\widehat R_{\infty}),(\xi))\to (B,I)$. Now the result follows. 
 \end{proof}

 As a corollary, we also have similar results in the relative prismatic site.
 \begin{cor}\label{Cor-relative cover}
       $(1)$ The log prism $(\frakS(R),(E),M_{\frakS(R)},\delta_{\log})$ is a cover of the final object in the topos $\Sh((R)/(\frakS,E))_{\Prism,\log})$.
       
       $(2)$ The perfect log prism $(\Ainf(\widehat R_{\infty}),I,M_{\Ainf(\widehat R_{\infty})},\delta_{\log})$ is a cover of the final object in the topos $\Sh((R)/(\frakS,E))_{\Prism,\log}^{\perf})$.
 \end{cor}

 
 Now we study the cosimplicial log prisms $(\frakS(R)_{\rm rel}^{\bullet},(E),M_{\frakS(R)_{\rm rel}^{\bullet}},\delta_{\log})$ associated with the log prism $(\frakS(R),(E),M_{\frakS(R)},\delta_{\log})$ in $((R)/(\frakS,(E)))_{\Prism,\log}$. 
 
 Explicitly, write $B_{\rm rel}^n=\frakS(R)^{\otimes_{\frakS}n}[[(1-\frac{T_{i,1}}{T_{i,1}}),\cdots,(1-\frac{T_{i,0}}{T_{i,n}})]]_{0\leq i\leq r}$, where $\frakS(R)^{\otimes_{\frakS}n}$ means the tensor product of $n+1$-copies of $\frakS(R)$ over $\frakS$. Then we have
 \[
 \frakS(R)^n_{\rm rel}=B_{\rm rel}^n\{\frac{1-\frac{\underline{T_1}}{\underline{T_0}}}{E(u)}\cdots\frac{1-\frac{\underline{T_n}}{\underline{T_0}}}{E(u)}\}^{\wedge}_{\delta}
 \]
 where $\frac{1-\frac{\underline{T_i}}{\underline{T_0}}}{E(u)}$ with $1\leq i\leq n$ means adding all $\frac{1-\frac{T_{j,i}}{T_{j,0}}}{E(u)}$ for $1\leq j\leq d$. Here, we write $T_{j,i}$ for the corresponding element in the $i$-th component in $B_{\rm rel}^n$. Note that the element $\frac{1-\frac{T_{0,i}}{T_{0,0}}}{E(u)}$ automatically belongs to $\frakS(R)_{\rm rel}^n$ by the relation $u=\prod_{0\leq j\leq r} T_{j,i}$. Now, one can check that for any $1\leq m\leq n$ and any $0\leq i\leq r$, $\frac{T_{i,0}}{T_{i,m}}$ is a unit in $\frakS(R)^n_{\rm rel}$. So the log structure induced by $\bN^{r+1}\to \frakS(R)^n_{\rm rel}$ sending $e_i$ to $T_{i,m}$ for any $0\leq i\leq r$ is independent of the choice of $m$ and will be denoted by $M_{\rm rel}^n$.
 
 \begin{lem}\label{Lem-cech nerve rel-I}
   The cosimplicial log prism $(\frakS(R)_{\rm rel}^{\bullet},(E(u)),M_{\rm rel}^{\bullet},\delta_{\log})$ is the \v Cech nerve of $(\frakS(R),(E(u)),M_{\frakS(R)},\delta_{\log})$ in $(R/(\frakS,E))_{\Prism,\log}$.
 \end{lem}
 \begin{proof}
   The proof is similar to that of Lemma \ref{Lem-cech nerve abs}. We only show $(\frakS(R)_{\rm rel}^{1},(E(u)),M_{\rm rel}^{1},\delta_{\log})$ is the self coproduct of $(\frakS(R),(E(u)),M_{\frakS(R)},\delta_{\log})$ and the general case follows from a similar argument. Since $B_{\rm rel}^1$ is faithfully flat over $\frakS(R)$, we can check that 
   \[p,E(u), 1-\frac{T_{1,1}}{T_{1,0}},\dots,1-\frac{T_{r,1}}{T_{r,0}}, T_{r+1,1}-T_{r+1,0}, \dots, T_{d,1}-T_{d,0}\]
   forms a regular sequence in $B_{\rm rel}^1$. So \cite[Proposition 3.13]{BS-a} implies that $(\frakS(R)^1_{\rm rel},(E))$ is a faithfully flat cover of $(\frakS(R),(E))$ and hence a bounded prism such that $p,E$ forms a regular sequence in $\frakS(R)^1_{\rm rel}$.
   
   Consider the monoid $\bN^{r+1}\oplus_{\bN}\bN^{r+1}\cong(\oplus_{i=0}^r\bN e_{i,0})\oplus_{\bN \underline 1}(\oplus_{i=0}^r\bN e_{i,1})$, which is the push-out of the diagram
   \[\oplus_{i=0}^r\bN e_{i,0}\xleftarrow{\underline 1\mapsto e_{0,0}+\cdots+e_{r,0}}{\bN \underline 1}\xrightarrow{\underline 1\mapsto e_{0,1}+\cdots+e_{r,1}}\oplus_{i=0}^r\bN e_{i,1}.\]
   There is an obvious way to make $(\frakS(R)^{\widehat \otimes_{\frakS}2},(E(u)),\bN^{r+1}\oplus_{\bN}\bN^{r+1},\delta_{\log})$ a prelog prism over $(\frakS,(E(u)),\bN,\delta_{\log})$. There exists a ``multiplication'' map of prelog rings
   \[(\frakS(R)^{\otimes_{\frakS}2},\bN^{r+1}\oplus_{\bN}\bN^{r+1})\to(R,\bN^{r+1})\]
   such that the induced surjection on monoids $\pr:\bN^{r+1}\oplus_{\bN}\bN^{r+1}\to\bN^{r+1}$ is given by $\pr(e_{i,1})=\pr(e_{i,2})=e_i$. Let $N$ be the submonoid of $\bZ^{r+1}\oplus_{\bZ}\bZ^{r+1}$ generated by $\bN^{r+1}\oplus_{\bN}\bN^{r+1}$ and $(\pr^{\gp})^{-1}(0)$. Then $N = (\oplus_{i=0}^r\bN e_{i,0}\oplus_{\bN\underline 1}\oplus_{i=0}^r\bN e_{i,1})+\sum_{i=1}^r\bZ(e_{i,0}-e_{i,1})$. By \cite[Construction 2.18]{Kos}, the exactification of $(\frakS(R)^{\widehat \otimes_{\frakS}2},(E(u)),\bN^{r+1}\oplus_{\bN}\bN^{r+1},\delta_{\log})$ is exactly $(\frakS(R)^1_{\rm rel,0}:= (\frakS(R)^{\otimes_{\frakS}2}[(\frac{T_{0,1}}{T_{0,0}})^{\pm 1},\dots,(\frac{T_{r,1}}{T_{r,0}})^{\pm 1}])^{\wedge},(E(u)),\bN^{r+1},\delta_{\log})$, where the completion is with respect to the $(p,E(u))$-adic topology and the prelog structure $\bN^{r+1}\cong \oplus_{i=0}^r\bN e_i\to\frakS(R)^1_{\rm rel,0}$ is given by the map sending $e_i$'s to $T_{i,0}$'s. By construction, for any log prism $(A,I,M,\delta_{\log})\in(R/(\frakS,(E)))_{\Prism,\log}$, to which there are two morphisms from $(\frakS(R),(E(u)),\bN^{r+1},\delta_{\log})$, there exists a unique morphism of log prisms 
   \[f:(\frakS(R)^1_{\rm rel,0},(E(u)),(\bN^{r+1})^a,\delta_{\log})\to(A,I,M,\delta_{\log}).\]
   Since $1-\frac{T_{1,1}}{T_{1,0}},\dots,1-\frac{T_{r,1}}{T_{r,0}}, T_{r+1,1}-T_{r+1,0}, \dots, T_{d,1}-T_{d,0}$ is a $(p,E)$-adically regular sequence relative to $\frakS(R)$, by \cite[Proposition]{BS-a}, $\tilde \frakS(R)^1_{\rm rel,0}=\frakS(R)^1_{\rm rel,0}\{\frac{1-T_{1,1}/T_{1,0}}{E(u)},\dots,\frac{1-T_{r,1}/T_{r,0}}{E(u)},\frac{T_{r+1,1}-T_{r+1,0}}{E(u)},\dots,\frac{T_{d,1}-T_{d,0}}{E(u)}\}^{\wedge}_{\delta}$ exists and is $(p,E(u))$-completely flat over $\frakS(R)$. In particular, $p,E(u)$ is a regular sequence in $\tilde \frakS(R)^1_{\rm rel,0}$ and hence so is $u,E(u)$. By construction, $(\tilde \frakS(R)^1_{\rm rel,0},(E(u)),(\bN^{r+1})^a,\delta_{\log})$ belongs to $(R/(\frakS,(E)))_{\Prism,\log}$ and $f$ factors over the natural map 
   \[(\frakS(R)^1_{\rm rel,0},(E(u)),(\bN^{r+1})^a,\delta_{\log})\to(\tilde \frakS(R)^1_{\rm rel,0},(E(u)),(\bN^{r+1})^a,\delta_{\log})\]
   uniquely. In other words, $(\tilde \frakS(R)^1_{\rm rel,0},(E(u)),(\bN^{r+1})^a,\delta_{\log})$ is the self coproduct of $(\frakS(R),(E(u)),M_{\frakS(R)},\delta_{\log})$ in $(R/(\frakS,(E)))_{\Prism,\log}$.
   
   It remains to show that $(\tilde \frakS(R)^1_{\rm rel,0},(E(u)),(\bN^{r+1})^a,\delta_{\log})\cong (\frakS(R)^1_{\rm rel},(E(u)),M_{\rm rel}^1,\delta_{\log})$. The universal property for the former provides a unique morphism from $(\tilde \frakS(R)^1_{\rm rel,0},(E(u)),(\bN^{r+1})^a,\delta_{\log})$ to $(\frakS(R)^1_{\rm rel},(E(u)),M_{\rm rel}^1,\delta_{\log})$. We need to construct its inverse. By definition of $\frakS(R)^1_{\rm rel}$, it is enough to show that $1-\frac{T_{i,0}}{T_{i,1}}$ is topologically nilpotent and that $\frac{1-T_{i,1}/T_{i,0}}{E(u)}$ exists in $\tilde \frakS(R)^1_{\rm rel,0}$ for any $i$. By $E(u)$-adic completeness of $\tilde \frakS(R)^1_{\rm rel,0}$, we only need to show that $\frac{1-T_{i,1}/T_{i,0}}{E(u)}$ exist for any $i$. Note that both $\frac{T_{i,1}-T_{i,0}}{E(u)}$ and $\frac{T_{i,1}-T_{i,0}}{T_{i,0}}$ exists in $\tilde \frakS(R)^1_{\rm rel,0}$. Since $T_{i,0}$ and $E(u)$ meet transversally in $\tilde \frakS(R)^1_{\rm rel,0}$ for $1\leq i\leq r$ (as $u$ and $E(u)$ do), so $\frac{1-T_{i,1}/T_{i,0}}{E(u)}$ exists for any $1\leq i\leq r$. For $r+1\leq i\leq d$, the existence is trivial because $T_{i,0}$ is invertible in this case.
 \end{proof}
 
 \begin{lem}
   For any $1\leq i\leq n$ and $1\leq j\leq d$, let $Y_{j,i}$ denote the images of $\frac{1-\frac{T_{j,i}}{T_{j,0}}}{E(u)}$ in $\frakS(R)_{\rm rel}^n/E(u_0)$ respectively. Then we have 
   \[
   \frakS(R)_{\rm rel}^n/E(u)\cong R\{\underline{Y_1},\cdots,\underline{Y_n}\}^{\wedge}_{\pd}
   \]
   where the right hand side is the $p$-adic completion of the free pd-algebra over $R$ with variables $\{\underline{Y_1},\cdots,\underline{Y_n}\}$. Moreover, let $p_i:\frakS(R)_{\rm rel}^n/(E)\to\frakS(R)_{\rm rel}^{n+1}/E$ be the structure morphism induced by the order-preserving injection
   \[\{0,\dots,n\}\to\{0,\dots,i-1,i+1,\dots,n+1\},\]
   then we have 
   \begin{equation}\label{Equ-structure morphism on variable}
   \begin{split}
      &p_i(\underline Y_j) = \left\{
      \begin{array}{rcl}
        \underline Y_{j+1}- \underline Y_1, & i=0\\
         \underline Y_j, & 1\leq j<i\\
         \underline Y_{j+1}, & 0<i\leq j
      \end{array}
      \right.
  \end{split}.
  \end{equation}
 \end{lem}
 
 \begin{proof}
 This can essentially reduce to the proofs of \cite[Lemma 2.6 and 2.7]{MW-b} (note that all $T_{j,i}$'s are of rank 1).
 \end{proof}
 
 \begin{dfn}
    A log Higgs module over $R$ is a finite projective module $M$ over $R$ together with an $R$-linear morphism
    \[
    \theta: M\to M\otimes_{R}\widehat \Omega^1_{R/\calO_K,\log}\{-1\}
    \]
    such that $\theta\wedge\theta=0$. If we write $\theta=\sum_{i=1}^d\theta_i\otimes \dlog T_i\{-1\}$ with $\theta_i:M\to M$, then this condition is equivalent to saying $\theta_i\theta_j=\theta_j\theta_i$. 
    
    For any $n$-tuple $\underline m=(m_1,\cdots,m_n)\in \bN^n$, we put
    \[
    \theta^{\underline m}=\prod_{i=1}^n\theta_i^{m_i}\in {\rm End}_R(M).
    \]
    We say $\theta$ is topologically nilpotent if $\theta^{\underline m}$ tends to $0$ as $|\underline m|:=\sum_{i=1}^nm_i$ tends to infinity.
    
    Let $\HIG^{\log}(R)$ denote the category of topologically nilpotent log Higgs modules over $R$.
 \end{dfn}

 \begin{thm}
   There is an equivalence between the categories
   \[
   \Vect((R)/(\frakS,(E))_{\Prism,\log},\overline\calO_{\Prism})\simeq \HIG^{\log}(R).
   \]
 \end{thm}
 \begin{proof}
 Given a Hodge--Tate crystal $\bM$, we have the corresponding stratification
 \[
 \epsilon: M\otimes_{R,p_1}R\{\underline Y\}\to M\otimes_{R,p_0}R\{\underline Y\}.
 \]
 Note that $p_0$ and $p_1$ are both the natural inclusion $R\hookrightarrow R\{\underline Y\}$. Without loss of generality, we may assume $M$ is finite free over $R$ with a basis $\{e_1,\cdots,e_l\}$. Then we can write $\epsilon(\underline e)=\underline e\sum_{I\in \bN^d}A_I\underline Y^{[I]}$ with $A_I\in M_l(R)$. Then one can show (by putting all $X$ to be $0$ in the calculations after Lemma \ref{pd poly-abs}) that the cocycle condition is equivalent to the following conditions:

     (1) $A_{I}=id$ when $I=(0,\cdots,0)$;
     
     (2) if we write $\theta_{\epsilon,i}=A_{I}$ with $i$-th component in $I$ being 1 and other components being 0, then $\theta_{\epsilon,i}\theta_{\epsilon,j}=\theta_{\epsilon,j}\theta_{\epsilon,i}$;
     
     (3) $A_I=\prod_{i=1}^d \theta_{\epsilon,i}^{m_i}$ with $I=(m_1,\cdots,m_d)$ satisfying $\lim_{|I|\to+\infty}A_I=0$.

 So we can define a topologically nilpotent log Higgs module as $\theta=\sum_{i=1}^d\theta_{\epsilon,i}\otimes \dlog T_i\{-1\}.$ Conversely, given a Higgs module $(M,\theta)\in \HIG^{\log}(R)$ with $\theta=\sum_{i=1}^d\theta_{i}\otimes \dlog T_i\{-1\}$, we can get a stratification by setting $A_I=\prod_{i=1}^d \theta_{i}^{m_i}$ with $I=(m_1,\cdots,m_d)$.
 \end{proof}

Now we compare the cohomology of a Hodge--Tate crystal and its associated Higgs complex. 
 \begin{thm}
  Let $\bM\in\Vect((R)/(\frakS,(E)))_{\Prism,\log},\overline \calO_{\Prism})$ be a Hodge--Tate crystal with associated topological quasi-nilpotent log Higgs module $(M,\theta_M)$. There is a quasi-isomorphim 
   \[\RGamma_{\Prism}(\bM)\simeq\HIG(M,\theta_M).\]
\end{thm}

\begin{proof}
The proof is the same as the proof of \cite[Theorem 4.12]{Tian} (because the results in \cite{BJ} still hold true in the semi-stable case by considering log-crystalline cohomology and log-de Rham cohomology).
\end{proof}

In particular, one can get the following corollary directly.

\begin{thm}\label{Thm-relative coho dim}
  Let $\mathfrak X$ be a semi-stable $p$-adic formal scheme over $\calO_K$ of relative dimension $d$, with the log structure $\calM_{\frakX}=\calO_{X}^{\times}\cap \calO_{\frakX}$ where $X$ is the generic fiber. Let $\nu:\Sh((\frakX/(\frakS,(E)))_{\Prism,\log})\to\Sh(\frakX_{\et})$ be the natural morphism of topoi\footnote{See \cite[Remark 4.4]{Kos}.}. Then for any Hodge--Tate crystal $\bM$ in $\Vect((\frakX)/(\frakS,(E))_{\Prism,\log},\overline\calO_{\Prism})$, $\rR\nu_*\bM$ is a perfect complex of $\calO_{\frakX}$-modules with tor-amplitude $[0,d]$. If moreover $\frakX$ is proper, then $R\Gamma_{\Prism}(\bM)$ is a perfect complex of $\calO_K$-modules with tor-amplitude $[0,2d]$.
\end{thm}
\begin{rmk}
  Similar results as stated in Theorem \ref{Thm-relative coho dim} about the cohomological finiteness for Hodge--Tate crystals also hold for prismatic crystals (i.e. crystals on $(\frakX/(\frakS,(E)))_{\Prism,\log}$ with coefficients in $\calO_{\Prism}$). More precisely, if $\frakX$ is proper semistable over $\calO_K$ of relative dimension $d$, then for any prismatic crystal $\bM$ on $(\frakX/(\frakS,(E)))_{\Prism,\log}$, $\RGamma_{\Prism}(\bM)$ is a perfect complex with tor-amplitude $[0,2d]$. This can be deduced from Theorem \ref{Thm-relative coho dim} together with derived Nakayama's lemma directly.
\end{rmk}

 \subsubsection{The absolute case}
 
 Now we study the cosimplicial log prisms $(\frakS(R)^{\bullet},(E),M_{\frakS(R)^{\bullet}},\delta_{\log})$ corresponding to the \v Cech nerve of the log prism $(\frakS(R),(E),M_{\frakS(R)},\delta_{\log})$ in $(R)_{\Prism,\log}$.
 
 Explicitly, write $B^n=\frakS(R)^{\otimes n}[[(1-\frac{u_0}{u_1}),\cdots,(1-\frac{u_0}{u_n}),(1-\frac{T_{i,0}}{T_{i,1}}),\cdots,(1-\frac{T_{i,0}}{T_{i,n}})]]_{0\leq i\leq r}$, where $\frakS(R)^{\otimes n}$ means the tensor product of $n+1$-copies of $\frakS(R)$ over $W(k)$. Then we have

 \[
 \frakS(R)^n=B^n\{\frac{1-\frac{u_1}{u_0}}{E(u_0)},\cdots,\frac{1-\frac{u_n}{u_0}}{E(u_0)},\frac{1-\frac{\underline{T_1}}{\underline{T_0}}}{E(u_0)}\cdots\frac{1-\frac{\underline{T_n}}{\underline{T_0}}}{E(u_0)}\}^{\wedge}_{\delta}
 \]
 where $\frac{1-\frac{\underline{T_i}}{\underline{T_0}}}{E(u_0)}$ with $1\leq i\leq n$ means adding all $\frac{1-\frac{T_{j,i}}{T_{j,0}}}{E(u_0)}$ for $1\leq j\leq d$. Here, we write $T_{j,i}$ for the corresponding element in the $i$-th component in $B^n$. Note that the element $\frac{1-\frac{T_{0,i}}{T_{0,0}}}{E(u_0)}$ automatically belongs to $\frakS(R)^n$ by the relation $u_i=\prod_{0\leq j\leq r} T_{j,i}$. So the log structures induced by the map $\bN^{r+1}\to\frakS(R)^n$ sending $e_i$'s to $T_{i,m}$'s are independent of the choice of $0\leq m\leq n$. We denote this log structure by $M_{\frakS(R)^n}$. A similar argument for the proof of Lemma \ref{Lem-cech nerve rel-I} shows that $(\frakS(R)^{\bullet},(E(u_0)),M_{\frakS(R)^{\bullet}},\delta_{\log})$ is the \v Cech nerve of $(\frakS(R),(E(u_0)),M_{\frakS(R)},\delta_{\log})$ in $(R)_{\Prism,\log}$. Then the following Lemma can be deduced from the same argument as in the proof of Lemma \ref{pd poly-abs} and therefore we omit its proof here.
 
 \begin{lem}\label{pd-poly absolute}
   For any $1\leq i\leq n$ and $1\leq j\leq d$, let $X_i$, $Y_{j,i}$ denote the images of $\frac{1-\frac{u_i}{u_0}}{E(u_0)}$ and $\frac{1-\frac{T_{j,i}}{T_{j,0}}}{E(u_0)}$ in $\frakS(R)^n/E(u_0)$ respectively. Then we have 
   \[
   \frakS(R)^n/E(u_0)\cong R\{X_1,\cdots,X_n,\underline{Y_1},\cdots,\underline{Y_n}\}^{\wedge}_{\pd}
   \]
   where the right hand side is the $p$-adic completion of the free pd-algebra over $R$ with variables $\{X_1,\cdots,X_n,\underline{Y_1},\cdots,\underline{Y_n}\}$. Moreover, let $p_i:\frakS(R)^n/(E)\to\frakS(R)^{n+1}/E$ be the structure morphism induced by the order-preserving injection
   \[\{0,\dots,n\}\to\{0,\dots,i-1,i+1,\dots,n+1\},\]
   then we have 
   \begin{equation}\label{Equ-structure morphism on variables}
   \begin{split}
      &p_i(X_j) = \left\{
      \begin{array}{rcl}
           (X_{j+1}-X_1)(1-\pi E'(\pi)X_1)^{-1}, & i=0  \\
           X_j, & 1\leq j<i \\
           X_{j+1}, & 0<i\leq j;
      \end{array}
      \right.\\
      &p_i(\underline Y_j) = \left\{
      \begin{array}{rcl}
        (\underline Y_{j+1}- \underline Y_1)(1-\pi E'(\pi)X_1)^{-1}, & i=0\\
         \underline Y_j, & 1\leq j<i\\
         \underline Y_{j+1}, & 0<i\leq j.
      \end{array}
      \right.
  \end{split}
  \end{equation}
 \end{lem}

Next we want to give an equivalent description of the Hodge--Tate crystals on the absolute site $(R)_{\Prism,\log}$. By Lemma \ref{Lem-cover rel}, we see that giving a Hodge--Tate crystal $\bM$ of rank $l$ in $\Vect(( R)_{\Prism,\log},\overline \calO_{\Prism})$ is equivalent to giving a finite projective $R$-module $M$ of rank $l$ with a stratification 
\[
M\otimes_{R,p_1}R\{X,\underline{Y}\}^{\wedge}_{\pd}\to M\otimes_{R,p_0}R\{X,\underline{Y}\}^{\wedge}_{\pd}
\]
satisfying the cocycle condition with respect to the cosimplicial log prisms $(\frakS(R)^{\bullet},(E),M_{\frakS(R)^{\bullet}},\delta_{\log})$.

By choosing an affine covering of $R$, we may assume $M$ is a finite free $R$-module. Let $e_1,\dots,e_l$ be an $R$-basis of $M$, then 
 \begin{equation}\label{Equ-form of stratification}
     \varepsilon(\underline e) = \underline e\sum_{i\geq 0;I\in\bN^d}A_{i,I}X^{[i]}\underline Y^{[I]}
 \end{equation}
 for some $A_{i,I} \in \rM_l(R)$ satisfying $\lim_{i+|I|\to+\infty}A_{i,I} = 0$, where $X^{[i]}$ denotes the $i$-th pd--power ``$\frac{X^i}{i!}$'' of $X$ and for any $I=(i_1,\dots,i_d)\in\bN^d$, $\underline Y^{[I]}$ denotes $Y_1^{[i_1]}\cdots Y_d^{[i_d]}$. Then we have
 \begin{equation*}
 \begin{split}
      p_2^*(\varepsilon)\circ p_0^*(\varepsilon)(\underline e)
     & = p_2^*(\varepsilon)( \underline e\sum_{i\geq 0;I\in\bN^d}A_{i,I}(1-\pi E'(\pi)X_1)^{-i-|I|}(X_2-X_1)^{[i]}(\underline Y_2-\underline Y_1)^{[I]})\\
     & = \underline e\sum_{i,j\geq 0;I,J\in\bN^d}A_{j,J}A_{i,I}X_1^{[j]}(1-\pi E'(\pi)X_1)^{-i-|I|}(X_2-X_1)^{[i]}\underline Y_1^{[J]}(\underline Y_2-\underline Y_1)^{[I]}\\
     & = \underline e \sum_{i\geq 0,I\in\bN^d}P_{i,I}(X_1,Y_1)X_2^{[i]}\underline Y_2^{[I]},
 \end{split}
 \end{equation*}
 where for any $i\geq 0$ and $I\in\bN^d$,
 \[P_{i,I}(X_1,\underline Y_1) = \sum_{j,l\geq 0;J,L\in \bN^d}A_{j,J}A_{i+l,I+L}(-1)^{l+|L|}(1-\pi E'(\pi)X_1)^{-i-l-|I|-|L|}X_1^{[j]}X_1^{[l]}\underline Y_1^{[J]}\underline Y_1^{[L]}.\]
 On the other hand, we have
 \[p_1^*(\varepsilon)(\underline e) = \underline e\sum_{i\geq 0;I\in\bN^d}A_{i,I}X_2^{[i]}\underline Y_2^{[I]}.\]
 So the stratification $\varepsilon$ satisfies the cocycle condition if and only if $A_{0,\vec 0} = I$ and for any $i\geq 0$ and $I\in\bN^d$,
 \begin{equation}\label{Equ-compare coefficients}
      \sum_{j,l\geq 0;J,L\in \bN^d}A_{j,J} A_{i+l,I+L}(-1)^{l+|L|}(1-\pi E'(\pi)X_1)^{-i-l-|I|-|L|}X_1^{[j]}X_1^{[l]}\underline Y_1^{[J]}\underline Y_1^{[L]} = A_{i,I}.
 \end{equation}
 
 Let $\underline Y_1 = 0$ in (\ref{Equ-compare coefficients}), we get 
 \begin{equation}\label{Equ-Let Y be zero}
 \begin{split}
   A_{i,I} & = \sum_{j,l\geq 0}A_{j,\vec 0} A_{i+l,I}(-1)^{l}(1-\pi E'(\pi) X_1)^{-i-l-|I|}X_1^{[j]}X^{[l]}\\
   & = \sum_{j,l\geq 0}A_{j,\vec 0}A_{i+l,I}(-1)^l\sum_{k\geq 0}\binom{-i-l-|I|}{k}(-\pi E'(\pi) X)^kX_1^{[j]}X_1^{[l]}.
 \end{split}
 \end{equation}
 Comparing the coefficients of $X_1$, we get
 \[A_{1,\vec 0}A_{i,I}-A_{0,\vec 0}A_{i+1,I}+(i+|I|)\pi E'(\pi) A_{0,\vec 0}A_{i,I} = 0\]
which shows that for any $n\geq 0$,
 \begin{equation}\label{Equ-first iteration}
    A_{n+1,I} = (A_{1,0}+|I|\pi E'(\pi)+n\pi E'(\pi))A_{n,I} = \prod_{i=0}^n(A_{1,0}+|I|\pi E'(\pi)+i\pi E'(\pi))A_{0,I}.
 \end{equation}
 
 Let $X_1 = 0$ in (\ref{Equ-compare coefficients}), we get
 \begin{equation}\label{Equ-Let X be zero}
 \begin{split}
   A_{i,I} & = \sum_{J,L\in\bN^d}A_{0,J} A_{i,I+L}(-1)^{|L|}\underline Y_1^{[J]}\underline Y_1^{[L]}.
 \end{split}
 \end{equation}
 Comparing the coefficient of $Y_{k,1}$, we get
 \[A_{0,E_k}A_{i,I}-A_{0,\vec 0}A_{i,I+E_k}=0\]
which shows that for any $n\geq 0$,
 \begin{equation}\label{Equ-second iteration}
    A_{n,I} = A_{0,E_k}A_{n,I-E_k} = A_{0,E_k}^{i_k}A_{n,I-i_kE_k},
 \end{equation}
 where $E_k$ denotes the generator of the $k$-th component of $\bN^d$.
 
 Combining (\ref{Equ-first iteration}) with (\ref{Equ-second iteration}), we deduce that for any $1\leq i,j\leq d$,
 \begin{equation}\label{Equ-commutativity}
 \begin{split}
     & [A_{0,E_i},A_{0,E_j}] = A_{0,E_i}A_{0,E_j}-A_{0,E_j}A_{0,E_i} = 0,\\
     & [A_{0,E_i},A_{1,0}] = A_{0,E_i}A_{1,0}-A_{1,0}A_{0,E_i} =\pi E'(\pi)A_{0,E_i},
 \end{split}
 \end{equation}
 and that for any $n\geq 0$ and $I = (i_1,\dots,i_d)\in\bN^d$,
 \begin{equation}\label{Equ-determine coefficients}
 \begin{split}
     A_{n,I} & = \prod_{i=0}^{n-1}(A_{1,0}+|I|\pi E'(\pi)+i\pi E'(\pi))A_{0,E_1}^{i_1}\cdots A_{0,E_d}^{i_d}\\
     & = A_{0,E_1}^{i_1}\cdots A_{0,E_d}^{i_d}\prod_{i=0}^{n-1}(i\pi E'(\pi)+A_{1,0}).
 \end{split}
 \end{equation}

\begin{prop}\label{matrix of stratification}
   Let $M$ be a finite free $R$-module $l$ with a basis $e_1,\dots,e_l$ and \[\varepsilon:M\otimes_{R,p_1}R\{X,\underline Y\}^{\wedge}_{\pd}\to M\otimes_{R,p_0}R\{X,\underline Y\}^{\wedge}_{\pd}\]
   be a stratification on $M$ with respect to $\frakS(R)^{\bullet}$ such that
   \[\varepsilon(\underline e) = \underline e\sum_{n\geq 0;I\in\bN}A_{n,I}X^{[n]}\underline Y^{[I]}\]
   for $A_{n,I}$'s in $\rM_l(R)$. Then the following are equivalent:
   
       $(1)$ $(M,\varepsilon)$ is induced by a Hodge--Tate crystal $\bM\in \Vect((R)_{\Prism,\log},\overline \calO_{\Prism})$.
       
       $(2)$ There are matrices $A,\Theta_1,\dots,\Theta_d\in\rM_l(R)\cong\End_R(M)$ satisfying the following conditions:
       
           \qquad $(a)$ For any $1\leq i,j\leq d$, $[\Theta_i,\Theta_j]=0$;
           
           \qquad $(b)$ For any $1\leq i\leq j$, $[\Theta_i,A]=\pi E'(\pi)\Theta_i$;
           
           \qquad $(c)$ $\lim_{n\to+\infty}\prod_{i=0}^{n-1}(i\pi E'(\pi)+A) = 0$.
           
       In this case, $\Theta_i$'s are nilpotent and for any $n\geq 0$, $I=(i_1,\dots,i_d)\in\bN^d$,
       \begin{equation*}
       \begin{split}
           A_{n,I}=\Theta_1^{i_1}\cdots\Theta_d^{i_d}\prod^{n-1}_{i=0}(i\pi E'(\pi)+A).
       \end{split}
       \end{equation*}
 \end{prop}
 \begin{proof}
 The proof is the same as that of \cite[Proposition 2.6]{MW-c}. In particular, we put $\alpha=\pi E'(\pi)$ in \cite[Lemma 2.7]{MW-c}.
 \end{proof}

 \begin{dfn}
    By an enhanced log Higgs module over $R$, we mean a triple $(M,\theta_M, \phi)$ where
    
    $(1)$ $M$ is a finite projective module over $R$;
    
    $(2)$ $\theta_M$ is an $R$-linear homomorphism
    \[
    \theta_M: M\to M\otimes_R\widehat \Omega^1_{R/\calO_K,\log}\{-1\}
    \]
   such that $\theta_M\wedge \theta_M=0$. Let $\HIG(M,\theta_M)$ denote the complex
          \[M\xrightarrow{\theta_M} M\otimes_{R}\widehat \Omega^1_{R/\calO_K,\log}\{-1\}\xrightarrow{\theta_M} M\otimes_{R}\widehat \Omega^2_{R/\calO_K,\log}\{-2\}\cdots;\]
   
   $(3)$ $\phi:M\to M$ is an $R$-linear endomorphism of $M$ such that 
       \[\lim_{n\to+\infty}\prod_{i=0}^n(\phi+i\pi E'(\pi)) = 0\] 
       for the $p$-adic topology on $\End(M)$ and that $\{\phi+i\pi E'(\pi)\id\}_{i\geq 0}$ induces an endomorphism of $\HIG(M,\theta_M)$; that is, the following diagram 
       \begin{equation}\label{Equ-Higgs field with operator}
           \xymatrix@C=0.45cm{
             M\ar[d]^{\phi}\ar[r]^{\theta_M\qquad\quad}& M\otimes_{R} \widehat\Omega^1_{R,\log}\{-1\} \ar[d]^{\phi+\pi E'(\pi)\id}\ar[r]^{\theta_M} &  M\otimes_{R} \widehat\Omega^2_{R,\log}\{-2\} \ar[d]^{\phi+2\pi E'(\pi)\id}\ar[r]^{\qquad\quad\theta_M}&\cdots\ar[r]^{\theta_M\qquad}&M\otimes_{R} \widehat\Omega^d_{R,\log}\{-d\}\ar[d]^{\phi+d\pi E'(\pi)\id}\\
             M\ar[r]^{\theta_M\qquad\quad}& M\otimes_{R} \widehat\Omega^1_{R,\log}\{-1\} \ar[r]^{\theta_M} &  M\otimes_{R} \widehat\Omega^2_{R,\log}\{-2\} \ar[r]^{\qquad\quad\theta_M}&\cdots\ar[r]^{\theta_M\qquad}&M\otimes_{R} \widehat\Omega^d_{R,\log}\{-d\}
           }
    \end{equation}
       is commutative. 
       
    Denote by $\HIG(M,\theta_M,\phi)$ the total complex of (\ref{Equ-Higgs field with operator}) and by $\HIG^{\log}_*(R)$ the category of enhanced log Higgs modules over $R$. 
 \end{dfn}
 Now, the following theorem follows from Proposition \ref{matrix of stratification} directly.
 \begin{thm}\label{Thm-HT crystal as enhanced Higgs field}
   The evaluation at $(\frakS(R),(E), M_{\frakS(R)},\delta_{\log})$ induces an equivalence from the category $\Vect((R)_{\Prism,\log},\overline \calO_{\Prism})$ of Hodge--Tate crystals to the category $\HIG^{\log}_*(R)$ of enhanced log Higgs modules. A similar result holds for rational Hodge--Tate crystals after replacing $\HIG^{\log}_*(R)$ by $\HIG^{\log}_*(R[\frac{1}{p}])$.
 \end{thm}
 
We can also compare the prismatic cohomology of a Hodge--Tate crystal and the Higgs complex associated with the corresponding Higgs modules.

\begin{thm}\label{Thm-local Prismatic cohomology}
  Let $\bM\in\Vect((R)_{\Prism,\log},\overline \calO_{\Prism})$ be a Hodge--Tate crystal with associated enhanced log Higgs module $(M,\theta_M,\phi_M)$. Then there is a quasi-isomorphism 
   \[\RGamma_{\Prism}(\bM)\simeq\HIG(M,\theta_M,\phi_M).\]
   A similar result holds for rational Hodge--Tate crystals.
\end{thm}

\begin{proof}
The proof is the same as that of \cite[Theorem 2.12]{MW-b}, except that we need to put $\alpha=\pi E'(\pi)$ and replace modules of continuous differentials by modules of continuous log differentials.
\end{proof}

A direct corollary is the following theorem.
\begin{thm}
  Let $\frakX$ be a semi-stable $p$-adic formal scheme over $\calO_K$ of relative dimension $d$, with the log structure $\calM_{\frakX}=\calO_{X}^{\times}\cap \calO_{\frakX}$ where $X$ is the generic fiber. Let $\nu:\Sh((\frakX,\calM_{\frakX})_{\Prism})\to\Sh(\frakX_{\et})$ be the natural morphism of topoi. Then for any Hodge--Tate crystal $\bM\in\Vect((\frakX,\calM_{\frakX})_{\Prism},\overline\calO_{\Prism})$, $\rR\nu_*\bM$ is a perfect complex of $\calO_{\frakX}$-modules with tor-amplitude $[0,d+1]$. If moreover $\frakX$ is proper, then $R\Gamma_{\Prism}(\bM)$ is a perfect complex of $\calO_K$-modules with tor-amplitude $[0,2d+1]$. Similar results also hold for rational Hodge--Tate crystals after replacing $\calO_{\frakX}$-modules and $\calO_K$-modules by $\calO_{\frakX}[\frac{1}{p}]$-modules and $K$-vector spaces, respectively.
\end{thm}

 If $\frakX = \Spf(R)$ is small affine and smooth over $\calO_K$ (i.e. we take $r=0$ in Example \ref{Exam-rel case}), then the log structure $M_{\frakX}$ is induced by the composition $(\bN\xrightarrow{1\to\pi}\calO_K\to R)$. In this case, we still have a natural functor
 \[\Vect((R)_{\Prism},\overline \calO_{\Prism})\to\Vect((\bN\to R)_{\Prism},\overline \calO_{\Prism}),\]
 which is induced by restriction. By Theorem \ref{MW-b} and Theorem \ref{Thm-HT crystal as enhanced Higgs field}, this gives rise to a functor
 \[\HIG_*^{\nil}(R)\to\HIG_*^{\log}(R).\]
 By comparing Lemma \ref{pd-poly absolute} with \cite[Lemma 2.2]{MW-c}, we see that the functor sends an enhanced Higgs module $(M,\theta_M,\phi_M)$ to an enhanced log Higgs module $(M,\theta_M,\pi\phi_M)$. So we get the following result:
 \begin{cor}\label{Cor-fully f-rel-local}
   The restriction functor $\Vect((R)_{\Prism},\overline \calO_{\Prism})\to\Vect((\bN\to R)_{\Prism},\overline \calO_{\Prism})$ is fully faithful and fits into the following commutative diagram:
   \[
   \xymatrix@C=0.45cm{
   \Vect((R)_{\Prism},\overline \calO_{\Prism})\ar[d]^{\simeq}\ar[r]&\Vect((\bN\to R)_{\Prism},\overline \calO_{\Prism})\ar[d]\\
   \HIG_*^{\nil}(R)\ar[r]&\HIG_*^{\log}(R),
   }
   \]
   where the bottom functor sends $(M,\theta_M,\phi_M)$ to $(M,\theta_M,\pi\phi_M)$.
   Similar results hold for rational Hodge--Tate crystals and in this case, for any $\bM\in\Vect((R)_{\Prism},\overline \calO_{\Prism}[\frac{1}{p}])$, we get a quasi-isomorphism
   \[\RGamma((R)_{\Prism},\bM)\cong\RGamma((\bN\to R)_{\Prism},\bM).\]
 \end{cor}
 \begin{proof}
   It remains to prove the desired quasi-isomorphism in the rational case. By virtues of Theorem \ref{MW-b} and Theorem \ref{Thm-local Prismatic cohomology}, we only need to check
   \[\HIG(M,\theta_M,\phi_M)\simeq \HIG(M,\theta_M,\pi\phi_M),\]
   where $(M,\theta_M,\phi_M)$ is the enhanced Higgs module over $R[\frac{1}{p}]$ corresponding to the given rational Hodge--Tate crystal $\bM\in \Vect((R)_{\Prism},\overline \calO_{\Prism}[\frac{1}{p}])$. Now the result is clear as $\pi$ is invertible.
 \end{proof}

\subsection{Hodge--Tate crystals as generalised representations}
We still assume $\frakX = \Spf(R)$ is small affine with $\calM_{\frakX}$ induced by the prelog structure given in Example \ref{Exam-rel case}. For simplicity, we write $R^{\Box}:=\calO_K\za T_0,\dots, T_r,T_{r+1}^{\pm 1},\dots, T_d^{\pm 1}\ya/(T_0\cdots T_r-\pi)$ and there is a fixed framing, i.e. an \'etale map $\Box: R^{\Box}\to R$. 

 Let $R_{C}^{\Box}$ denote the base change $R^{\Box}\widehat\otimes_{\calO_K}\calO_{C}=\calO_{C}\za T_0,\dots, T_r,T_{r+1}^{\pm 1},\dots, T_d^{\pm 1}\ya/( T_0\cdots T_r-\pi)$. We fix a compatible family of pirmitive $p$-power roots $\{\pi^{\frac{1}{p^m}}\}_{m\geq 0}$ of $\pi$ as in Notations. For each $m\geq 0$, we then consider the $R_{C}^{\Box}$-algebra 
\[
R_{C,m}^{\Box}:=\calO_{C}\za T_0^{\frac{1}{p^m}},\dots, T_r^{\frac{1}{p^m}},T_{r+1}^{\pm {\frac{1}{p^m}}},\dots, T_d^{\pm {\frac{1}{p^m}}}\ya/( T_0^{\frac{1}{p^m}}\cdots T_r^{\frac{1}{p^m}}-\pi^{\frac{1}{p^m}})
\]
and let $\hat R_{\infty}^{\Box}:=(\varinjlim R_{C,m}^{\Box})^{\wedge}$, which is a perfectoid ring over $\calO_{C}$. 

Let $R_{C}:=R\widehat\otimes_{R^{\Box}}R_{C}^{\Box}$ and $\hat R_{\infty}:=R_{C}\widehat\otimes_{R_{C}^{\Box}}\hat R_{\infty}^{\Box}$. We consider the adic generic fibers $X_{C}=\Spa(R_{C}[\frac{1}{p}],R_{C})$ of $\frakX_{{C}}=\Spf(R_{C})$ and put $X_{\infty}=\Spa(\hat R_{\infty}[\frac{1}{p}],\hat R_{\infty})$.

Then $X_{\infty}$ is a Galois cover of both $X_{C}$ and $X$ with the Galois groups $\Gamma_{\geo}$ and $\Gamma$, respectively. For any $1\leq i,j\leq d$ and $n\geq 0$, let $\gamma_i\in\Gamma_{\geo}$ such that $\gamma_i(T_j^{\frac{1}{p^n}}) = \zeta_{p^n}^{\delta_{ij}}T_j^{\frac{1}{p^n}}$ and that $\gamma_i(T_0^{\frac{1}{p^n}}) = \zeta_{p^n}^{-\epsilon}T_0^{\frac{1}{p^n}}$ with $\epsilon=1$ for $1\leq i\leq r$ and $\epsilon = 0$ for $r+1\leq i\leq d$, then 
   \[\Gamma_{\geo}\cong \Zp\gamma_1\oplus\cdots\oplus\Zp\gamma_d.\footnote{It is more natural to write $\Gamma_{\geo}$ as the subgroup of $\oplus_{i=0}^d\Zp\widetilde \gamma_i\to\Zp$ such that $\widetilde \gamma = \prod_{i=0}^d\widetilde \gamma_i^{n_i}\in\Gamma_{\geo}$ if and only if $n_0+\cdots+n_r=0$, where $\widetilde \gamma_i(T_j^{\frac{1}{p^n}}) = \zeta_{p^n}^{\delta_{ij}}T_j^{\frac{1}{p^n}}$ for any $n\geq 0$ and $0\leq i,j\leq d$. Clearly, if we put $\gamma_i = \widetilde \gamma_0^{-1}\widetilde \gamma_i$ for $1\leq i\leq r$ and $\gamma_i=\widetilde \gamma_i$ for $r+1\leq i\leq d$, then $\Gamma_{\geo}\cong\oplus_{i=1}^d\Zp\gamma_i$.}\]
   The group $\Gamma_{\geo}$ is a normal subgroup of $\Gamma$ fitting into the following exact sequence
   \begin{equation}\label{Equ-exact sequence for Gal grp}
       0\to\Gamma_{\geo}\to\Gamma\to G_K\to 1,
   \end{equation}
   where $G_K = \Gal(\bar K/K)$ is the absolute Galois group of $K$. Moreover, we have $\Gamma\cong \Gamma_{\geo}\rtimes G_K$ such that for any $g\in G_K$ and $1\leq i\leq d$,
   \[g\gamma_i g^{-1} = \gamma_i^{\chi(g)},\]
   where $\chi:G_K\to\bZ_p^{\times}$ denotes the $p$-adic cyclotomic character.
   
   In this subsection, we want to relate Hodge--Tate crystals on the absolute logarithmic prismatic site of $(\frakX,\calM_{\frakX})$ to the generalised representations on $X_{\proet}$, i.e vector bundles with coefficients in $\hat\calO_X$. In the local case, the latter is equivalent to the category $\Rep_{\Gamma}(\hat R_{\infty}[\frac{1}{p}])$, i.e. the category of finite projective $\hat R_{\infty}[\frac{1}{p}])$-modules with semi-linear $\Gamma$-actions.
   
   Consider the log prism $(\Ainf(\hat R_{\infty}),(\xi),M_{\log},\delta_{\log})$ as in Example \ref{Exam-rel case}. For simplicity, we write $A_{\infty}:=\Ainf(\hat R_{\infty})$. Let $(A_{\infty}^{(\bullet)},(\xi),M_{\log},\delta_{\log})$ denote the cosimplicial perfect log prisms associated with $(\Ainf(\hat R_{\infty}),(\xi),M_{\log},\delta_{\log})$ in $(\bN^{r+1}\to R)_{\Prism}^{\perf}$. Ignoring the log structures, we let $(A_{\infty}^{[\bullet]},(\xi)$ denote the cosimplicial perfect prisms associated with $(\Ainf(\hat R_{\infty}),(\xi))$ in $(R)_{\Prism}^{\perf}$. Then we have the following lemma, which is a direct consequence of Lemma \ref{Cor-Equal perfect site}.
   
   \begin{lem}\label{Lem-compare perfection}
     For any $n\geq 0$, we have $A_{\infty}^{(n)}\cong A_{\infty}^{[n]}$.
   \end{lem}

    
   
   By virtues of the above lemma, we can get the following theorem.
   
   \begin{thm}\label{perfect site crystal}
   Evaluating on the perfect log prism $(\Ainf(\hat R_{\infty}),(\xi),M_{\log},\delta_{\log})$ induces an equivalence between the catgeory $\Vect((R)_{\Prism,\log}^{\perf},\overline \calO_{\Prism}[\frac{1}{p}])$ and the category $\Rep_{\Gamma}(\hat R_{\infty}[\frac{1}{p}])$.
   \end{thm}
 
 \begin{proof}
  This follows directly from the above lemma and \cite[Theorem 2.23]{MW-c}.
 \end{proof}
 
 
 Now we can compose the following functors
 \[
 \Vect((R)_{\Prism,\log},\overline\calO_{\Prism})\to \Vect((R)^{\perf}_{\Prism,\log},\overline\calO_{\Prism})\to \Rep_{\Gamma}(\hat R_{\infty}[\frac{1}{p}]).
 \]
    For any Hodge--Tate crystal $\bM\in \Vect((R)_{\Prism,\log},\overline \calO_{\Prism})$, we denote the associated representation of $\Gamma$ by $V(\bM)$. Then we want to describe $V(\bM)$ in an explicit way.

 Let $A_{\inf}(R)$ be the lifting of $R_{C}$ over $A_{\inf}$ determined by the framing $\Box$. Then $(A_{\inf}(R),(\xi),M_{\log},\delta_{\log})$ with the log structure induced by the log structure on $\frakS(R)$ is a log prism in $(R)_{\Prism,\log}$ by setting $\delta(T_i) = 0$ for any $1\leq i\leq d$. A similar argument used in the proof of \cite[Lemma 2.11]{MW-a} shows that $(A_{\infty},(\xi),M_{\log},\delta_{\log})$ is exactly the perfection of $(A_{\inf}(R),(\xi)),M_{\log},\delta_{\log})$. So there is a morphism of cosimplicial rings
   \[\frakS(R)^{\bullet}\to A_{\infty}^{(\bullet)}.\]
   As a consequence, we get a natural morphism
   \begin{equation}\label{Equ-variable to function}
     \frakS(R)^{\bullet}/(E)\xrightarrow{\cong} R\{X_1,\dots,X_{\bullet},\underline Y_1,\dots,\underline Y_{\bullet}\}^{\wedge}_{\pd}\to\rC(\Gamma^{\bullet},\widehat R_{\infty}[\frac{1}{p}])
   \end{equation}
   of cosimplicial rings.
   \begin{lem}\label{variable as function}
     Regard $X_1, Y_{1,1}, \dots, Y_{d,1}$ as functions from $\Gamma$ to $\widehat R_{\infty}[\frac{1}{p}]$. For any $g\in G_K$ and $n_1,\dots,n_d\in\Zp$, if we put $\sigma = \gamma_1^{n_1}\cdots\gamma_d^{n_d}g$, then
     \begin{equation}\label{Equ-variable as function}
     \begin{split}
         & X_1(\sigma) = c(g)\lambda(1-\zeta_p),\\
         & Y_{i,1}(\sigma) = n_i\lambda(1-\zeta_p),~ \forall 1\leq i\leq d,
     \end{split}
     \end{equation}
     where $\lambda$ is the image of $\frac{\xi}{E([\pi^{\flat}])}$ in $C$ and $c(g)\in\Zp$ is determined by $g(\pi^{\flat}) = \epsilon^{c(g)}\pi^{\flat}$.
   \end{lem}
   \begin{proof}
     Recall that for any $1\leq i\leq d$, $Y_{i,1}$ is the image of $\frac{1-\frac{T_{i,1}}{T_{i,0}}}{E(u_0)}$ modulo $E$ while $X_1$ is the image of $\frac{1-\frac{u_1}{u_0}}{E(u_0)}$. As elements in $\rC(\Gamma,\rW(\widehat R_{\infty}^{\flat}))$, $u_0$ and $T_{i,0}$ are constant functions while $u_1$ and $T_{i,1}$ are evaluation functions. So we see that 
     \begin{equation*}
         \begin{split}
             X_1(\sigma) & \equiv \frac{1-\frac{g([\pi^{\flat}])}{[\pi^{\flat}]}}{E([\pi^{\flat}])}\mod (E)\\
             & \equiv \frac{(1-[\epsilon]^{c(g)})}{1-[\epsilon]}\frac{\xi}{E(\pi^{\flat})}(1-[\epsilon]^{\frac{1}{p}})\mod (E)\\
             &=c(g)\lambda(1-\zeta_p).
         \end{split}
     \end{equation*}
     Similarly, we conclude that for any $1\leq i\leq d$,
     \begin{equation*}
         \begin{split}
             Y_{i,1}(\sigma) & \equiv \frac{[T_i^{\flat}]-\sigma([T_i^{\flat}])}{E([\pi^{\flat}])[T_i^{\flat}]}\mod (E)\\
             & \equiv \frac{(1-[\epsilon]^{n_i})}{1-[\epsilon]}\frac{\xi}{E(\pi^{\flat})}(1-[\epsilon]^{\frac{1}{p}})\mod (E)\\
             &=n_i\lambda(1-\zeta_p).
         \end{split}
     \end{equation*}
     These complete the proof.
   \end{proof}

 \begin{thm}\label{Thm-cocycle in explicit way}
     Let $\bM\in\Vect((R)_{\Prism,\log},\overline \calO_{\Prism})$ be a Hodge--Tate crystal with the associated enhanced Higgs module $(M,\theta_M,\phi_M)$. Let $\Theta_i$'s be the matrices of $\theta_M$ as defined in the proof of Theorem \ref{Thm-HT crystal as enhanced Higgs field}. Then the cocycle in $\rH^1(\Gamma,\GL(M[\frac{1}{p}]))$ corresponding to $V(\bM)$ is given by 
     \begin{equation}\label{Equ-cocycle}
         U(\sigma) = \exp(\lambda(1-\zeta_p)\sum_{i=1}^dn_i\Theta_i)(1-c(g)\lambda(1-\zeta_p)\pi E'(\pi))^{-\frac{\phi_M}{\pi E'(\pi)}},
     \end{equation}
     for any $\sigma = \gamma_1^{n_1}\cdots\gamma_d^{n_d}g\in\Gamma$, where $\lambda$ and $c(g)$ are defined in Lemma \ref{variable as function}.
   \end{thm}
   \begin{proof}
     By the formulae of stratification given in Proposition \ref{matrix of stratification}, the cocycle is given by
     \[U(\sigma) = \sum_{i_1,\dots,i_d,n\geq 0}\Theta_1^{i_1}\cdots\Theta_d^{i_d}\prod^{n-1}_{i=0}(i\pi E'(\pi)+A)Y_1(\sigma)^{[i_1]}\cdots Y_d(\sigma)^{[i_d]}X(\sigma)^{[n]}.\]
     Now the result follows from Lemma \ref{variable as function} directly.
   \end{proof}

    As a consequnece of the above theorem, we see that the functor $\bM\mapsto V(\bM)$ factors through 
 \[
 \Vect((R)_{\Prism,\log},\overline\calO_{\Prism})\to \Rep_{\Gamma}(R_{C})\to \Rep_{\Gamma}(R_{C}[\frac{1}{p}])\to \Rep_{\Gamma}(\hat R_{\infty}[\frac{1}{p}]).
 \]
 From now on, the story is essentially the same as that in \cite{MW-c}, only with minor changes.
 
 \begin{prop}\label{Prop-0}
The functor $\Vect((R)_{\Prism,\log},\overline \calO_{\Prism})\to \Rep_{\Gamma}(R_{C})$ is fully faithful.
 \end{prop}
 
 \begin{proof}
   The proof is the same as that of \cite[Proposition 2.28]{MW-c}.
 \end{proof}
 
 Since $\Spa(R[\frac{1}{p}],R)$ is smooth over $K$, the similar arguments for the proof of \cite[Proposition 2.29]{MW-c} imply the following result.
\begin{prop}\label{Prop-fully faithful-loc-II}
     There is an equivalence from the category $\Rep_{\Gamma}(R_{C}[\frac{1}{p}])$ to the category $\Rep_{\Gamma}(\widehat R_{\infty}[\frac{1}{p}])$ induced by the base change
     \[V\mapsto V_{\infty}:= V\otimes_{R_{C}[\frac{1}{p}]}\widehat R_{\infty}[\frac{1}{p}],\]
     which induces a quasi-isomorphism
     \[R\Gamma(\Gamma,V)\cong R\Gamma(\Gamma,V_{\infty}).\]
  \end{prop}
As in the good reduction case, we will give another equivalent description of this proposition. For any $n\geq 0$, let $K_n = K(\zeta_{p^n},\pi^{\frac{1}{p^n}})$. Put
\[R_n^{\Box}=\calO_{K_n}\za T_0^{\frac{1}{p^n}},\cdots,T_r^{\frac{1}{p^n}},T_{r+1}^{\pm\frac{1}{p^n}},\cdots,T_d^{\pm \frac{1}{p^n}}\ya/(T_0^{\frac{1}{p^n}}\cdots T_r^{\frac{1}{p^n}}-\pi^{\frac{1}{p^n}}).
\]
Then $R_n^{\Box}[\frac{1}{p}]=K_n\za T_1^{\pm\frac{1}{p^n}},\cdots,T_d^{\pm \frac{1}{p^n}}\ya$. 

Let $X_{K_n}=\Spa(R_{K_n}[\frac{1}{p}],R_{K_n})$ and $X_n=\Spa(R_n[\frac{1}{p}],R_n)$ denote the base changes of $X$ to  $\Spa(K_n,\calO_{K_n})$ and $\Spa(R_n^{\Box}[\frac{1}{p}],R_n^{\Box})$ respectively. 

Let $K_{\cyc,\infty}=\cup_n K_n$ and $\hat K_{\cyc,\infty}$ be its $p$-adic completion, $R_{\cyc,\infty}^\Box:=\cup_nR_n^\Box$ and $\hat R_{\cyc,\infty}^\Box$ be its $p$-adic completion. 
 
 Consider $X_{\hat K_{\cyc,\infty}} = \Spa(R_{\hat K_{\cyc,\infty}}[\frac{1}{p}],R_{\hat K_{\cyc,\infty}})$ and $X_{\cyc,\infty} = \Spa(\widehat R_{\cyc,\infty}[\frac{1}{p}],\widehat R_{\cyc,\infty})$ which are the base changes of $X$ to $\Spa(\hat K_{\cyc,\infty},\calO_{\hat K_{\cyc,\infty}})$ and $\Spa(\hat R_{\cyc,\infty}^\Box[\frac{1}{p}],\hat R_{\cyc,\infty}^\Box)$ respectively. Then we see that $X_{\cyc,\infty}$ is a Galois cover of $X$ with Galois group $\Gamma_{\cyc,\infty}$ and is also a Galois cover of $X_{\hat K_{\cyc,\infty}}$ whose Galois group can be identified with $\Gamma_{\geo}$. We still have a short exact sequence
  \[0\to\Gamma_{\geo}\to\Gamma_{\cyc,\infty}\to\widehat G_K\to 1,\]
  where $\widehat G_K=\Gal(K_{\cyc,\infty}/K)$. 
  Let $\Rep_{\Gamma_{\cyc}}(R_{\hat K_{\cyc,\infty}}[\frac{1}{p}])$ and $\Rep_{\Gamma_{\cyc,\infty}}(\widehat R_{\cyc,\infty}[\frac{1}{p}])$ be the category of (semi-linear) representations of $\Gamma_{\cyc,\infty}$ over $R_{\hat K_{\cyc,\infty}}[\frac{1}{p}]$ and $\widehat R_{\cyc,\infty}[\frac{1}{p}]$, respectively. Then by Faltings' almost purity theorem, the following functors
  \[\Rep_{\Gamma_{\cyc,\infty}}(R_{\hat K_{\cyc,\infty}}[\frac{1}{p}])\to \Rep_{\Gamma}(R_{\Cp}[\frac{1}{p}])\]
  and 
  \[\Rep_{\Gamma_{\cyc,\infty}}(\hat R_{\cyc,\infty}[\frac{1}{p}])\to \Rep_{\Gamma}(\hat R_{\infty}[\frac{1}{p}])\]
  are both equivalences induced by the corresponding base changes. So Proposition \ref{Prop-fully faithful-loc-II} can be reformulated as follows.
  \begin{prop}\label{Prop-fully faithful-loc-III}
     There is an equivalence from the category $\Rep_{\Gamma_{\cyc,\infty}}(R_{\hat K_{\cyc,\infty}}[\frac{1}{p}])$ to the category $\Rep_{\Gamma_{\cyc,\infty}}(\widehat R_{\cyc,\infty}[\frac{1}{p}])$ induced by the base change
     \[V\mapsto V_{\cyc,\infty}:= V\otimes_{R_{\hat K_{\cyc,\infty}}}\widehat R_{\cyc,\infty}.\] 
     Moreover, the natural morphism $V\to V_{\cyc,\infty}$ induces a quasi-isomorphism 
     \[R\Gamma(\Gamma_{\cyc,\infty},V)\simeq R\Gamma(\Gamma_{\cyc,\infty},V_{\cyc,\infty}).\]
  \end{prop}
 \begin{proof}
   The proof is similar to that of \cite[Proposition 2.31]{MW-b}. 
 \end{proof}
 
  Note that Proposition \ref{Prop-0} also holds after inverting $p$, so combining this with Proposition \ref{Prop-fully faithful-loc-II}, we get
  \begin{thm}\label{Thm-fully faithful-loc}
     The natural functor 
     \[\Vect((R)_{\Prism,\log},\overline \calO_{\Prism}[\frac{1}{p}])\to\Rep_{\Gamma}(\widehat R_{\infty}[\frac{1}{p}])\]
     is fully faithful.
  \end{thm}

  \subsection{Local $p$-adic Simpson correspondence} 
 \subsubsection{Local Simpson functor for generalised representations}\label{sssec-local Simpson}
 In this subsection, we shall assign to every representation $V_{\infty}\in\Rep_{\Gamma}(\widehat R_{\infty}[\frac{1}{p}])$ a $G_K$-Higgs module over $R_{C}[\frac{1}{p}]$and an arithmetic Higgs module over $R_{K_{cyc}}[\frac{1}{p}]$ in the sense of \cite[Definition 3.4, 3.6]{MW-c}.
 
 As in the good reduction case, the main ingredient for constructing the desired local Simpson functor is a period ring $S_{\cyc,\infty}$ over $\widehat R_{\cyc,\infty}[\frac{1}{p}]$ together with a universal $\widehat R_{\cyc,\infty}[\frac{1}{p}]$-linear $\Gamma_{\cyc,\infty}$-equivariant Higgs field 
 \[\Theta: S_{\cyc,\infty}\to S_{\cyc,\infty}\otimes_R\widehat \Omega_{R,\log}^1(-1).\]
This ring can be viewed as sections on $X_{\cyc,\infty}$ of a certain period sheaf $\calO\bC$ on the pro-\'etale site $X_{\proet}$ (c.f. the ${\rm gr}^0\calO\bB_{\rm dR}$ in \cite[Corollary 6.15]{Sch-b} or \cite[Remark 2.1]{LZ}). The $S_{\cyc,\infty}$ and $\Theta$ can be described in a more explicit way; namely, there exists an isomorphism of $\widehat R_{\cyc,\infty}[\frac{1}{p}]$-algebras
 \begin{equation}\label{Equ-isom for period ring-loc}
     S_{\cyc,\infty}\cong \widehat R_{\cyc,\infty}[\frac{1}{p}][Y_1,\dots,Y_d]
 \end{equation}
 where $\widehat R_{\cyc,\infty}[\frac{1}{p}][Y_1,\dots,Y_d]$ is the free algebra over $\widehat R_{\cyc,\infty}[\frac{1}{p}]$ on the set $\{Y_1,\dots,Y_d\}$. Via the isomorphism (\ref{Equ-isom for period ring-loc}), the Higgs field $\theta$ can be written as 
 \begin{equation}\label{Equ-Theta-loc}
     \Theta = \sum_{i=1}^d\frac{\partial}{\partial Y_i}\otimes\frac{\dlog(T_i)}{t},
 \end{equation}
 where $t$ denotes the $p$-adic analogue of $2\pi i$ (and one identifies Tate twist $\Zp(n)$ with $\Zp t^n$ for any $n\in \bZ$), and for any $\sigma =\gamma_1^{n_1}\cdots\gamma_d^{n_d}\in\Gamma_{\geo}$, $g\in G_K$ and any $1\leq j\leq d$,
 \begin{equation}\label{Equ-action on Y}
 \begin{split}
       g(Y_j) = \chi(g)^{-1}Y_j,\\
       \sigma(Y_j) = Y_j+n_j.
 \end{split} 
 \end{equation}
 
 \begin{rmk}
The local description of the period sheaf $\calO\bC$ is studied in \cite{Sch-b} for affinoid spaces admitting a standard \'etale map to the torus. Our case is slightly different. Locally, we have an \'etale map $\Spa(R[\frac{1}{p}],R)\to \Spa(R^\Box[\frac{1}{p}],R^\Box)$. But one can check that the proofs in \cite{Sch-b} still work in our case and so one can get descriptions of $\calO\bC$ on $X_{\cyc,\infty}$ as above. 
 \end{rmk}
 
 Then we have the following theorem.
 \begin{thm}[{\cite[Theorem 2.39]{MW-c}}]\label{Thm-local Simpson}
   Keep notations as before. Then for any $V_{\cyc,\infty}\in\Rep_{\Gamma_{\cyc,\infty}}(\widehat R_{\cyc,\infty}[\frac{1}{p}])$, the restriction $\theta_{H(V_{\cyc,\infty})}$ of 
   \[\Theta_{V_{\cyc,\infty}}:=\id_{V_{\cyc,\infty}}\otimes\Theta:V_{\cyc,\infty}\otimes_{\widehat R_{\cyc,\infty}}S_{\cyc,\infty}\to V_{\cyc,\infty}\otimes_{\widehat R_{\cyc,\infty}}S_{\cyc,\infty}\otimes_R\widehat \Omega^1_{R,\log}(-1)\]
   to $H(V_{\cyc,\infty}):=(V_{\cyc,\infty}\otimes_{\widehat R_{\cyc,\infty}}S_{\cyc,\infty})^{\Gamma_{\geo}}$ defines a nilpotent Higgs module $(H(V_{\cyc,\infty}),\theta_{H(V_{\cyc,\infty})})$ over $R_{\hat K_{\cyc,\infty}}[\frac{1}{p}]$ such that the following assertions are true:

       $(1)$ Define $\Theta_{H(V_{\cyc,\infty})}=\theta_{H(V_{\cyc,\infty})}\otimes\id_{S_{\cyc,\infty}}+\id_{H(V_{\cyc,\infty})}\otimes\Theta$. Then the natural map 
       \[H(V_{\cyc,\infty})\otimes_{R_{\hat K_{\cyc,\infty}}}S_{\cyc,\infty}\to V_{\cyc,\infty}\otimes_{\widehat R_{\cyc,\infty}}S_{\cyc,\infty}\]
       is a $\Gamma_{\cyc,\infty}$-equivariant isomorphism and identifies $\Theta_{H(V_{\cyc,\infty})}$ with $\Theta_{V_{\cyc,\infty}}$.
       
       $(2)$ Let $\HIG(H(V_{\cyc,\infty}),\theta_{H(V_{\cyc,\infty})})$ be the associated Higgs complex of $(H(V_{\cyc,\infty}),\theta_{H(V_{\cyc,\infty})})$. Then there is a $\widehat G_K$-equivariant quasi-isomorphism
       \[R\Gamma(\Gamma_{\geo},V_{\cyc,\infty})\simeq \HIG(H(V_{\cyc,\infty}),\theta_{H(V_{\cyc,\infty})}).\]
       As a consequence, we get a quasi-isomorphism
       \[R\Gamma(\Gamma_{\cyc,\infty},V_{\cyc,\infty})\simeq R\Gamma(\widehat G_K,\HIG(H(V_{\cyc}),\theta_{H(V_{\cyc})})).\]
       
       $(3)$ Let $\HIG^{\nil}_{\widehat G_K}(R_{\hat K_{\cyc,\infty}}[\frac{1}{p}])$ be the category of nilpotent Higgs modules $(H,\theta_H)$ over $R_{\hat K_{\cyc,\infty}}[\frac{1}{p}]$, which are endowed with continuous $\widehat G_K$-actions such that $\theta_H$'s are $\widehat G_K$-equivariant. Then the functor
       \[H:\Rep_{\Gamma_{\cyc,\infty}}(\widehat R_{\cyc,\infty}[\frac{1}{p}])\to\HIG_{\widehat G_K}^{\nil}(R_{\hat K_{\cyc,\infty}}[\frac{1}{p}])\]
       defines an equivalence of categories and preserves tensor products and dualities. More precisely, for a Higgs module $(H,\theta_H)$ over $R_{\hat K_{\cyc,\infty}}[\frac{1}{p}]$ with a $\widehat G_K$-action, define $\Theta_H = \theta_H\otimes\id+\id\otimes\Theta$. Then the quasi-inverse $V_{\cyc,\infty}$ of $H$ is given by 
       \[V_{\cyc,\infty}(H,\theta_H) = (H\otimes_{R_{\hat K_{\cyc,\infty}}}S_{\cyc,\infty})^{\Theta_H=0}.\]
 \end{thm}

   

  \begin{rmk}
   A similar local result in the log smooth case was also obtained in \cite[Theorem 15.1, 15.2]{Tsu}.
 \end{rmk}
 
 \begin{rmk}\label{Rmk-indenpendent of chart}
   The functor $H$ depends on the choice of toric chart for the moment. However, by using the period sheaf $\OC$ on $X_{\proet}$, we can consider the functor 
   \[H:\Vect(\Spa(R[\frac{1}{p}],R)_{\proet},\widehat \calO_X)\to\HIG^{\nil}_{G_K}(R_{\Cp}[\frac{1}{p}]),\]
   which is independent of the choice of charts. Therefore, the local constructions glue. This leads to a global version of Theorem \ref{Thm-local Simpson}, which appears in Theorem \ref{Thm-HT crystal as Higgs bundles-G} and \cite[Theorem 3.13]{MW-c}.
 \end{rmk}
  
  Similar to $\HIG^{\nil}_{\widehat G_K}(R_{\hat K_{\cyc,\infty}}[\frac{1}{p}])$, we define the category $\HIG^{\nil}_{\Gamma_K}(R_{\hat K_{\cyc}}[\frac{1}{p}])$ (resp. $\HIG^{\nil}_{\Gamma_K}(R_{K_{\cyc}}[\frac{1}{p}])$) consisting of pairs $(H,\theta_H)$, where $H$ is a representation of $\Gamma_K$ over $R_{\hat K_{\cyc}}[\frac{1}{p}]$ (resp. $R_{K_{\cyc}}[\frac{1}{p}]$), i.e. a finite projective module over $R_{\hat K_{\cyc}}[\frac{1}{p}]$ (resp. $R_{K_{\cyc}}[\frac{1}{p}]$) endowed with a continuous semi-linear $\Gamma_K$-action, and $\theta_H$ is a Higgs field on $H$ which is $\Gamma_K$-equivariant. Here, $R_{K_{\cyc}}:=R\otimes_{\calO_K}\calO_{K_{\cyc}}$ and $R_{\hat K_{\cyc}}$ is the $p$-adic completion of $R_{K_{cyc}}$. For our convenience, we denote $\Rep_{\Gamma_K}(R_{K_{\cyc}}[\frac{1}{p}])$ (resp. $\Rep_{\Gamma_K}(R_{\hat K_{\cyc}}[\frac{1}{p}])$) the category of representations of $\Gamma_K$ over $R_{K_{\cyc}}[\frac{1}{p}]$ (resp. $R_{\hat K_{\cyc}}[\frac{1}{p}]$).
  
 \begin{cor}\label{Cor-upgrade}
   The functor $H:\Rep_{\Gamma_{\cyc,\infty}}(\widehat R_{\cyc,\infty}[\frac{1}{p}])\to\HIG^{\nil}_{\widehat G_K}(R_{\hat K_{\cyc,\infty}}[\frac{1}{p}])$ upgrades to an equivalence, which is still denoted by $H$, from $\Rep_{\Gamma_{\cyc,\infty}}(\widehat R_{\cyc}[\frac{1}{p}])$ to $\HIG^{\nil}_{\Gamma_K}(R_{K_{\cyc}}[\frac{1}{p}])$.
 \end{cor}
 \begin{proof}
   By Faltings' almost purity theorem, we have a natural equivalence
   \[\Rep_{\widehat G_K}(R_{\hat K_{\cyc,\infty}}[\frac{1}{p}])\simeq \Rep_{\Gamma_K}(R_{\hat K_{\cyc}}[\frac{1}{p}]),\]
   which induces an equivalence between $\HIG^{\nil}_{\widehat G_K}(R_{\hat K_{\cyc,\infty}}[\frac{1}{p}])$ and $\HIG^{\nil}_{\Gamma_K}(R_{\hat K_{\cyc}}[\frac{1}{p}])$. Now the result follows from Proposition \ref{Prop-arith decompletion} (1).
 \end{proof}
 The other ingredient to construct the Simpson functor is the following decompletion result.
 \begin{prop}\label{Prop-arith decompletion}Keep notations as above.
 
       $(1)$ The base change induces an equivalence from $\Rep_{\Gamma_K}(R_{K_{\cyc}}[\frac{1}{p}])$ to $\Rep_{\Gamma_K}(R_{\hat K_{\cyc}}[\frac{1}{p}])$ such that that for any $V\in \Rep_{\Gamma_K}(R_{K_{\cyc}}[\frac{1}{p}])$, there is a canonical quasi-isomorphism
       \[R\Gamma(\Gamma_K,V)\simeq R\Gamma(\Gamma_K,V\otimes_{R_{K_{\cyc}}}R_{\hat K_{\cyc}}).\]
       
       $(2)$ For any $V\in\Rep_{\Gamma_K}(R_{K_{\cyc}}[\frac{1}{p}])$, there exists a unique element $\phi_V\in\End_{R_{K_{\cyc}}[\frac{1}{p}]}(V)$ such that for any $v\in V$, there exists some $n\gg 0$ satisfying that for any $g\in\Gamma_{K(\zeta_{p^n})}:=\Gal(K_{\cyc}/K(\zeta_{p^n}))$,
       \[g(v) = \exp(\log\chi(g)\phi_V)(v).\]
       Such a $\phi_V$ is called the {\bf arithmetic Sen operator} of $V$.
 \end{prop}
 \begin{proof}
   Note that that $(\{R_{K(\zeta_{p^n})}[\frac{1}{p}]\}_{n\geq 0},\{\Gamma_{K(\zeta_{p^n})}\}_{n\geq 0})$ forms a stably decompleting pair in the sense of \cite[Definition 4.4]{DLLZ}. So the result follows from the argument in the proof of \cite[Proposition 2.44]{MW-c}.
 \end{proof}
 
 Now we can state and prove the main result in this subsection.
 \begin{thm}\label{Thm-local Simpson for generealised repn}
   There exists a faithful functor 
   \[H:\Rep_{\Gamma}(\widehat R_{\infty}[\frac{1}{p}])\to\HIG^{\rm arith}(R_{K_{\cyc}}[\frac{1}{p}]),\]
   which preserves tensor products and dualities. Here, $\HIG^{\rm arith}(R_{K_{\cyc}}[\frac{1}{p}])$ is the category of arithmatic Higgs modules over $R_{K_{\cyc}}[\frac{1}{p}]$ in the sense of \cite[Definition 2.37]{MW-c}.
 \end{thm}
 \begin{proof}
   The result follows from a similar argument in the proof of \cite[Theorem 2.45]{MW-c}. The only difference is that we have to use Corollary \ref{Cor-upgrade} instead of \cite[Corollary 2.43]{MW-c} in loc.cit..
 \end{proof}
 \begin{rmk}\label{Rmk-Local Simpson for generalized repn}
  Replacing $K_{\cyc}$ be $C$, we obtain from Theorem \ref{Thm-local Simpson for generealised repn} a faithful functor
   \[H:\Rep_{\Gamma}(\widehat R_{\infty}[\frac{1}{p}])\to\HIG^{\rm arith}(R_{C}[\frac{1}{p}]),\]
   which factors through an equivalence
   \[\Rep_{\Gamma}(\widehat R_{\infty}[\frac{1}{p}])\to\HIG^{\rm nil}_{G_K}(R_{C}[\frac{1}{p}])\]
   where the categories $\HIG^{\rm nil}_{G_K}(R_{C}[\frac{1}{p}])$ and $\HIG^{\rm arith}(R_{C}[\frac{1}{p}])$ are defined in the obvious way. We also use this ``$C$-version'' of above theorem from now on.
 \end{rmk}
    We can make the following conjecture as in the good reduction case.
   \begin{conj}\label{Conj-loc-Sen-cohomology}
     Let $(H,\theta_H,\phi_H)$ be an arithmetic Higgs module over $R_{{C}}[\frac{1}{p}]$ coming from a representation $V\in\Rep_{\Gamma}(\widehat R_{\infty}[\frac{1}{p}])$. Then there is a quasi-isomorphism
     \[R\Gamma(\Gamma,V)\otimes_{R[\frac{1}{p}]}R_{C}[\frac{1}{p}]\simeq \HIG(H,\theta_H,\phi_H).\]
   \end{conj}

 \subsubsection{Local inverse Simpson functor for enhanced Higgs modules}
 Now we go ahead to construct an inverse Simpson functor defined on the category of enhanced log Higgs modules. The first result is the following theorem.
 \begin{thm}\label{Thm-local inverse simpson-integral}
   There is a fully faithful functor 
   \[V:\HIG^{\log}_*(R)\to \Rep_{\Gamma}(R_{C})\]
   from the category of enhance Higgs modules over $R$ to the category of representations of $\Gamma$ over $R_{\Cp}$.
 \end{thm}
 \begin{proof}
   Consider the composition of functors
   \[\HIG_*^{\log}(R)\xrightarrow{\simeq} \Vect((R)_{\Prism},\overline \calO_{\Prism})\to \Vect((R)_{\Prism}^{\perf},\overline \calO_{\Prism}[\frac{1}{p}])\xrightarrow{\simeq}\Rep_{\Gamma}(\widehat R_{\infty}[\frac{1}{p}]).\]
   Then by the explicit description of $\Gamma$-action given in Theorem \ref{Thm-cocycle in explicit way}, the above composition factors over $\Rep_{\Gamma}(R_{C})\to\Rep_{\Gamma}(\widehat R_{\infty}[\frac{1}{p}])$ and hence induces a functor
   \[V:\HIG_*^{\log}(R)\to\Rep_{\Gamma}(R_{C}).\]
   Then the result is a consequence of Theorem \ref{Thm-HT crystal as enhanced Higgs field} combined with Proposition \ref{Prop-0}.
 \end{proof}
 We also have the rational version of the above theorem.
 \begin{thm}\label{Thm-local inverse Simpson-rational}
   There is a fully faithful functor 
   \[V:\HIG^{\log}_*(R[\frac{1}{p}])\to \Rep_{\Gamma}(\widehat R_{\infty}[\frac{1}{p}])\]
   from the category of enhance log Higgs modules over $R[\frac{1}{p}]$ to the category of representations of $\Gamma$ over $\widehat R_{\infty}[\frac{1}{p}]$. Moreover, let $V(M):=V(M,\theta_M,\phi_M)$ be the representation associated with the enhance log Higgs module $(M,\theta_M,\phi_M)$. Then we have a quasi-isomorphism
   \[R\Gamma(\Gamma_{\geo},V(M))\simeq \HIG(M\otimes_{R}R_{C},\theta_M)\]
   which is compatible with $G_K$-actions.
 \end{thm}
 \begin{proof}
   The functor $V$ is given by the composition
   \[\HIG_*^{\nil}(R[\frac{1}{p}])\xrightarrow{\simeq} \Vect((R)_{\Prism},\overline \calO_{\Prism}[\frac{1}{p}])\to \Vect((R)_{\Prism}^{\perf},\overline \calO_{\Prism}[\frac{1}{p}])\xrightarrow{\simeq}\Rep_{\Gamma}(\widehat R_{\infty}[\frac{1}{p}]).\]
   The proof then is the same as that of \cite[Theorem 2.51]{MW-c}.
  
 \end{proof}
 \begin{rmk}
   Theorem \ref{Thm-local inverse simpson-integral} and Theorem \ref{Thm-local inverse Simpson-rational} all depend on the choice of the framing $\Box$ on $R$. In the rational case, we know that $\Rep_{\Gamma}(\widehat R_{\infty}[\frac{1}{p}])$ is equivalent to the category $\Vect(X,\OX)$ of generalised representations on $X$. We will see later that this point of view will provide us with the possibility of getting a global functor.
    \end{rmk}

 In Theorem \ref{Thm-local inverse Simpson-rational}, the full faithfulness of the functor $V$ amounts to the isomorphism
 \[\rH^0(\HIG(M,\theta_M,\phi_M))\simeq \rH^0(\Gamma,V(M)).\]
 More generally, we can also ask whether the following conjecture is true.
 \begin{conj}\label{Conj-loc-relative Sen}
   There is a quasi-isomorphism 
   \[R\Gamma(\Gamma,V(M,\theta_M,\phi_M))\simeq \HIG(M,\theta_M,\phi_M).\]
 \end{conj}

 \subsection{A typical example}\label{local functors}
  As in the case of good reduction, we end this section by constructing two faithful functors from the category $\HIG^{\log}_*(R[\frac{1}{p}])$ to the category of $\HIG^{\rm arith}(R_{C}[\frac{1}{p}])$.
 
   \begin{prop}\label{Prop-local comparison}
    $(1)$ The functor $\HIG^{\log}_*(R[\frac{1}{p}])\to\HIG^{\rm arith}(R_{C}[\frac{1}{p}])$
     sending $(M,\theta_M,\phi_M)$ to $(M\otimes_RR_{C}[\frac{1}{p}],(\zeta_p-1)\lambda\theta_M,-\frac{1}{\pi E'(\pi)}\phi_M)$ is faithful and induces quasi-isomorphisms
     \[\HIG(M,\theta_M)\otimes_RR_{C}[\frac{1}{p}]\simeq \HIG(M\otimes_RR_{\Cp}[\frac{1}{p}],(\zeta_p-1)\lambda\theta_M)\]
   and 
   \[\HIG(M,\theta_M,\phi_M)\otimes_RR_{C}[\frac{1}{p}]\simeq \HIG(M\otimes_RR_{C}[\frac{1}{p}],(\zeta_p-1)\lambda\theta_M,-\frac{1}{\pi E'(\pi)}\phi_M).\]
   
   $(2)$ The functor $\HIG^{\log}_*(R[\frac{1}{p}])\to \HIG^{\nil}_{G_K}(R_{C}[\frac{1}{p}])$ sending $(M,\theta_M,\phi_M)$ to $(M\otimes_RR_{C}[\frac{1}{p}],(\zeta_p-1)\lambda\theta_M)$ together with the $G_K$-action on $H(M)$ by 
   \[g\mapsto (1-c(g)\lambda(1-\zeta_p)\pi E'(\pi))^{-\frac{\phi_M}{\pi E'(\pi)}}\]
   is fully faithful. Its composition with the faithful functor $\HIG^{\nil}_{G_K}(R_{C}[\frac{1}{p}])\to\HIG^{\rm arith}(R_{C}[\frac{1}{p}])$ then gives the second faithful functor desired.
   \end{prop}
 
 We guess these two faithful functors are actually the same. More precisely, we make the following conjecture.
  \begin{conj}\label{Conj-local Sen comparison}
     For a Hodge--Tate crystal $\bM$ with associated enhanced log Higgs module $(M,\theta_M,\phi_M)$ and representation $V_{\infty}(\bM)$ of $\Gamma$ over $\widehat R_{\infty}[\frac{1}{p}]$, the arithmetic Sen operator of $V_{\infty}(\bM)$ is $-\frac{\phi_M}{\pi E'(\pi)}$.
  \end{conj}
  
  The above conjecture together with Conjecture \ref{Conj-loc-Sen-cohomology} implies the following quasi-isomorphisms
  \begin{equation*}
      \begin{split}
          R\Gamma_{\Prism}(\bM)\otimes_RR_{C}[\frac{1}{p}]&\simeq \HIG(M,\theta_M,\phi_M)\otimes_RR_{C}[\frac{1}{p}]\\
          &\simeq \HIG(M\otimes_RR_{C}[\frac{1}{p}],(\zeta_p-1)\lambda\theta_M,-\frac{1}{\pi E'(\pi)}\phi_M)\\
          &\simeq R\Gamma(\Gamma,V_{\infty}(\bM))\otimes_RR_{C},
      \end{split}
  \end{equation*}
  where the first two quasi-isomorphisms follow from Theorem \ref{Thm-local Prismatic cohomology} and Proposition \ref{Prop-local comparison} (1). In particular, after taking $G_K$-equivariants, we get 
  \[R\Gamma_{\Prism}(\bM)\simeq\HIG(M,\theta_M,\phi_M)\simeq R\Gamma(\Gamma,V_{\infty}(\bM))\]
  which is exactly what we guess in Conjecture \ref{Conj-loc-relative Sen}. We state the result as a lemma.
  \begin{lem}
   Conjecture \ref{Conj-loc-relative Sen} is a consequence of Conjecture \ref{Conj-loc-Sen-cohomology} and Conjecture \ref{Conj-local Sen comparison}.
  \end{lem}
 
  \section{Globalizations}
 \subsection{Definitions and Preliminaries}
 In this section, we globalise local constructions in the previous section. From now on, we always assume $\frakX$ is a semi-stable $p$-adic formal scheme over $\calO_K$ of relative dimension $d$ with rigid analytic generic fibre $X$. For any $p$-adic complete subfield $L$ of $C$, we denote $\frakX_L$ and $X_L$ the base-changes of $\frakX$ and $X$ along natural morphisms $\Spf(\calO_L)\to\Spf(\calO_K)$ and $\Spa(L,\calO_L)\to\Spa(K,\calO_K)$, respectively.

  \begin{dfn}\label{Dfn-enhanced Higgs module-G}
   By an {\bf enhanced log Higgs bundle} on $\frakX_{\et}$ with coefficients in $\calO_{\frakX}$, we mean a triple $(\calM,\theta_{\calM},\phi_{\calM})$ satisfying the following properties:
   
       $(1)$ $\calM$ is a finite locally free $\calO_{\frakX}$-module and \[\theta_{\calM}:\calM\to\calM\otimes_{\calO_{\frakX}}\widehat \Omega^1_{\frakX,\log}\{-1\}\]
       defines a nilpotent Higgs field on $\calM$, i.e. $\theta_M$ is a section of $\underline{\End}(\calM)\otimes_{\calO_{\frakX}}\widehat \Omega^1_{\frakX,\log}\{-1\}$ which is nilpotent and satisfies $\theta_{\calM}\wedge\theta_{\calM}=0$. Here ``$\bullet\{-1\}$'' denotes the Breuil--Kisin twist of $\bullet$. Denote by $\HIG(\calM,\theta_{\calM})$ the induced Higgs complex.
       
       $(2)$ $\phi_M\in\End(\calM)$ is ``topologically nilpotent'' in the following sense:
       \[\lim_{n\to+\infty}\prod_{i=0}^n(\phi_M+i\pi E'(\pi))=0\]
       and induces an endomorphism of $\HIG(\calM,\theta_M)$; that is, the following diagram
       \begin{equation*}
           \xymatrix@C=0.45cm{
           \calM \ar[rr]^{\theta_{\calM}\quad}\ar[d]^{\phi_{\calM}}&&\calM\otimes\widehat \Omega^1_{\frakX,\log}\{-1\} \ar[rr]^{\quad\theta_{\calM}}\ar[d]^{\phi_{\calM}+\pi E'(\pi)\id}&&\cdots\ar[rr]^{\theta_{\calM}\quad}&&\calM\otimes\widehat \Omega^d_{\frakX,\log}\{-d\}\ar[d]^{\phi_{\calM}+d\pi E'(\pi)\id}\\
           \calM \ar[rr]^{\theta_{\calM}\quad}&&\calM\otimes\widehat \Omega^1_{\frakX,\log}\{-1\} \ar[rr]^{\quad\theta_{\calM}}&&\cdots\ar[rr]^{\theta_{\calM}\quad}&&\calM\otimes\widehat \Omega^d_{\frakX,\log}\{-d\}
           }
       \end{equation*}
       commutes. Denote $\HIG(\calM,\theta_{\calM},\phi_{\calM})$ the total complex of this bicomplex.

   Denote by $\HIG^{\log}_*(\frakX,\calO_{\frakX})$ the category of enhanced log Higgs bundles over $\frakX$. Similarly, we define enhanced log Higgs bundles on $\frakX_{\et}$ with coefficients in $\calO_{\frakX}[\frac{1}{p}]$ and denote the corresponding category by $\HIG^{\log}_*(\frakX,\calO_{\frakX}[\frac{1}{p}]).$
 \end{dfn}
 \begin{dfn}[\emph{\cite[Definition 3.4]{MW-c}}]\label{Dfn-Higgs module with Galois-G}
   By an {\bf $G_K$-Higgs bundle} on $X_{C,\et}$, we mean a pair $(\calH,\theta_{\calH})$ such that 
   
        $(1)$ $\calH$ is a locally finite free $\calO_{X_{C}}$-module and 
       \[\theta_{\calH}:\calH\to\calH\otimes_{\calO_X}\Omega^1_{X}(-1)\]
       defines a nilpotent Higgs field on $\calH$, where $\bullet(-1)$ denotes the Tate twist of $\bullet$. Denote by $\HIG(\calH,\theta_{\calH})$ the induced Higgs complex.
       
       $(2)$ $\calH$ is equipped with a semi-linear continuous $G_K$-action such that $\theta_{\calH}$ is $G_K$-equivariant.
  
   Denote by $\HIG^{\nil}_{G_K}(X_{C})$ the category of $G_K$-Higgs bundles on $X_{C,\et}$. Similarly, one can define $\Gamma_K$-Higgs bundles on $X_{\hat K_{\cyc},\et}$ and denote the corresponding category by $\HIG^{\nil}_{\Gamma_K}(X_{\hat K_{\cyc}})$.
 \end{dfn}
 \begin{rmk}\label{Rmk-Higgs with Galois}
   By Faltings' almost purity theorem, there is a natural equivalence of categories
   \[\HIG^{\nil}_{\Gamma_K}(X_{\hat K_{\cyc}})\to\HIG^{\nil}_{G_K}(X_{C})\]
   induced by base-change. We shall use this equivalence freely and identify these two categories in the rest of this paper.
 \end{rmk}

 Recall the following definition appearing essentially in \cite[Section 3]{Pet}.
 \begin{dfn}[\emph{\cite[Definition 3.6]{MW-c}}]\label{Dfn-arithmetic Higgs module-G}
   Let $X_{K_{cyc}}$ be the ringed space whose underlying space coincides with $X$'s and the structure sheaf is $\calO_{X_{K_{\cyc}}}:=\calO_X\otimes_KK_{\cyc}$ (note that this notation does not contradict our convention at the beginning of this section as $K_{\cyc}$ is not $p$-complete).
   
   By an {\bf arithmetic Higgs bundle} on $X_{K_{\cyc},\et}$, we mean a triple $(\calH,\theta_{\calH},\phi_{\calH})$ such that 

      $(1)$ $\calH$ is a finite locally free $\calO_{X_{K_{\cyc}}}$-module and 
      \[\theta_{\calH}:\calH\to\calH\otimes_{\calO_{X}}\Omega^1_{X}(-1)\]
       defines a nilpotent Higgs field on $\calH$. Denote by $\HIG(\calH,\theta_{\calH})$ the induced Higgs complex.
       
      $(2)$ $\phi_{\calH}\in\End(\calH)$ induces an endomorphism of $\HIG(\calH,\theta_{\calH})$; that is, it makes the following diagram
       \begin{equation*}
           \xymatrix@C=0.45cm{
           \calH \ar[rr]^{\theta_{\calH}\quad}\ar[d]^{\phi_{\calH}}&&\calH\otimes_{\calO_X} \Omega^1_{X}(-1) \ar[rr]^{\quad\theta_{\calH}}\ar[d]^{\phi_{\calH}-\id}&&\cdots\ar[rr]^{\theta_{\calH}\quad}&&\calH\otimes_{\calO_X} \Omega^d_{X}(-d)\ar[d]^{\phi_{\calH}-d\cdot\id}\\
           \calH \ar[rr]^{\theta_{\calH}\quad}&&\calH\otimes_{\calO_X} \Omega^1_{X}(-1) \ar[rr]^{\quad\theta_{\calH}}&&\cdots\ar[rr]^{\theta_{\calH}\quad}&&\calH\otimes_{\calO_X} \Omega^d_{X}(-d)
           }
       \end{equation*}
       commute. Denote $\HIG(\calH,\theta_{\calH},\phi_{\calH})$ the total complex of this bicomplex.

   Denote $\HIG^{\rm arith}(X_{K_{\cyc}})$ the category of arithmetic Higgs bundles on $X_{K_{\cyc}}$.
 \end{dfn}

 Now we have the following ({\bf NOT necessarily commutative}) diagram  which is the global version of functors in Subsection \ref{local functors}
 
 \begin{equation}\label{Diag-relation among Higgss}
     \xymatrix@C=0.5cm{
     \HIG^{\log}_*(\frakX,\calO_{\frakX}[\frac{1}{p}])\ar[rrd]_{F_3}\ar[rr]^{F_1}&&\HIG^{\nil}_{G_K}(X_{C})\ar[d]^{F_2}\\
     &&\HIG^{\rm arith}(X_{K_{\cyc}}).
     }
 \end{equation}
 We specify the meaning of functors $F_i$'s as follows:
 
 The functor $F_3$ sends a Hodge--Tate crystal $(\calM,\theta_{\calM},\phi_M)$ in $\HIG^{\log}_*(\frakX,\calO_{\frakX}[\frac{1}{p}])$ to the arithmetic Higgs bundle
 \begin{equation}\label{F3}
     (\calH=\calM\otimes_{\calO_{\frakX}}\calO_{X_{K_{\cyc}}},\theta_{\calH}=\theta_{\calM}\otimes\id,\phi_{\calH} = -\frac{\phi_{\calM}}{\pi E'(\pi)}\otimes\id).
 \end{equation}
 
 The definition of $F_1$ is suggested by the local case (c.f. Proposition \ref{Prop-local comparison} (2)): It sends an enhanced log Higgs bundle $(\calM,\theta_{\calM},\phi_M)$ in $\HIG^{\log}_*(\frakX,\calO_{\frakX}[\frac{1}{p}])$ to the $G_K$-Higgs bundle 
 \begin{equation}\label{F1}
     (\calH=\calM\otimes_{\calO_{\frakX}}\calO_{X_{C}},\theta_{\calH}=\theta_{\calM}\otimes\lambda(\zeta_p-1)\id)
 \end{equation}
 with the $G_K$-action on $\calH$ being induced by the formulae
 \begin{equation}\label{F1-Galois action}
     g\mapsto (1-c(g)\lambda(1-\zeta_p)\pi E'(\pi))^{-\frac{\phi_M}{\pi E'(\pi)}}.
 \end{equation}
 \begin{lem}\label{Lem-F1 is ff}
   The $F_1$ given above is a well-defined fully faithful functor.
 \end{lem}
 \begin{proof}
   The proof is the same as that of \cite[Lemma 3.8]{MW-c}, which relies on Proposition \ref{Prop-local comparison}.
 \end{proof}
 
 The functor $F_2$ is already defined in \cite{MW-c} (in the paragraph below \cite[Remark 3.10]{MW-c}), which again appears essentially in \cite[Proposition 3.2]{Pet}. 
 
\begin{thm}\label{Thm-enhanced to Galois to arith}
The functors $F_2$ and $F_3$ are faithful and $F_1$ is fully faithful.
 \end{thm}
 \begin{proof}
   This reduces to Theorem \ref{Thm-fully faithful-loc} and Proposition \ref{Prop-local comparison}.
 \end{proof}
 
 Now, there are two faithful functors 
 \[F_2\circ F_1,F_3:\HIG^{\log}_*(\frakX,\calO_{\frakX}[\frac{1}{p}])\to\HIG^{\rm arith}(X_{K_{\cyc}}).\] 
 We have the following conjecture which is true in the case of $\Spf(\calO_K)$:
 \begin{conj}\label{Conj-F3=F2F1}
    Keep notations as above. Then $F_3\simeq F_2\circ F_1$.
 \end{conj}

 \subsection{The inverse Simpson functor for enhanced Higgs bundles}
 In this subsection, we construct an inverse Simpson functor from the category $\HIG^{\log}_*(\frakX,\calO_{\frakX}[\frac{1}{p}])$ to $\Vect(X,\OX)$ by using prismatic methods. The first main result is the following theorem, which establishes a bridge between Hodge--Tate crystals and generalised representations.
 \begin{thm}\label{Thm-perfect HT is generalised repn}
   There is a canonical equivalence between the category  $\Vect((\frakX)^{\perf}_{\Prism,\log},\overline \calO_{\Prism}[\frac{1}{p}])$ of Hodge--Tate crystals on $(\frakX)_{\Prism,\log}^{\perf}$ and the category $\Vect(X,\OX)$ of generalised representations on $X_{\proet}$.
 \end{thm}
 \begin{proof}
   This follows from Lemma \ref{perfect site crystal} and \cite[Theorem 3.17]{MW-c}. \end{proof}
  \begin{rmk}
    By Remark \ref{rigidity} and \cite[Proposition 2.7]{BS-b}, the presheaf $\calO_{\Prism}[\frac{1}{\calI_{\Prism}}]^{\wedge}_p$ on $(\frakX,\calM_{\frakX})_{\Prism}$ sending $(A,I,M,\delta_{\log})$ to $A[\frac{1}{I}]^{\wedge}_p$
    is indeed a sheaf. Then one may define the category $\Vect^F((\frakX,\calM_{\frakX})_{\Prism},\calO_{\Prism}[\frac{1}{\calI_{\Prism}}]^{\wedge}_p)$ of Laurent $F$-crystals $(\bM,\varphi_{\bM})$ on $(\frakX,\calM_{\frakX})_{\Prism}$ by requiring that
    $\bM$ is a sheaf of $\calO_{\Prism}[\frac{1}{\calI_{\Prism}}]^{\wedge}_p$ such that the following conditions hold:
    
    $(1)$ For any $\frakA = (A,I,M,\delta_{\log})\in (\frakX,\calM_{\frakX})_{\Prism}$, $\bM(\frakA)$ is a finite projective $A[\frac{1}{I}]^{\wedge}_p$-module together with an $A[\frac{1}{I}]^{\wedge}_p$-linear isomorphism \[\varphi_{\bM}(\frakA):\bM(\frakA)\otimes_{A[\frac{1}{I}]^{\wedge}_p,\varphi}A[\frac{1}{I}]^{\wedge}_p\to\bM(\frakA).\]
          
    $(2)$ For any morphism $\frakA = (A,I,M,\delta_{\log})\to \frakB = (B,IB,N,\delta_{\log})$ in $(\frakX,\calM_{\frakX})_{\Prism}$, there is a canonical isomorphism
    \[\bM(\frakA)\otimes_{A[\frac{1}{I}]^{\wedge}_p}B[\frac{1}{I}]^{\wedge}_p\to \bM(\frakB),\]
    which is compatible with $\varphi_{\bM}$.
    
    Now, Lemma \ref{perfect site crystal} implies that there is a canonical equivalence of categories
    \[\Vect^{F}((\frakX)_{\Prism}^{\perf},\calO_{\Prism}[\frac{1}{\calI_{\Prism}}]^{\wedge}_p)\to \Vect^F((\frakX,\calM_{\frakX})^{\perf}_{\Prism},\calO_{\Prism}[\frac{1}{\calI_{\Prism}}]^{\wedge}_p),\]
    where both categories can be defined in the obvious way. Then \cite[Corollary 3.7]{BS-b} provides natural equivalences
    \[\Vect^{F}((\frakX)_{\Prism},\calO_{\Prism}[\frac{1}{\calI_{\Prism}}]^{\wedge}_p)\to\Vect^{F}((\frakX)_{\Prism}^{\perf},\calO_{\Prism}[\frac{1}{\calI_{\Prism}}]^{\wedge}_p)\to \rL\rS_{\Zp}(X),\]
    where $\rL\rS_{\Zp}(X)$ denotes the category of \'etale $\Zp$-local systems on $X$, the generic fibre of $\frakX$. In particular, we get an equivalence
    \[\Vect^F((\frakX,\calM_{\frakX})^{\perf}_{\Prism},\calO_{\Prism}[\frac{1}{\calI_{\Prism}}]^{\wedge}_p)\to\rL\rS_{\Zp}(X).\]
    
    Without checking details, we still believe that arguments for the proof of \cite[Theorem 3.1]{MW-a} give an equivalence\footnote{The authors know from Heng Du that he has established this equivalence.} 
    \[\Vect^F((\frakX,\calM_{\frakX})_{\Prism},\calO_{\Prism}[\frac{1}{\calI_{\Prism}}]^{\wedge}_p)\to\Vect^F((\frakX,\calM_{\frakX})^{\perf}_{\Prism},\calO_{\Prism}[\frac{1}{\calI_{\Prism}}]^{\wedge}_p).\] 
    In particular, this suggests a commutative diagram of equivalent categories
    \[
    \xymatrix@C=0.45cm{
    \Vect^{F}((\frakX)_{\Prism},\calO_{\Prism}[\frac{1}{\calI_{\Prism}}]^{\wedge}_p)\ar[dr]\ar[dd]\ar[rr]&&\Vect^{F}((\frakX)_{\Prism}^{\perf},\calO_{\Prism}[\frac{1}{\calI_{\Prism}}]^{\wedge}_p)\ar[dd]\ar[dl]\\
    & \rL\rS_{\Zp}(X)&\\
    \Vect^{F}((\frakX,\calM_{\frakX})_{\Prism},\calO_{\Prism}[\frac{1}{\calI_{\Prism}}]^{\wedge}_p)\ar[ur]\ar[rr]&&\Vect^{F}((\frakX,\calM_{\frakX})_{\Prism}^{\perf},\calO_{\Prism}[\frac{1}{\calI_{\Prism}}]^{\wedge}_p).\ar[ul]
    }
    \]
    We also believe that the approach to showing \cite[Theorem 4.11]{MW-a} also works in the logarithmic case; that is, for a Laurent $F$-crystal $\bM$ on $(\frakX,\calM_{\frakX})_{\Prism}$ with associated $\Zp$-local system $\calL$ on $X_{\et}$, there is a quasi-isomorphism
    \[\RGamma((\frakX,\calM_{\frakX})_{\Prism},\bM)^{\varphi=1}\simeq\RGamma(X_{\et},\calL).\]
    We will not deal with these problems in this paper.
  \end{rmk}
 
 Now we can get the following important theorem as in the case of good reduction, whose proof is the same as that of \cite[Theorem 3.19]{MW-c} after we have established all the needed results in the previous sections.
 
\begin{thm}\label{Thm-HT crystal as Higgs bundles-G}
   Let $\frakX$ be a semistable $p$-adic formal scheme over $\calO_K$. Then there is a canonical equivalence of the categories
   \[M:\Vect((\frakX)_{\Prism,\log},\overline \calO_{\Prism}[\frac{1}{p}])\to\HIG^{\log}_*(\frakX,\overline \calO_{\frakX}[\frac{1}{p}]),\]
   which makes the following diagram commute
   \begin{equation}\label{Diag-Main}
       \xymatrix@C=0.5cm{
         \Vect((\frakX)_{\Prism,\log},\overline \calO_{\Prism}[\frac{1}{p}])\ar[r]^R\ar[d]^M&\Vect((\frakX)_{\Prism,\log}^{\perf},\overline \calO_{\Prism}[\frac{1}{p}])\ar[r]^{\quad\cong}&\Vect(X,\OX)\ar[d]^{H}\\
         \HIG^{\log}_*(\frakX, \calO_{\frakX}[\frac{1}{p}])\ar[rr]^{F_1}&&\HIG^{\nil}_{G_K}(X_{C}).
       }
   \end{equation}
 \end{thm}
 
 \begin{rmk}
   The functor $H$ is already constructed in \cite[Theorem 3.13]{MW-c}, which is only concerned about the generic fiber. But in order to be compatible with all our local calculations, we can not use the toric charts in the good reduction case. The local results we needed are provided by Theorem \ref{Thm-local Simpson}. 
 \end{rmk}
 
 \begin{rmk}
   At first glance, it seems that we can get the inverse Simpson functor without passing through the prismatic world. Namely, we can compose the functors $F_1$ and $H^{-1}$. But we want to emphasize that the local version of the functor $F_1$ can not be detected without diving into the categories of Hodge--Tate crystals, which reflect the shape of the $G_K$-actions.
 \end{rmk}
 As a corollary of Thereom \ref{Thm-HT crystal as Higgs bundles-G}, the full faithfulness of $R$ is straightforward, as all other arrows in diagram \ref{Diag-Main} are fully faithful.
 \begin{cor}\label{Cor-R is ff}
   The restriction $\Vect((\frakX)_{\Prism,\log},\overline \calO_{\Prism}[\frac{1}{p}])\xrightarrow{R}\Vect((\frakX)_{\Prism,\log}^{\perf},\overline \calO_{\Prism}[\frac{1}{p}])$ is fully faithful.
 \end{cor}
 
 Since full faithfulness is a local property, by Corollary \ref{Cor-fully f-rel-local}, when $\frakX$ is smooth (so the log structure $\calM_{\frakX}$ is induced by the composition $\bN\xrightarrow{1\mapsto \pi}\calO_K\to\calO_{\frakX}$), we have
 \begin{cor}
   Assume $\frakX$ is smooth. There is a fully faithful functor
   \[\Vect((\frakX)_{\Prism},\overline \calO_{\Prism}[\frac{1}{p}]) \to \Vect((\frakX,\calM_{\frakX})_{\Prism},\overline \calO_{\Prism}[\frac{1}{p}])\]
   which is induced by restriction and fits into the following commutative diagram:
   \[
   \xymatrix@C=0.45cm{
     \Vect((\frakX)_{\Prism},\overline \calO_{\Prism}[\frac{1}{p}]) \ar[d]\ar[r]& \Vect((\frakX,\calM_{\frakX})_{\Prism},\overline \calO_{\Prism}[\frac{1}{p}])\ar[d]\\
     \HIG^{\nil}_*(\frakX,\calO_{\frakX}[\frac{1}{p}])\ar[r]&\HIG^{\log}_*(\frakX,\calO_{\frakX}[\frac{1}{p}]),
   }
   \]
   where the bottom functor sends an enhanced Higgs bundle $(\calH,\theta_{\calH},\phi_{\calH})$ to the enhanced log Higgs bundle $(\calH,\theta_{\calH},\pi\phi_{\calH})$.
 \end{cor}
 \begin{rmk}\label{Rmk-Big comm diagram}
   By checking definitions of $F_1$ in both good and semi-stable reduction cases, we can easily see that we indeed have the following  commutative diagram:
   \[
     \xymatrix@C=0.45cm{
     \Vect((\frakX)_{\Prism},\overline \calO_{\Prism}[\frac{1}{p}]) \ar[r]^{\simeq\quad}& \Vect((\frakX,\calM_{\frakX})_{\Prism},\overline \calO_{\Prism}[\frac{1}{p}])\ar[r]^{\qquad\simeq}&\Vect(X,\widehat \calO_X)\ar[dd]^{\simeq}\\
     \Vect((\frakX)_{\Prism},\overline \calO_{\Prism}[\frac{1}{p}]) \ar[d]^{\simeq}\ar[u]\ar[r]& \Vect((\frakX,\calM_{\frakX})_{\Prism},\overline \calO_{\Prism}[\frac{1}{p}])\ar[d]^{\simeq}\ar[u]\\
     \HIG^{\nil}_*(\frakX,\calO_{\frakX}[\frac{1}{p}])\ar[r]&\HIG^{\log}_*(\frakX,\calO_{\frakX}[\frac{1}{p}])\ar[r]^{F_1}&\HIG^{\nil}_{G_K}(X_C).
   }
   \]
   Each arrow above is actually a full faithful functor and is an equivalence if it is labelled by ``$\simeq$'' .
 \end{rmk}
 
 Now, there are two ways to assign to a Hodge--Tate crystal an arithmetic Higgs bundles:
 \[F_3\circ M, F_2\circ H\circ R=F_2\circ F_1\circ M: \Vect((\frakX)_{\Prism,\log},\overline \calO_{\Prism}[\frac{1}{p}]) \to \HIG^{\rm arith}(X_{C}).\]
 Then the Conjecture \ref{Conj-F3=F2F1} can be restated as follows:
 \begin{conj}\label{Conj-HT vs Sen}
   For any rational Hodge-Tate crystal $\bM\in\Vect((\frakX)_{\Prism,\log},\overline \calO_{\Prism}[\frac{1}{p}])$, denote by $(\calM,\theta_{\calM},\phi_{\calM})$ and $V(\bM)$ the corresponding enhanced log Higgs bundle and generalised representation. Then the arithmetic Sen operator of $V(\bM)$ is $-\frac{\phi_{\calM}}{\pi E'(\pi)}$.
 \end{conj}
 \begin{rmk}
   The conjecture is a global version of Conjecture \ref{Conj-local Sen comparison} and can be checked \'etale locally on $\frakX$. The authors knew from Hui Gao that he could confirm Conjecture \ref{Conj-HT vs Sen} by using a relative version of Kummer Sen theory appearing in \cite{Gao}. So we will not try to attack the conjecture in this paper.
 \end{rmk}
 Finally, we make the following conjecture on the global prismatic cohomology:
 \begin{conj}\label{Conj-global prism coho}
   Keep notations as above. Then there exist quasi-isomorphisms
   \[R\Gamma_{\Prism}(\bM)\simeq R\Gamma(\frakX_{\et},\HIG(\calM,\theta_{\calM},\phi_{\calM}))\simeq R\Gamma(X_{\proet},V(\bM)).\]
 \end{conj}


\begin{thebibliography}{99}
 \bibitem[ALB19]{ALB} J. Ansch\"{u}tz, A.-C. Le Bras: {\it Prismatic Dieudonn\'e theory}, arxiv:1907.10525v3 (2019).
 
 \bibitem[Bha]{Bha} B. Bhatt: {\it Lecture 4: Perfect prism and perfectoid rings}, \href{http://www-personal.umich.edu/~bhattb/teaching/prismatic-columbia/lecture4-perfect-prisms-and-perfectoids.pdf}{http://www-personal.umich.edu/bhattb/teaching/prismatic-columbia/lecture4-perfect-prisms-and-perfectoids.pdf}.
 
 \bibitem[BJ]{BJ} B. Bhatt, A. J. de Jong: {\it Crystalline cohomology and de Rham cohomology},arxiv:1110.5001, (2011).
  
 \bibitem[BL22a]{BL-a} B. Bhatt, J. Lurie: {\it Absolute prismatic cohomology}, arXiv:2201.06120, (2022).
 
 \bibitem[BL22b]{BL-b} Bhatt, J. Lurie: {\it The prismatization of $p$-adic formal schemes}, arXiv:2201.06124, (2022).
 
 \bibitem[BMS18]{BMS-a} B. Bhatt, M. Morrow and P. Scholze: {\it Integral $p$-adic Hodge theory}, Publications math\'ematiques de l'IH\'ES volume 128, pages 219–397 (2018).
 
 \bibitem[BMS19]{BMS-b} B. Bhatt, M. Morrow and P. Scholze: {\it Topological Hochschild homology and integral $p$
-adic Hodge theory}, Publications math\'ematiques de l'IH\'ES volume 129, pages 199–310 (2019).
 
 \bibitem[BS19]{BS-a} B. Bhatt, P. Scholze: {\it Prisms and prismatic cohomology}, arXiv:1905.08229v3, (2019).
 
 \bibitem[BS21]{BS-b} B. Bhatt, P. Scholze: {\it Prismatic F-crystals and crystalline Galois representations}, arxiv:2106.14735v1, (2021).

 

 \bibitem[DLMS22]{DLMS} H. Du, T. Liu, Y.-S. Moon, K. Shimizu: {\it Completed prismatic $F$-crystals and crystalline $\Zp$-local systems}, arxiv:2203.03444, (2022).
 
 \bibitem[DL21]{DL} H. Du, T. Liu: {\it A prismatic approach to $(\varphi,\hat G)$-modules and $F$-crystals}, arXiv:2107.12240v2, (2021).
 
 \bibitem[DLLZ18]{DLLZ} H. Diao, K.-W. Lan, R. Liu, X. Zhu: {\it Logarithmic Riemman-Hilbert correspondences for rigid varieties}, arXiv:1803.05786v2, (2018).
  
 \bibitem[Dri20]{Dri} V. Drinfeld: {\it Prismatization}, arxiv:2005.04746, (2020).
 

 \bibitem[Gao22]{Gao} H. Gao: {\it Hodge--Tate prismatic crystals and Sen theory}, arXiv:2201.10136, (2022).
 
 \bibitem[GR22]{GR} H. Guo, E. Reinecke {\it Prismatic $F$-crystals and crystalline local systems}, arxiv:2203.09490, (2022).

 
 
 \bibitem[Har03]{Har} Urs T. Hartl: {\it Semi-stable models for rigid analytic spaces}, Manuscripta math. 110, 365–380, (2003).
 
 \bibitem[Liu08]{Liu} T. Liu: {\it On lattices in semi-stable representations: a proof of a conjecture of Breuil}, Compositio Mathematica, Vol 144, Issue 1, pp. 61-88, (2008). 
  
 \bibitem[LZ17]{LZ}R. Liu, X. Zhu: {\it Rigidty and a Riemman-Hilbert correspondence for $p$-adic local systems}, Inv. Math. 207, pp. 291-343 (2017).

 \bibitem[Kos20]{Kos} T. Koshikawa: {Logarithmic prismatic cohomology I}, arXiv:2007.14037, (2020).
 
 \bibitem[MT20]{MT} M. Morrow, T. Tsuji: {\it Generalised representations as $q$-connections in integral $p$-adic Hodge theory}, arXiv:2010.04059v2, (2020).
			
 \bibitem[MW21a]{MW-a} Y. Min, Y. Wang: {\it Relative $(\varphi,\Gamma)$-modules and prismatic $F$-crystals},	arXiv:2110.06076, (2021).
 
 \bibitem[MW21b]{MW-b} Y. Min, Y. Wang: {\it On the Hodge--Tate crystals over $\calO_K$}, arXiv:2112.10140, (2021). 
 
 \bibitem[MW22]{MW-c} Y. Min, Y. Wang: {\it $P$-adic Simpson correspondence via prismatic crystals}, arXiv:2202.08030, (2022).
 
 
 \bibitem[Ols05]{Ols} M. C. Olsson: {\it The logarithmic cotangent complex}, Math. Ann. 333, 859–931, (2005). 
 
 \bibitem[Pet20]{Pet} A. Petrov {\it: Geometrically irreducible $p$-adic local systems are de Rham up to a twist} arxiv:2012.23372v3 (2020).
 
 \bibitem[Sen80]{Sen} S. Sen {\it: Continuous cohomology and $p$-adic Galois representations} Inventiones math. 62, 89-116, (1980).


 \bibitem[Sch13]{Sch-b}P. Scholze: {\it $p$-adic Hodge theory for rigid-analytic varieties},  Forum of Mathematics, Pi, 1, e1, 77 pages, (2013).
 
			
 \bibitem[Tian21]{Tian} Y. Tian: {\it Finiteness and duality for the cohomology of prismatic crystals}, arXiv:2109.00801v1, (2021).
			
 \bibitem[Tsu18]{Tsu} T. Tsuji: {\it Notes on the local p-adic Simpson correspondence}, Math. Ann. 371, (2018).
 
 \bibitem[Wang17]{Wang} X. Wang: {\it Weight elimination in two dimensions when $p = 2$}, arXiv:1711.09035, (2017).
 
 \bibitem[Wu21]{Wu} Z. Wu: {\it Galois representations, $(\varphi,\Gamma)$-modules and prismatic $F$-crystals}, arXiv:2104.12105v3, (2021).
\end{thebibliography}
\end{document}